\documentclass[11pt, twoside]{article}

\usepackage[english]{babel}

\usepackage{bbm}
\usepackage{amssymb}
\usepackage{amsfonts}
\usepackage{amsmath}
\usepackage{amsthm}
\usepackage{color}
\usepackage{mathrsfs}
\usepackage{txfonts}
\usepackage{bbm}
\usepackage{enumerate}
\usepackage{anysize}
\usepackage{indentfirst}

\usepackage{latexsym}

\usepackage[colorlinks=true,
linkcolor=blue,
citecolor=red,
urlcolor=magenta,
]{hyperref}

\usepackage[title,titletoc]{appendix}

\allowdisplaybreaks

\newcommand{\dd}{\,d}

\newtheorem{theorem}{Theorem}[section]
\newtheorem{lemma}[theorem]{Lemma}
\newtheorem{corollary}[theorem]{Corollary}
\newtheorem{proposition}[theorem]{Proposition}

\theoremstyle{definition}
\newtheorem{remark}[theorem]{Remark}

\newcounter{assum}

\renewcommand{\appendix}{\par
\setcounter{section}{0}%
\setcounter{subsection}{0}%
\setcounter{subsubsection}{0}%
\gdef\thesection{\@Alph\c@section}%
\gdef\thesubsection{\@Alph\c@section.\@arabic\c@subsection}%
\gdef\theHsection{\@Alph\c@section.}%
\gdef\theHsubsection{\@Alph\c@section.\@arabic\c@subsection}%
\csname appendixmore\endcsname
}

\numberwithin{equation}{section}

\def\XXint#1#2#3{{\setbox0=\hbox{$#1{#2#3}{\int}$ }
\vcenter{\hbox{$#2#3$ }}\kern-.6\wd0}}

\def\XXint#1#2#3{{\setbox0=\hbox{$#1{#2#3}{\int}$ }
\vcenter{\hbox{$#2#3$ }}\kern-.55\wd0}}

\begin{document}

\arraycolsep=1pt

\title{\bf\Large Stability for the Affine Sobolev Inequality\\ and its Critical Points for $p \ge 2$
\footnotetext{ \hspace{-0.35cm} 2020 \emph{Mathematics Subject Classification}. Primary 46E35;
Secondary 26D10.
\endgraf \emph{Key words and phrases}  Affine Sobolev inequality, gradient stability, quantitative inequalities.
\endgraf
R.L.F. acknowledges partial support through US National Science
Foundation grant DMS-1954995, as well as through the German Research Foundation through EXC-2111-390814868,
TRR 352–Project-ID 470903074, and
FR 2664/3-1.
Y.L. and D.Y. are supported by the National
Natural Science Foundation of China
(Grant Nos. 12431006 and 12371093),
the Beijing Natural Science Foundation (Grant No. 1262011),
and the Fundamental Research Funds
for the Central Universities (Grant No. 2253200028).}}
\author{Rupert L. Frank, Yinqin Li and Dachun Yang}
\date{\today}
\maketitle

\vspace{-0.6cm}

\begin{center}
\begin{minipage}{13cm}
{\small {\bf Abstract}\quad
We prove stability for the affine Sobolev inequality for exponents $p\geq 2$ with best possible norm and best possible stability exponent. We also show a corresponding result for critical points of the functional in the absence of bubbling. An important ingredient in our proof is the classification of positive energy solutions to the critical affine $p$-Laplace equation.
}
\end{minipage}
\end{center}


\tableofcontents


\section{Introduction and main results}

Sharp Sobolev inequalities are central objects in nonlinear partial differential equations and geometric analysis. As we will describe later in some detail, in the last two decades there has been a lot of research concerning the stability question for these sharp inequalities and their critical points.

The aim of the present paper is to study the stability problem for the
\emph{affine Sobolev inequality}. This is a strengthening of the
classical Sobolev inequality that concerns the \emph{affine Sobolev energy} $\mathcal{E}_p(f)$, defined for $f\in\dot W^{1,p}(\mathbb R^n)$, the homogeneous Sobolev space (whose definition we recall after \eqref{eq-Sob}), by
\begin{align*}
  \mathcal E_p(f)
  :=
  \left[
    \int_{\mathbb S^{n-1}}
    \|\nabla_\xi f\|_p^{-n}\,d\xi
  \right]^{-\frac1n}.
\end{align*}
Here $\nabla_\xi f = \xi\cdot\nabla f$ and $\|\cdot\|_p$ denotes the $L^p$-norm on $\mathbb R^n$. The affine Sobolev inequality states that for $p\in[1,n)$ there is a constant $S_{\rm aff}>0$, depending only on $n$ and $p$, such that, for all $f\in \dot W^{1,p}$,
\begin{align}\label{eq-affSob}
\mathcal{E}_p(f)\ge S_{\rm aff}\|f\|_{p^*}.
\end{align}
Here $p^*=np/(n-p)$.

An important feature of the affine Sobolev inequality is its large symmetry group and, in particular, its affine invariance, which is \emph{not} shared by the classical Sobolev inequality. More specifically, if
$$
\tilde f(x) := |\det B \,|^\frac{n-p}{np} \, f(Bx+b)
$$
for some $B\in \operatorname{GL}(n)$ and $b\in\mathbb R^n$, then
$$
\frac{\mathcal{E}_p(\tilde f)}{\|\tilde f\|_{p^*}} = \frac{\mathcal{E}_p(f)}{\|f\|_{p^*}}.
$$
In part due to this invariance, the affine Sobolev inequality is deeply related to the Brunn--Min\-kow\-ski theory in convex geometry; see the survey \cite{g02} as well as the monographs \cite{bfr26,s14}.

In what follows we are interested in the sharp form of the affine Sobolev inequality \eqref{eq-affSob} and we agree to let $S_{\rm aff}$ denote the best possible constant in this inequality. It was found by Zhang \cite{z99} for $p=1$ and
by Lutwak et al.\ \cite{lyz02} for $p\in(1,n)$, who also characterized the optimizers in this inequality, namely the set
\begin{equation}
    \label{eq:defmaff}
    \mathcal{M}_{\rm aff}:=\left\{ f\in \dot W^{1,p}\setminus\{{\bf 0}\} :\ \mathcal{E}_p(f)= S_{\rm aff}\|f\|_{p^*} \right\}.
\end{equation}

Our first main result quantifies the fact that, when $p\geq 2$, functions $f\in \dot W^{1,p}$ for which $\mathcal{E}_p(f)/\|f\|_{p^*}$ is close to $S_{\rm aff}$ are close to elements in $\mathcal M_{\rm aff}$.

\begin{theorem}\label{thm-stab}
If $p\in[2,n)$, then there is a positive constant $c_{n,p}$,
depending only on $n$ and $p$,
such that, for any $f\in \dot{W}^{1,p}$,
\begin{align}\label{eq-stab}
&{\mathcal{E}_p(f)}^p
-S_{\rm aff}^p\|f\|_{p^*}^p\notag\\
&\quad\ge c_{n,p}
\inf_{v\in \mathcal M_{\rm aff},\, T\in \operatorname{SL}(n)}\left\{\|\nabla(f\circ T -v)\|_{p}^p
+\int_{\mathbb{R}^n}
|\nabla v|^{p-2}|\nabla(f\circ T- v)|^2
\right\}.
\end{align}
\end{theorem}

\begin{remark}
\begin{enumerate}[{\rm(i)}]
  \item This improves upon a result in a concurrent and independent work \cite{flz26b}. Indeed, we obtain a stronger conclusion than \eqref{eq-stab} in our proof; see \eqref{eq-stabstrengthened} below. Importantly, in \cite{flz26b} only the first term on the right-hand side of \eqref{eq-stab} (or its analogue in \eqref{eq-stabstrengthened}) appears.
  \item The two terms on the right-hand side are, respectively, the $p$-th and the second power of a norm. The first term, involving the $p$-th power, is reminiscent of work by Figalli and Zhang on the classical Sobolev inequality \cite{fz22}. In Section \ref{Sec-sharp} we will prove that this $p$-th power of the distance is best possible, again in analogy with the classical Sobolev inequality. The second term, involving the square of a norm, is reminiscent of \cite{fp24} and the quadratic, rather than $p$-th order vanishing of the distance is crucial in some applications \cite{iz25}. In Section \ref{Sec-sharp} we will show that this power $2$ of the norm is best possible. In this sense our result is optimal.
  \item The infimum over $B\in \operatorname{SL}(n)$ on the right-hand side of \eqref{eq-stab} makes the right-hand side affine invariant. Thus, our stability inequality inherits the same symmetries as the underlying sharp affine Sobolev inequality.
  \item  We should stress that the thrust of Theorem \ref{thm-stab} is the case $p>2$, as the result for $p=2$ follows easily from well-known results; see \eqref{eq-p=2} below. The assumption $p\geq 2$ is for technical reasons and it is an open problem to remove this restriction.
\end{enumerate}
\end{remark}

Our second main result concerns stability for critical points of the functional $\mathcal E_p(f)/\|f\|_{p_*}$ in the absence of bubbling. The Euler--Lagrange equation of the affine Sobolev quotient involves the \emph{affine $p$-Laplacian} $\Delta_p^{\rm aff}$, defined by
\begin{align}\label{eq:defaffinelapl}
\Delta_{p}^{\rm aff}(f) :=
\mathcal{E}_p(f)^{n+p}\operatorname{div}
\left(\int_{\mathbb{S}^{n-1}}
\|\nabla_\xi f\|_p^{-n-p}|\nabla_\xi f|^{p-2}[\nabla_\xi f]\xi\,d\xi\right).
\end{align}
For $f\in \dot W^{1,p}$ we have $\Delta_p^{\rm aff}(f)\in\dot W^{-1,p'}$, where $p'$ is the H\"older dual exponent of $p$, satisfying $\frac1p+\frac1{p'}=1$. Note that $\Delta_p^{\rm aff}$ is a nonlocal, quasilinear differential operator. It arises naturally in the context of the affine Sobolev inequality, since
$$
\frac{d}{dt}\Big|_{t=0} \mathcal E_p(f+tg) = - \mathcal E_p(f)^{1-p} \int_{\mathbb R^n} g \Delta_{p}^{\rm aff}(f).
$$
The affine invariance of the affine Sobolev energy shows that, for every $B\in \operatorname{GL}(n)$,
\begin{align}\label{eq-T}
\Delta_p^{\rm aff}(f\circ B)=|\det B\,|^{\frac pn} \,
(\Delta_p^{\rm aff}f)\circ B.
\end{align}

Let us introduce
\begin{equation}
    \label{eq:defmaff1}
    \mathcal M_{\rm aff}^{(1)} := \left\{ f\in \mathcal M_{\rm aff} : \|f\|_{p_*}=1 \right\}.
\end{equation}
Then the Euler--Lagrange equation for optimizers together with the normalization proves that, for all $f\in\mathcal M_{\rm aff}^{(1)}$,
\begin{equation}
    \label{eq:el}
    -\Delta_{p}^{\rm aff}(f) = S_{\rm aff}^p|f|^{p^*-2}f.
\end{equation}

Our second main result will say that if $f$ `almost' satisfies \eqref{eq:el} and its energy $\mathcal{E}_p(f)$ is not too far from $S_{\rm aff}$, then it is close to $\mathcal M_{\rm aff}^{(1)}$ in a quantitative sense. The quantity that measures that \eqref{eq:el} is almost satisfied is the following \emph{affine Euler--Lagrange residual}
\begin{align}
\label{eq:elresidual}
\delta_{\rm aff}(f)
:=\sup_{T\in \operatorname{SL}(n)} \left\|\Delta_p^{\rm aff}\left(f\circ T\right) + S_{\rm aff}^p |f\circ T|^{p^*-2} (f\circ T) \right\|_{
\dot{W}^{-1,p'}}.
\end{align}

\begin{theorem}\label{thm-bubble}
Let $p\in[2,n)$. There is a positive constant
$c_{n,p}$, depending only on $n$ and $p$,
such that, if $f\in\dot{W}^{1,p}$ satisfies
\begin{align}\label{thm-bubble-e8}
\frac12 S_{\rm aff}^p\le \mathcal{E}_p(f)^p\le\frac32 S_{\rm aff}^p,
\end{align}
then
\begin{align}\label{thm-bubble-e1}
\delta_{\rm aff}(f)\ge c_{n,p}
\inf_{T\in \operatorname{SL}(n),v\in\mathcal{M}_{\rm aff}^{(1)}}
\left\{\|\nabla(f\circ T-v)\|_p^{p-1}
+\left[\int_{\mathbb{R}^n}|\nabla v|^{p-2}|\nabla(f\circ T-v)|^2\right]^{\frac{p-1}{p}}\right\}.
\end{align}
\end{theorem}

\begin{remark}
\begin{enumerate}[{\rm(i)}]
  \item We obtain a slightly stronger conclusion than \eqref{thm-bubble-e1}; see \eqref{thm-bubble-e1strengthened} below.
  \item The two terms on the right-hand side of \eqref{thm-bubble-e1} are, respectively, the $(p-1)$-st and the $(p-1)/(2p)$-th power of a norm. The first term, involving the power $p-1$, is reminiscent of work by Liu and Zhang on the classical Sobolev inequality \cite{lz25}. In Section \ref{Sec-sharp} we will prove that this $(p-1)$-st power of the norm is best possible, again in analogy with critical points of the classical Sobolev inequality. We will also show there that the power $(p-1)/(2p)$ of the second norm is best possible, at least in the strengthened form \eqref{thm-bubble-e1strengthened} below. As far as we know, such a weighted $L^2$-norm has not appeared before in stability questions for critical points.
  \item The infimum over $T\in \operatorname{SL}(n)$ on the right-hand side of \eqref{thm-bubble-e1} makes the right-hand side affine invariant, while the left-hand side is affine invariant by definition of the residual $\delta_{\rm aff}(f)$.
  \item The term $\delta_{\rm aff}$ is the natural affine invariant quantity to measure the affine Euler--Lagrange residual.
Indeed, in \eqref{eq-p} below we will recall that $\mathcal E_p(f)$ is equivalent to $\min_{T\in \operatorname{SL}(n)} \|\nabla(f\circ T)\|_p$, from which it follows that, with $\langle\cdot,\cdot\rangle$ denoting the $\dot W^{-1,p'}\times \dot W^{1,p}$ duality,
\begin{align*}
\delta_{\rm aff}(f)
& =\sup_{T\in\operatorname{SL}(n)}\sup_{\varphi\ne {\bf 0}}
\frac{|\langle \Delta_p^{\rm aff}\left(f\circ T\right) + S_{\rm aff}^p |f\circ T|^{p^*-2} (f\circ T),\varphi\rangle|}{\|\nabla \varphi\|_p}\\
&=\sup_{\psi\ne {\bf 0}}\sup_{T'\in\operatorname{SL}(n)}
\frac{|\langle \Delta_p^{\rm aff}\left(f\right) + S_{\rm aff}^p |f|^{p^*-2} f,\psi\rangle|}{\|\nabla(\psi\circ T')\|_p} \\
& =\sup_{\psi\ne {\bf 0}}\frac{|\langle \Delta_p^{\rm aff}\left(f\right) + S_{\rm aff}^p |f|^{p^*-2} f,\psi\rangle|}{\inf_{T'\in\operatorname{SL}(n)}\|\nabla(\psi\circ T')\|_p} \\
&
\sim \sup_{\psi\ne {\bf 0}}\frac{|\langle \Delta_p^{\rm aff}\left(f\right) + S_{\rm aff}^p |f|^{p^*-2} f,\psi\rangle|}{\mathcal{E}_p(\psi)}.
\end{align*}
Here we used \eqref{eq-T}. This shows that $\delta_{\rm aff}$
is the natural dual ``norm'' associated with the affine Sobolev energy.
  \item As in Theorem \ref{thm-stab}, the thrust of Theorem \ref{thm-bubble} is the case $p>2$, since the result for $p=2$ follows easily from a theorem of Ciraolo et al.\ \cite{cfm18} and \eqref{eq:affinelapl2} below. Again, the assumption $p\geq 2$ is for technical reasons and it is an open problem to remove this restriction.
  \item One \emph{important ingredient} in the proof of Theorem \ref{thm-bubble} is a classification result for solutions of the Euler--Lagrange equation for the affine Sobolev quotient. This is stated and proved in Section~\ref{sec:classification}.
\end{enumerate}
\end{remark}


\subsection{Background}

\subsubsection*{The affine vs.~the classical Sobolev inequality}

It is natural to compare the affine Sobolev inequality with its classical variant, which states that for $p\in[1,n)$ there is a constant $S_{\rm Sob}$, depending only on $n$ and $p$, such that,
for all $f\in \dot W^{1,p}$,
\begin{align}\label{eq-Sob}
  \|\nabla f\|_p
  \ge
  S_{\rm Sob}\|f\|_{p^*}.
\end{align}
We recall that the homogeneous Sobolev space $\dot W^{1,p}$ consists of all weakly differentiable functions $f$ on $\mathbb R^n$ whose weak gradient belows to $L^p$ and which vanish at infinity in the sense that $|\{ |f|>\lambda\}|<\infty$ for any $\lambda>0$.

We agree to let $S_{\rm Sob}$ denote the \emph{sharp constant} in inequality \eqref{eq-Sob}. It was found by Rodemich, Aubin and Talenti \cite{r66,a76,t76}, who also characterized the optimizers in this inequality; that is, they determined the set
\begin{equation}
    \label{eq:defmsob}
    \mathcal{M}_{\rm Sob}:=\left\{ f\in \dot W^{1,p}\setminus\{{\bf 0}\} :\ \|\nabla f\|_p
    = S_{\rm Sob}\|f\|_{p^*} \right\}.
\end{equation}

An important feature of the sharp affine Sobolev inequality is that it strengthens the sharp classical Sobolev inequality. Indeed, by H\"older's inequality we find
\begin{align}\label{eq-holder}
  \mathcal E_p(f)\le
    |\mathbb{S}^{n-1}|^{-\frac1n-\frac1p}
  \left[
    \int_{\mathbb S^{n-1}}
    \|\nabla_\xi f\|_p^p\,d\xi
  \right]^{\frac1p}
  =
  |\mathbb{S}^{n-1}|^{-\frac1n-\frac1p}
\alpha_{n,p}
  \|\nabla f\|_p
\end{align}
with
\begin{align}\label{eq-alpha}
\alpha_{n,q}:=\left[\int_{\mathbb{S}^{n-1}}|e\cdot\xi|^q\,d\xi\right]^{\frac1q}
= \left[ \frac{2\pi^{\frac{n-1}{2}}\Gamma(\frac{q+1}{2})}{\Gamma(\frac{q+n}{2})} \right]^\frac1q,
\quad\forall\,e\in\mathbb{S}^{n-1}.
\end{align}
By comparing the explicit values of the constants, one finds that
\begin{equation}
    \label{eq-relaSob}
    S_{\rm Sob} = \left( |\mathbb{S}^{n-1}|^{-\frac1n-\frac1p}
\alpha_{n,p} \right)^{-1} S_{\rm aff},
\end{equation}
so the sharp affine Sobolev inequality implies the sharp classical Sobolev inequality.

We also note that
\begin{align}
    \label{eq:energyradial}
    \mathcal E_p(f)
  =
  |\mathbb{S}^{n-1}|^{-\frac1n-\frac1p} \,
\alpha_{n,p} \,
  \|\nabla f\|_p
  \qquad\text{for all radial}\ f\in \dot W^{1,p} \,,
\end{align}
since in this case $\|\nabla_\xi f\|_p$ is independent of $\xi\in\mathbb S^{n-1}$.

The affine Sobolev inequality is mostly of interest when $p\neq 2$, because for $p=2$ one finds, for any $f\in\dot{W}^{1,2}$,
\begin{align}\label{eq-p=2}
  \mathcal E_2(f)
  =
   |\mathbb{S}^{n-1}|^{-\frac1n-\frac12}
\alpha_{n,2}
  \min_{T\in \operatorname{SL}(n)}
  \|\nabla(f\circ T)\|_2;
\end{align}
see, for instance, \cite[Lemma 2.1 and Corollary 2.2]{st18}. The identity \eqref{eq-p=2} is special to the Hilbertian case and
does not hold for general $p$. For general $p\in(1,n)$, one only has the equivalence
\begin{align}\label{eq-p}
\mathcal E_p(f) \sim \min_{T\in \operatorname{SL}(n)} \|\nabla(f\circ T)\|_p,
\end{align}
where the positive equivalence constants depend only on $n$ and $p$; see \cite[Theorem 1.2]{hl16}. Moreover, these equivalence constants are not the constant appearing in \eqref{eq-relaSob}.
Thus, the above linear reduction is not available
in the non-Hilbertian case.

In passing we mention some other developments of affine Sobolev inequalities,
for example, the affine Sobolev--Zhang inequality on $BV(\mathbb R^n)$ \cite{w12}, the asymmetric affine Sobolev inequalities
\cite{hs09}, the affine Moser--Trudinger and Morrey--Sobolev inequalities \cite{clyz09} and fractional affine inequalities
\cite{DLTYY25,hl24,hl24-2}.

There is also a natural spectral and PDE side of the affine theory. Haddad et al.\
\cite{hjm21} developed affine Poincar\'e inequalities and affine spectral
inequalities. They studied in detail the affine $p$-Laplacian \eqref{eq:defaffinelapl} associated with the affine $p$-energy. Similar to \eqref{eq:energyradial}, we have
\begin{align*}
    \Delta_p^{\rm aff}(f) = |\mathbb S^{n-1}|^{-1-\frac pn}\, \alpha_{n,p}^p \, \Delta_p(f)
    \qquad\text{for all radial}\ f\in\dot W^{1,p},
\end{align*}
where $\Delta_p(f) = \operatorname{div}(|\nabla f|^{p-2}\nabla f)$ denotes the classical $p$-Laplacian of $f$.

In the Hilbertian case $p=2$, which was studied in \cite{st18}, one sees that
$$
\Delta_2^{\rm aff}(f) = \operatorname{div} (M_f \nabla f)
$$
with the positive definite matrix
$$
M_f := \mathcal{E}_2(f)^{n+2} \ \int_{\mathbb{S}^{n-1}}
\|\nabla_\xi f\|_2^{-n-2} \xi \otimes \xi \,d\xi .
$$
It follows easily \cite[Eq.~(2.14)]{st18} that for every $f$ there is a $T\in\operatorname{SL}(n)$ such that
\begin{align*}
    \Delta_2^{\rm aff}(f) \circ T = |\mathbb S^{n-1}|^{-1-\frac 2n}\, \alpha_{n,2}^2 \, \Delta(f\circ T) \,.
\end{align*}
It follows that
\begin{align}
    \label{eq:affinelapl2}
    \delta_{\rm aff}(f) \geq |\mathbb S^{n-1}|^{-1-\frac 2n}\, \alpha_{n,2}^2 \, \inf_{T\in\operatorname{SL}(n)} \left\|\Delta(f\circ T) + S_{\rm Sob}^2\, |f\circ T|^{2^*-2} (f\circ T) \right\|_{\dot W^{-1,2}}
\end{align}
when $p=2$.
We are not aware of a corresponding result for $p\neq 2$, even with a worse constant.


\subsubsection*{Stability of the classical and the affine Sobolev inequalities}

Lions \cite{l85} has proved qualitative stability for the classical Sobolev inequality in the sense that if the \emph{Sobolev
deficit}
\begin{align*}
\delta_{\rm Sob}(f):=\frac{\|\nabla f\|_p}{\|f\|_{p^*}}
  -S_{\rm Sob}
\end{align*}
is small, then the distance of $f$ to the set $\mathcal M_{\rm Sob}$ is small. A natural problem, raised by Brezis and Lieb \cite{bl85}, is whether there is a quantitative version of this stability result, in the sense that the Sobolev deficit controls a suitable distance from a function to the manifold $\mathcal M_{\rm Sob}$.
For $p=2$, this question was answered affirmatively by Bianchi and Egnell \cite{be91}. They
proved that the deficit controls the squared $\dot W^{1,2}$-distance to
$\mathcal M_{\rm Sob}$. Their estimate is sharp both in the strength of the distance and in
the exponent. A stability estimate that reflects the optimal dependence of the constant on the dimension $n$ was recently obtained in \cite{deffl25}.

For $p\neq2$, the stability problem is more difficult, since the Hilbert space structure is no longer available. After preliminary works in \cite{cfmp09,fn19,n20} the sharp gradient stability estimate for the full range $1<p<n$ was proved by
Figalli and Zhang \cite{fz22}. More precisely, if $p\in(1,n)$, then
there is a positive constant $c$, depending only on $n$ and $p$,
such that, for all $f\in\dot{W}^{1,p}$,
\begin{align}\label{eq:figallizhang}
 \delta_{\rm Sob}(f)
  \ge
  c
  \inf_{v\in\mathcal M_{\rm Sob}}
  \left(
  \frac{\|\nabla f-\nabla v\|_p}{\|\nabla f\|_p}
  \right)^{\max\{2,p\}}.
\end{align}
Thus, the classical Sobolev deficit controls the strongest natural distance to the family of extremal functions, namely the distance at the level of gradients.

Moreover, the exponent $\max\{2,p\}$ is sharp. In particular, the fact that the exponent is strictly bigger than two for $p>2$ is surprising. This Figalli--Zhang phenomenon has since then also been observed in other stability inequalities, including \cite{glz25,fp24,wz25,fpr25}. As we have already mentioned after Theorem \ref{thm-stab}, it was noted in \cite{fp24} and later independently in \cite{iz25} that the Figalli--Zhang bound can be strengthened for $p>2$ to
$$
\delta_{\rm Sob}(f)
  \ge
  c
  \inf_{v\in\mathcal M_{\rm Sob}}
  \left\{
  \frac{\|\nabla f-\nabla v\|_p^p}{\|\nabla f\|_p^p} + \frac{\||\nabla v|^{(p-2)/2}(\nabla f-\nabla v)\|_2^2}{\|\nabla f \|_p^2}
  \right\}.
$$
The second term on the right-hand side is important in some applications \cite{iz25}.

In passing we mention a completely different mechanism for stability of functional inequalities with an exponent strictly bigger than two, which was explored, for instance, in \cite{ens22,f24,fp24a}. See also the survey \cite{f22}.

So far, our discussion of stability has focused on the classical Sobolev inequality. There are also a few existing works in the affine case. Wang \cite{w13} studied
stability for the affine P\'olya--Szeg\H{o} principle and Nguyen \cite{n16} proved a stability result for
the affine Sobolev inequality on $BV(\mathbb R^n)$. Finding the optimal exponent in the latter case, however, is still an open problem. In addition to the work \cite{flz26b} by Fan et al., which we have already mentioned and where a slightly weaker form of Theorem \ref{thm-stab} was proved, we mention the sharp stability result for the affine fractional Sobolev inequality \cite{flz26} by the same authors.


\subsubsection*{Stability of critical points of the classical Sobolev inequalities}

Another natural direction is to study stability at the level of critical
points, rather than minimizers, of the Sobolev quotient. More precisely, one
may ask whether a function that almost solves the Euler--Lagrange equation
associated with the Sobolev quotient must be quantitatively close to a sum of far-apart optimizers. In the setting of the classical Sobolev inequality, this means bounding the
distance from $u$ to the family of sums of Aubin--Talenti functions in terms of the $\dot W^{-1,p'}$-norm of the residual
$$
-\operatorname{div}(|\nabla f|^{p-2}\nabla f)
- S_{\rm Sob}^p \, |f|^{p^*-2}f \,.
$$

The qualitative theory goes back to the compactness theorem of
Struwe \cite{s84}. In the case $p=2$, Struwe proved that a nonnegative
Palais--Smale sequence for the critical Sobolev functional decomposes, up to
a small error in $\dot W^{1,2}$, into a sum of weakly interacting Aubin--Talenti functions. This phenomenon is usually referred to as bubbling and therefore the minimizers are often referred to as `bubbles'. The
corresponding compactness theory for the critical $p$-Laplacian was
later developed by Alves \cite{a02} and Mercuri and Willem \cite{mw10}. Thus,
under an a priori energy condition that excludes the appearance of more
than one bubble, the qualitative theory implies that an almost critical
point is close to a single bubble.

The first sharp quantitative result in this direction was obtained, for
$p=2$ and in the single-bubble regime, by Ciraolo et al.\
\cite{cfm18}. They proved a linear estimate of the $\dot W^{1,2}$-distance
from a bubble in terms of the Euler--Lagrange residual. The multiple-bubble
case was then studied by Figalli and Glaudo \cite{fg20}, who proved that the
same linear control holds in dimensions $3\le n\le5$, while they showed that it fails in
dimensions $n\ge6$. The sharp quantitative estimates in the higher
dimensional bubbling regime were later obtained by Deng et al.\ \cite{dsw25}. The bound is no longer linear due to the interaction of the bubbles.

For general $1<p<n$, the problem is substantially more delicate and still not completely solved. Recently, Liu and Zhang \cite{lz25}, following the
sharp stability framework of Figalli and Zhang \cite{fz22}, proved the sharp
single-bubble estimate
$$
\inf_{v\in\mathcal M_{\rm Sob}^{(1)}} \left\|\nabla f-\nabla v \right\|_{p}^{\max\{1,p-1\}}
\le
C
\left\|\operatorname{div}(|\nabla f|^{p-2}\nabla f)
+ S_{\rm Sob}^p\, |f|^{p^*-2}f \right\|_{W^{-1,p'}}
$$
under the assumption
$$
\frac12\, S_{\rm Sob}^p \leq \|\nabla f\|_p^p \leq \frac32\, S_{\rm Sob}^p \,.
$$
Here
\begin{equation*}
    \mathcal M_{\rm Sob}^{(1)} := \left\{ f\in \mathcal M_{\rm Sob} : \|f\|_{p_*}=1 \right\}.
\end{equation*}
This result gives the sharp quantitative stability of the Struwe-type
decomposition in the absence of bubbling. The exponent $\max\{1,p-1\}$ is optimal and is related to the exponent $\max\{2,p\}$ in \eqref{eq:figallizhang}.

Independently, Ciraolo and Gatti \cite{cg26} developed a different approach based on $P$-functions and integral identities.
Their method yields a quantitative closeness result for perturbations of the
critical $p$-Laplace equation, with a non-sharp exponent. Although the exponent
is not optimal, their approach is rather flexible. This has recently been
exploited by Antonini et al.\ \cite{acg26} to study classification,
Struwe-type decomposition, and stability for the anisotropic critical
$p$-Laplace equation. The analogue of the multi-bubble results by Figalli and Glaudo \cite{fg20} and Deng et al.\ \cite{dsw25} is still open.

Related to the `other' mechanism for non-quadratic stability for Sobolev-type inequalities mentioned before, there is a corresponding `other' mechanism for non-linear stability for critical points. In the setting of Caffarelli--Kohn--Nirenberg inequalities this was investigated in \cite{zz26,ww26}.


\subsection{Organization of this paper}

\subsubsection*{Outline of the proof of Theorem \ref{thm-stab}}

The overall strategy follows the sharp
gradient stability argument of Figalli--Zhang \cite{fz22}
(see also \cite{be91,fn19}). We first use a compactness argument to reduce the problem to functions that are close, up to
affine transformations, to the manifold $\mathcal M_{\rm aff}$ in Section \ref{sec-compact}.
Then the main new difficulty is the expansion of the affine energy in the sense of
\cite[Section 2]{fz22},
which is the objective of Section \ref{sec-expan}. Unlike the classical
Dirichlet energy, $\mathcal E_p$ is obtained from the directional quantities
$\{\|\nabla_\xi f\|_p^p\}_{\xi\in\mathbb S^{n-1}}$
through a nonlinear negative-power average on the sphere. Therefore, after expanding
the directional energies, one also has to expand the outer affine functional. This
produces an additional second-order term of variance type on the sphere
[see \eqref{prop-expan-local-e2}]. Such a
term is absent in the classical Sobolev inequality. Its role is to detect the additional
trace-free affine directions in the tangent space of $\mathcal M_{\rm aff}$.

A key point is to control this new variance term. Following the affine Hessian
viewpoint used in the affine fractional setting of Fan et al.\ \cite{flz26}, we
establish an affine spectral gap inequality involving
this variance term in Section \ref{sec-spectral}
(precisely, Theorems \ref{thm-spectral} and \ref{thm-spectral-wei}). A similar analysis has appeared independently and concurrently in \cite{flz26b}. Combining these and the nonlinear
expansion, we obtain the desired lower bound for the affine deficit,
and complete the proof of Theorem \ref{thm-stab} in Section \ref{sec-proof}.


\subsubsection*{Outline of the proof of Theorem \ref{thm-bubble}}

The overall strategy follows Liu and Zhang \cite[Theorem 1.2]{lz25}.
As a preliminary step, Section \ref{sec:classification} classifies the solutions of the Euler--Lagrange equation for the affine Sobolev quotient. Relying on this classification, we use a compactness argument to reduce the problem to functions that are close, up to affine transformations, to the manifold $\mathcal M_{\rm aff}^{(1)}$ in Section \ref{sec-compact}. Then we expand the affine Laplacian around a normalized affine Aubin--Talenti bubble in
Proposition \ref{prop-expan-lap}. The expansion is parallel to the one used
in the proof of Theorem \ref{thm-stab}: after expanding the directional terms
and the outer affine functional,
one obtains the same variance operator on $\mathbb S^{n-1}$. Hence,
we need to establish the corresponding
affine spectral gap inequality in Theorem \ref{thm-spectral-wei}.

There is, however, an additional difficulty in the critical point problem. The normalized
affine bubbles solve the Euler--Lagrange equation with a \emph{fixed amplitude}, while the
spectral gap is naturally formulated after removing the full tangent space generated by the
amplitude, scaling, translation, and trace-free affine directions. To handle this
mismatch, we first perform the modulation on the enlarged cone of affine bubbles with
a free amplitude parameter. This gives the desired orthogonality with respect to this
full tangent space of directions. We then test the residual against the approximate bubble to
prove that the optimal amplitude is close to one [see \eqref{thm-bubble-e4}].
This is the affine counterpart of the
coefficient estimate in Figalli and Glaudo \cite[Proposition 3.11]{fg20}. Combining this amplitude estimate
with the above affine spectral gap inequality and the nonlinear expansion, we obtain the
desired quantitative stability estimate for the affine Euler--Lagrange equation.


\subsection{Explicit parametrization}

When presenting our main results, Theorems \ref{thm-stab} and \ref{thm-bubble}, we have opted for an `abstract' formulation, in the sense that only the sets $\mathcal M_{\rm aff}$ and $\mathcal{M}_{\rm aff}^{(1)}$ appear and not their explicit form. However, when proving these results, we need the explicit form of all the elements of these sets. We collect the relevant information here for future reference.

Before doing this, we can introduce slightly stronger forms of our main theorem, which we will show. The strengthening of Theorem \ref{thm-stab} is as follows: If $p\in[2,n)$, then there is a positive constant $c_{n,p}$,
depending only on $n$ and $p$,
such that, for any $f\in \dot{W}^{1,p}$,
\begin{align}\label{eq-stabstrengthened}
&{\mathcal{E}_p(f)}^p
-S_{\rm aff}^p\|f\|_{p^*}^p\notag\\
&\quad\ge c_{n,p}
\inf_{v\in \mathcal M_{\rm Sob},\, T\in \operatorname{SL}(n)}\left\{\|\nabla(f\circ T -v)\|_{p}^p
+\int_{\mathbb{R}^n}
|\nabla v|^{p-2}|\nabla(f\circ T- v)|^2
\right\}.
\end{align}
The strengthening of Theorem \ref{thm-bubble} is as follows: If $p\in[2,n)$, then there is a positive constant
$c_{n,p}$, depending only on $n$ and $p$,
such that, for every $f\in\dot{W}^{1,p}$
satisfying \eqref{thm-bubble-e8},
\begin{align}\label{thm-bubble-e1strengthened}
\delta_{\rm aff}(f)\ge c_{n,p}
\inf_{T\in \operatorname{SL}(n),v\in\mathcal{M}_{\rm Sob}^{(1)}}
\left\{\|\nabla(f\circ T-v)\|_p^{p-1}
+\left[\int_{\mathbb{R}^n}|\nabla v|^{p-2}|\nabla(f\circ T-v)|^2\right]^{\frac{p-1}{p}}\right\}.
\end{align}
The strengthening consists in the fact that in \eqref{eq-stabstrengthened} and \eqref{thm-bubble-e1strengthened} the sets $\mathcal M_{\rm Sob}$ and $\mathcal M_{\rm Sob}^{(1)}$ appear, rather than the larger sets $\mathcal M_{\rm aff}$ and $\mathcal M_{\rm aff}^{(1)}$. We mention that
the version of \eqref{eq-stabstrengthened} containing only
the first term on the right-hand side was proved in \cite{flz26b}.

\bigskip

We now turn to the explicit parametrization of the sets $\mathcal M_{\rm aff}$, $\mathcal M_{\rm Sob}$, and their normalized variants. An important role is played by the function
\begin{equation}
    \label{eq:standardminimizer}
    U(x) := c_0 \, (1+|x|^{p'})^{\frac{p-n}{p}},
\end{equation}
where the constant $c_0>0$, depending on $n$ and $p$, is determined in such a way that
\begin{equation}
    \label{eq:standardminimizernorm}
    \|U\|_{p^*}=1.
\end{equation}
Then the theorem by Lutwak et al.\ \cite{lyz02} (see also \cite{hjm16}) says that $U$ is the unique optimizer of the affine Sobolev inequality \eqref{eq-affSob} up to symmetries. More precisely, for the set $\mathcal M_{\rm aff}$, defined in \eqref{eq:defmaff}, we have
\begin{align}\label{def-min}
\mathcal{M}_{\rm aff}=\left\{c\, |\det B \,|^{\frac{p-n}{np}} \,
U\left(B^{-1}(\cdot-x_0)\right):
c\in\mathbb{R}\setminus\{0\},\
B\in\operatorname{GL}(n),\ x_0\in\mathbb{R}^n\right\}.
\end{align}

This theorem should be compared with the Rodemich--Aubin--Talenti theorem, which characterizes the optimizers of the classical Sobolev inequality \eqref{eq-Sob}. In this classical setting, $U_0$ is still the unique minimizer up to symmetries of the inequality, but the group of symmetries is smaller. We see that $\mathcal M_{\rm Sob}$, defined in \eqref{eq:defmsob}, is given by the same formula as \eqref{def-min}, but with matrices $B$ restricted to be of the form $B=\lambda Q$ with $\lambda>0$ and $Q\in O(n)$.

We prefer to work with a different parametrization of $\mathcal M_{\rm aff}$ from that
in \eqref{def-min}. For every $B\in \operatorname{GL}(n)$,
by the polar decomposition,
there are unique $\lambda\in(0,\infty)$,
$S\in M_n(\mathbb{R})$ satisfying $S=S^T$ and $\operatorname{Tr} S=0$,
and $Q\in O(n)$ such that
$B=\lambda Q e^{S}$.
Since $U$ is radial,
it is invariant under $Q\in O(n)$ and hence, for every
$x_0\in\mathbb R^n$,
$$
U\left(B^{-1}(\cdot-x_0)\right)
=U\left(\lambda^{-1} e^{-S}(\cdot-x_0)\right).
$$
Therefore, by \eqref{def-min},
we obtain
\begin{align}\label{def-min-equiv}
\mathcal{M}_{\rm aff}=\left\{cT_{\lambda,S,x}U:
c\in\mathbb{R}\setminus\{0\},\ \lambda\in(0,\infty),
\ S=S^T,\ \operatorname{Tr}S=0,\ x\in\mathbb{R}^n\right\},
\end{align}
where, for any $f\in L^1_{\rm loc}$,
\begin{align}
\label{eq:deftlambdasx}
T_{\lambda,S,x}f(\cdot):=\lambda^{-\frac{n-p}{p}}f\left(\lambda^{-1}e^{-S}
(\cdot-x)\right).
\end{align}
Here and in what follows,
we refer to $(c,\lambda,S,x)$ as the \emph{affine transformation parameters}
and
$cT_{\lambda,S,x}$ as the \emph{affine transformation}.
The advantage
of this parametrization is that the parameters
$(c,\lambda,S,x)$ are uniquely determined by the corresponding extremal
function. Moreover, we define
$$
\mathscr{A}:=(\mathbb{R}\setminus\{0\})\times(0,\infty)
\times\left\{
S\in M_n(\mathbb{R}): S=S^T,\ \operatorname{Tr}S=0
\right\}\times\mathbb{R}^n
$$
and consider $\mathscr A$ as a finite-dimensional parameter space, endowed
with the topology induced from
$
\mathbb R\times(0,\infty)\times M_n(\mathbb R)\times\mathbb R^n.
$

After introducing these notation, we can restate our first main result \eqref{thm-bubble-e1strengthened} as follows: If $p\in[2,n)$, then there is a positive constant $c_{n,p}$,
depending only on $n$ and $p$,
such that, for any $f\in \dot{W}^{1,p}$,
\begin{align}\label{eq-stab-e}
&{\mathcal{E}_p(f)}^p
-S_{\rm aff}^p\|f\|_{p^*}^p\notag\\
&\quad\ge c_{n,p}
\inf_{(c,\lambda,S,x)\in\mathscr{A}}\left\{\|\nabla(T_{\lambda,S,x}f-cU)\|_{p}^p
+\int_{\mathbb{R}^n}
|c\nabla U|^{p-2}|\nabla(T_{\lambda,S,x}f-cU)|^2
\right\}.
\end{align}
The equivalence of \eqref{thm-bubble-e1strengthened} and \eqref{eq-stab-e} follows from a simple change of variables, noting that the functions $c \lambda^\frac{n-p}{p} U(\lambda \cdot +x)$ exhaust $\mathcal M_{\rm Sob}$ as $(c,\lambda,x)$ run through $(\mathbb R\setminus\{0\})\times(0,\infty) \times\mathbb R^n$.

Similarly, we can restate our second main result \eqref{thm-bubble-e1strengthened} as follows: If $p\in[2,n)$, then there is a positive constant
$c_{n,p}$, depending only on $n$ and $p$,
such that, for every $f\in\dot{W}^{1,p}$
satisfying \eqref{thm-bubble-e8},
\begin{align}\label{thm-bubble-e6}
\delta_{\rm aff}(f)\ge c_{n,p}
\inf_{\genfrac{}{}{0pt}{}{(c,\lambda,S,x)\in\mathscr{A}}{c\in\{-1,1\}}}\left\{\|\nabla (T_{\lambda,S,x}f-cU)\|_p^{p-1}
+\left[\int_{\mathbb{R}^n}
|\nabla U|^{p-2}|\nabla (T_{\lambda,S,x}f-cU)|^2\right]^{
\frac{p-1}{p}}\right\}.
\end{align}

In what follows we will prove our main results in the forms \eqref{eq-stab-e} and \eqref{thm-bubble-e6}.

\medskip

\emph{Notation.}
Let $\mathbb{N}:=\{1,2,\ldots\}$ and $\mathbb{Z}_+:=\mathbb{N}\cup\{0\}$.
In addition, we always denote by $C$ a
\emph{positive constant} that is independent
of the main parameters involved, but may vary from line to line.
We use $C_{\alpha,\dots}$ to denote a positive constant depending
on the indicated parameters $\alpha,\, \dots$.
The notation $f\lesssim g$ means $f\le Cg$
and, if $f\lesssim g\lesssim f$, then we write $f\sim g$.
Moreover, the notation $s\to 0^+$
means that there is a $\delta\in(0,\infty)$
such that $s\in (0,\delta)$ and $s\to 0$.


\section{Classification of positive energy solutions to
the critical affine $p$-Laplace equation}\label{sec:classification}

In this section, we show the following classification theorem
of the positive energy solutions to the critical affine equation. We recall that the set $\mathcal M_{\rm aff}^{(1)}$ was defined in \eqref{eq:defmaff1}.

\begin{theorem}\label{thm-solu}
If $p\in(1,n)$ and $f$ is a solution of the critical equation
\begin{align*}
\begin{cases}
\displaystyle{-\Delta_p^{{\rm aff}}(f)=S_{\rm aff}^p f^{p^*-1}},\\
f>0,\quad f\in\dot{W}^{1,p},
\end{cases}
\end{align*}
then $f\in\mathcal{M}^{(1)}_{\rm aff}$.
\end{theorem}

There are two important ingredients in our proof of this theorem. The first is a classification theorem for anisotropic $p$-Laplace equations by Ciraolo et al.~\cite{cfr20}. To apply this theorem, we interpret our affine $p$-Laplace equation as an anisotropic $p$-Laplace equation with an anisotropic norm that depends on $f$ itself \cite{hjm21}. The result from \cite{cfr20} then implies, among other things, that the superlevel sets of the solution are concentric unit balls in the dual of the given anisotropic norm. It is at this point that our second ingredient comes in. That is, we prove that if the superlevel sets of a critical point are concentric balls with respect to the Minkowski functional of a sufficiently smooth convex set, then this set is necessarily an ellipsoid. This is the content of Proposition \ref{prop-variation}, which is the technical main result of this section. Combining these two ingredients one easily obtains the classification result in Theorem \ref{thm-solu}.

We begin the proof by explaining in more detail the rewriting of the affine $p$-Laplacian as an anisotropic $p$-Laplacian. Indeed, for $f\in\dot{W}^{1,p}$,
we have
\begin{align}\label{eq-Lap-Hf}
\Delta_p^{\rm aff}(f)=\operatorname{div}\left(\nabla\left(\frac{H_f^p}{p}\right)(\nabla f)\right),
\end{align}
where, for every $\zeta\in\mathbb{R}^n$,
\begin{align}
\label{eq:defhf}
H_f(\zeta):=\mathcal{E}_p(f)^{1+\frac np}\left(
\int_{\mathbb{S}^{n-1}}\|\nabla_\xi f\|_p^{-n-p}|\zeta\cdot\xi|^p\,d\xi\right)^{\frac1p}.
\end{align}
The right-hand side of \eqref{eq-Lap-Hf} is precisely the
anisotropic $p$-Laplacian (or Finsler $p$-Laplacian) $\Delta_{p}^{H_f}$
associated with the norm $H_f$. To apply the result in \cite{cfr20}, we still need to verify that
the Hessian $\nabla^2 (\frac12H_f^2)$ satisfies
the uniform positive definiteness condition \cite[(1.11)]{cfr20}.
Hence, we first establish the following claim.

\begin{lemma}\label{lem-solu}
Let $p\in(1,\infty)$ and let $f\in\dot{W}^{1,p}$ be
a non-constant function. Then
there are positive constants $\lambda$ and $\Lambda$, depending only on $n$,
$p$, and $f$, such that, for all $\zeta,\eta\in\mathbb{R}^n$,
\begin{align}\label{lem-solu-e0}
\lambda \, |\eta|^2\le
\left\langle\nabla^2\left(\frac{1}{2}H_{f}^2\right)(\zeta) \, \eta,\eta\right\rangle
\le \Lambda \, |\eta|^2.
\end{align}
\end{lemma}

\begin{proof}
Define $K:=\{\xi\in\mathbb{R}^{n}:\|\nabla_\xi f\|_p\le 1\}$. Since $f\in\dot{W}^{1,p}$ is non-constant, from
\cite[(3.1)]{n16}, it follows that $K$ is an origin-symmetric
convex body with $\mathbf{0}\in\operatorname{Int} K$. (A \emph{convex body} is a convex set that is compact with nonempty interior.)
For every $\zeta\in\mathbb{R}^n$
and $t\in(0,\infty)$, define
$$F(\zeta):=\frac{1}{|K|}\int_K|\zeta\cdot\xi|^p\,d\xi$$ and $g(t):=\frac12 t^{\frac 2p}.$
Note that, applying polar coordinates, we find that, for every $\zeta\in\mathbb{R}^n$,
\begin{align*}
H_f(\zeta)^p=
(n+p) \, \mathcal{E}_p^{n+p}(f)\int_{K} |\zeta\cdot\xi|^p\,d\xi.
\end{align*}
Hence, to complete the present proof,
we only need to establish \eqref{lem-solu-e0}
with $\frac12H_f^2$ therein replaced by $g\circ F$.
By the homogeneity, we only need to consider $\eta\in\mathbb{S}^{n-1}$.
For any $t\in(0,\infty)$, define $\Phi(t):=g(F(\zeta+t\eta))$. Then, by the chain rule, we have
\begin{align}\label{lem-solu-e1}
\left\langle\nabla^2\left(g\circ F\right)(\zeta)\,\eta,\eta\right\rangle
=\Phi''(0)
=g'(F(\zeta))\left\langle\nabla^2 F(\zeta)\,\eta,\eta\right\rangle
+g''(F(\zeta))\left(\nabla F(\zeta)\cdot \eta\right)^2.
\end{align}
Note that
\begin{align*}
\nabla F(\zeta)=\frac{p}{|K|}\int_K|\zeta\cdot\xi|^{p-2}(\zeta\cdot\xi)\xi\,d\xi
\end{align*}
and
\begin{align*}
\left\langle\nabla^2 F(\zeta)\,\eta,\eta\right\rangle=
\frac{p(p-1)}{|K|}\int_K |\zeta\cdot\xi|^{p-2}(\eta\cdot\xi)^2\,d\xi.
\end{align*}
Substituting these into \eqref{lem-solu-e1}, we obtain
\begin{align}\label{lem-solu-e2}
\left\langle\nabla^2\left(g\circ F\right)(\zeta)\, \eta,\eta\right\rangle
&=\frac{p-1}{|K|}F(\zeta)^{\frac2p-1}
\int_K|\zeta\cdot\xi|^{p-2}(\eta\cdot\xi)^2\,d\xi\notag\\
&\qquad+\frac{2-p}{|K|^2}F(\zeta)^{\frac2p-2}
\left[\int_{K}|\zeta\cdot\xi|^{p-2}(\zeta\cdot\xi)(\eta\cdot\xi)\,d\xi\right]^2.
\end{align}
Therefore, from the Cauchy--Schwarz inequality, we easily infer that
the upper estimate in \eqref{lem-solu-e0} holds.
We now focus on the lower estimate. Indeed, by the homogeneity again,
we conclude that the left-hand side
of \eqref{lem-solu-e2} is homogeneous of degree $0$
with respect to $\zeta$.
This implies that we only need to consider $\zeta\in\mathbb{S}^{n-1}$.
Next, we consider the following two cases for $p$.

\emph{Case 1. $p\in(1,2]$}.
In this case, we can only consider the first term in
the right-hand side of \eqref{lem-solu-e2}.
Moreover, from the fact that $K$ has nonempty interior,
it follows that $|K|>0$ and hence
\begin{align*}
\int_K|\zeta\cdot\xi|^{p-2}(\eta\cdot\xi)^2\,d\xi
&\ge\left(\sup_{u\in K}|u|\right)^{p-2}\int_{K}(\eta\cdot\xi)^2\,d\xi\\
&=\left(\sup_{u\in K}|u|\right)^{p-2}\int_{K\setminus\{\xi:\eta\cdot\xi=0\}}(\eta\cdot\xi)^2\,d\xi>0.
\end{align*}
By this, the compactness of $\mathbb{S}^{n-1}$, and the continuity of
the corresponding integral functional, there is a $\lambda\in(0,\infty)$
such that, for all $\zeta,\eta\in\mathbb{S}^{n-1}$,
\begin{align*}
\left\langle\nabla^2\left(g\circ F\right)(\zeta)\, \eta,\eta\right\rangle
\ge \frac{p-1}{|K|}F(\zeta)^{\frac2p-1}
\int_K|\zeta\cdot\xi|^{p-2}(\eta\cdot\xi)^2\,d\xi\ge\lambda.
\end{align*}
This, together with the homogeneity, proves the lower estimate in \eqref{lem-solu-e0}
in this case.

\emph{Case 2. $p\in(2,\infty)$}.
From H\"older's inequality, we deduce that
\begin{align*}
\left[\int_{K}|\zeta\cdot\xi|^{p-2}(\zeta\cdot\xi)(\eta\cdot\xi)\,d\xi\right]^2
&\le \int_K|\zeta\cdot\xi|^p\,d\xi\int_K|\zeta\cdot\xi|^{p-2}(\eta\cdot\xi)^2\,d\xi\\
&=|K|F(\zeta)\int_K|\zeta\cdot\xi|^{p-2}(\eta\cdot\xi)^2\,d\xi.
\end{align*}
Then, applying the fact that $K$ has an interior again,
we obtain
\begin{align*}
\operatorname{LHS}\eqref{lem-solu-e2}&\ge
\frac{1}{|K|}F(\zeta)^{\frac2p-1}
\int_K|\zeta\cdot\xi|^{p-2}(\eta\cdot\xi)^2\,d\xi\\
&=\frac{1}{|K|}F(\zeta)^{\frac2p-1}
\int_{K\setminus(\{\xi:|\zeta\cdot\xi||\eta\cdot\xi|=0\})}
|\zeta\cdot\xi|^{p-2}(\eta\cdot\xi)^2\,d\xi>0.
\end{align*}
Here and in what follows, we use $\operatorname{LHS}$ and $\operatorname{RHS}$
to denote, respectively, the left-hand side and the right-hand side.
Hence, similar to Case 1, we can conclude the lower estimate in \eqref{lem-solu-e0}
in this case.

This then completes the present proof.
\end{proof}

For a convex body $K\subset\mathbb R^n$ we define the \emph{Minkowski functional} by setting
\begin{align*}
\|x\|_K:=\inf\left\{\lambda\in(0,\infty):
\lambda x\in K\right\}
\qquad\text{for all}\ x\in\mathbb R^n \,.
\end{align*}
We say that a convex body $K$ is of class $C^2_+$ if
its boundary is $C^2$ with everywhere positive curvature.
To show the classification result Theorem \ref{thm-solu}, we need the classification result for convex bodies.

\begin{proposition}\label{prop-variation}
Let $K\subset\mathbb{R}^n$ be an
origin-symmetric convex body of class $C^2_+$ with $\mathbf{0}\in\operatorname{int}K$.  Let
$\Phi\in C[0,\infty)\cap C^2(0,\infty)$ satisfy
\begin{equation} \label{eq:profile-assumption}
 0<\int_0^\infty
 \left[
   |\Phi'(r)|^p+
   |r\Phi''(r)|^p
 \right]r^{n-1}\dd r<\infty
\end{equation}
and let $\sigma$ be a Borel function on $\mathbb{R}$.
If
\begin{equation*}
 J_\sigma(v)
 :=\frac1p\mathcal E_p(v)^p-
   \int_{\mathbb{R}^n}\sigma(v(x))\,d x
\end{equation*}
is well defined and Fr\'echet differentiable at $u(\cdot):=\Phi(\|\cdot\|_{K})$ on $\dot W^{1,p}$
and
\begin{equation}
 DJ_\sigma(u)=0,
 \label{eq:criticality}
\end{equation}
then $K$ is an ellipsoid centered at the origin.
\end{proposition}

To prove this proposition, we need the continuous Steiner symmetrization (see, for example,
\cite{msy25}).
Fix a convex body $K\subset\mathbb R^n$ and a direction $\xi\in\mathbb S^{n-1}$, and let
$\overline D_\xi$ denote the image of the orthogonal projection
of $K$ onto $\xi^\perp$ and
$D_\xi:=\operatorname{int}\overline D_\xi.$
For every $x\in\mathbb R^n$,
we can uniquely decompose
$x=y+s\xi$ with
$y\in\xi^\perp$ and $s\in\mathbb R$.
For any $y\in\overline D_\xi$, define
$$
f_\xi(y)
:=
\max\{s\in\mathbb R:y+s\xi\in K\}
\quad\text{and}\quad
g_\xi(y)
:=
\max\{s\in\mathbb R:y-s\xi\in K\}.
$$
Since $K$ is convex, every fiber parallel to $\xi$ is an interval,
and hence
$$
K
=
\left\{
y+s\xi:
y\in\overline D_\xi,\
-g_\xi(y)\le s\le f_\xi(y)
\right\}.
$$
Moreover, $\overline D_\xi$ is convex and the functions
$f_\xi$ and $g_\xi$ are concave.
If $K\in C^2_+$, then
$f_\xi,g_\xi\in C^2(D_\xi)\cap C(\overline D_\xi)$ and, for any $y\in D_\xi$,
$-\nabla^2f_\xi(y)$ and $-\nabla^2g_\xi(y)$ are both positive definite.
For $t\in[0,2]$,  the \emph{continuous Steiner symmetrization} of $K$ in the direction
$\xi$ is defined by setting
$$
S_\xi^t K
:=
\left\{
y+s\xi:
y\in\overline D_\xi,
-g_{\xi,t}(y)\le s\le f_{\xi,t}(y)
\right\},
$$
where, for any $y\in\overline D_\xi$,
\begin{align*}
f_{\xi,t}(y)
:=
\left(1-\frac{t}{2}\right)f_\xi(y)
+\frac{t}{2}g_\xi(y)\quad\text{and}\quad
g_{\xi,t}(y)
:=
\frac{t}{2}f_\xi(y)
+\left(1-\frac{t}{2}\right)g_\xi(y).
\end{align*}
It is easy to verify that
$S_\xi^0K=K$,
that $S_\xi^1K$ is exactly the classical Steiner symmetrization
(see \cite{s14}) with respect to $\xi^\perp$, and that
$S_\xi^2K$ is the reflection of $K$ across $\xi^\perp$.

For every differentiable function $h$, define
\begin{equation}
    \label{eq:langlenotation}
    \langle h\rangle (x):=h(x)-x\cdot \nabla h(x)
    \qquad\text{for all}\ x\in\operatorname{Dom}(h) \,.
\end{equation}
Geometrically, $\langle h\rangle(x)$ is the vertical intercept of the tangent hyperplane to the graph of $h$
at $(x,h(x))$; equivalently, after normalization by $\sqrt{1+|\nabla h(x)|^2}$, it is the signed
distance of this tangent hyperplane from the origin. We use the following simple fact;
see also \cite[p.\,125]{lyz00}.

\begin{lemma}\label{linear}
    If $\operatorname{Dom}(h)$ is a convex body containing the origin and $h({\bf0})=0$, then $\langle h\rangle \equiv0$ implies that $h$ is linear.
\end{lemma}

\begin{proof}
Indeed, for every $x\in \operatorname{Dom}(h)$,
$\langle h\rangle(x)=0$ gives $\frac{d}{dr}(\frac{h(rx)}{r})=0$
for every $r\in(0,1]$. Hence,
$h(rx)=rh(x)$ for every $x\in\operatorname{Dom}(h)$ and $r\in(0,1]$.
From this and $h({\bf0})=0$, we infer that, as $r\to0^+$,
\begin{align*}
rh(x)=h(rx)=h({\bf0})+r\nabla h({\bf 0})\cdot x+o(r)
=r\nabla h({\bf0})\cdot x+o(r) \,,
\end{align*}
and hence $h(x)=\nabla h({\bf0})\cdot x$.
This shows the above claim.
\end{proof}

The following proof of Proposition \ref{prop-variation} is close in
spirit to the continuous Steiner argument of Dai and Wang \cite{dw26},
inspired by the earlier approach of Milman et al.\ \cite{msy25}.
Here we use this mechanism in an infinitesimal variational form.

\begin{proof}[Proof of Proposition \ref{prop-variation}]
The idea is to apply the continuous Steiner symmetrization
in every direction $\xi\in\mathbb{S}^{n-1}$.
After a rotation of coordinates, we
may assume, without loss of generality, that $\xi=\mathbf{e}_n:=(0,\ldots,0,1)$. To simplify
the notation, we shall henceforth suppress the subscript $\xi$ from
all the quantities associated with this direction, except that we
write
$
K_t:=S_\xi^tK
$
for every $t\in[0,2]$.
For example, we write $D$ in place of
$D_\xi$. Thus,
$$
K_t
=
\left\{
(y,s):
y\in\overline D,\
-g_t(y)\le s\le f_t(y)
\right\}.
$$
In what follows, for any $x\in K$, we always decompose it as
$x=(y,s)$ with $y\in\overline{D}$.
Moreover, since $K$ is origin-symmetric,
we infer that $\overline{D}$ is also origin-symmetric in $\mathbb{R}^{n-1}$ and
$f(\cdot)=g(- \, \cdot)$, and hence
$f_t,g_t$ have the same property for every $t\in[0,2]$.
For any $t\in[0,2]$, define $u_{t}(\cdot):=\Phi(\|\cdot\|_{K_t})$.
Then $u_0=u$. We now divide this proof into the following five steps.

\medskip

\emph{Step 1.}
We first show that, for every sufficiently small $t$,
\begin{align}\label{prop-variation-e1}
\|u_t-u_0\|_{\dot{W}^{1,p}}
=\|\nabla u_t-\nabla u_0\|_p\lesssim t
\end{align}
with the implicit positive constant independent of $t$.
Define $H_t(\cdot):=\|\cdot\|_{K_t}$ .
Then, for every $t\in[0,2]$,
$\nabla u_t=(\Phi'\circ H_t)\nabla H_t$, which implies that
\begin{align*}
\nabla u_t-\nabla u_0=\left(\Phi'\circ H_t-\Phi'\cdot H_0\right)\nabla H_t
+\Phi'\circ H_0\left(\nabla H_t-\nabla H_0\right).
\end{align*}
Since $H_t$ is positively one-homogeneous and differentiable away
from the origin, its gradient is positively zero-homogeneous. Therefore,
by \eqref{eq:profile-assumption}, to prove \eqref{prop-variation-e1},
we only need to show that, for sufficiently small $t$,
\begin{align}\label{prop-vatiation-e2}
\|H_t-H_0\|_{L^\infty(\mathbb{S}^{n-1})}
+\|\nabla H_t-\nabla H_0\|_{L^\infty(\mathbb{S}^{n-1})}\lesssim t.
\end{align}
Indeed, applying the definition, we find that, for every $t\in[0,2]$ and
$(y,s)\in K_t$,
$(y,s+\frac{f(y)-g(y)}{2}t)\in K$. Then there is a positive constant $C_0$
such that $K_t\subset K+C_0 t B({\bf 0},1)$.
From this and ${\bf 0}\in\operatorname{Int}K$, we further deduce that
there is an $r_0\in(0,\infty)$ such that $B({\bf0},r_0)\subset K$ and hence
$K_t\subset (1+\frac{C_0}{r_0}t)K$ for all $t\in[0,2]$.
Similarly, it holds that, for all $t\in[0,2]$,
\begin{align}\label{prop-variation-e4}
\left(1-\frac{C_0}{r_0}t\right)K\subset K_t
\subset\left(1+\frac{C_0}{r_0}t\right)K.
\end{align}
Using this and the definition of $H_t$, we obtain, for every
$t\in[0,2]$, $\|H_t-H_0\|_{L^\infty(\mathbb{S}^{n-1})}\lesssim t.$
This completes the non-gradient estimate in \eqref{prop-vatiation-e2}.
On the other hand, fix $\omega\in\mathbb{S}^{n-1}$ and
define $x_t:=\frac{\omega}{H_t(\omega)}\in\partial K_t$. By the positive zero-homogeneity
and the gradient formula of Minkowski functionals at the boundary, we find that
\begin{align}\label{prop-variation-e3}
\nabla H_t(\omega)=\nabla H_t(x_t)=\frac{\nu_{K_t}(x_t)}{x_t\cdot\nu_{K_t}(x_t)},
\end{align}
where $\nu_{K_t}(x_t)$ denotes the outward unit normal to
$\partial {K_t}$ at $x_t$.
Applying the positivity of $H_0$ and the proved
non-gradient estimate in \eqref{prop-vatiation-e2}, we conclude that,
for sufficiently small $t$,
$|x_t-x_0|\lesssim t$ with the implicit positive constant
depending only on $n$ and $K$.
Since each $K_t$ is convex, the tangent hyperplane to $\partial K_t$
at each $\xi\in\partial K_t$ is a supporting hyperplane; that is,
\begin{align}\label{prop-variation-e5}
(z-\xi)\cdot\nu_{K_t}(\xi)\le 0
\quad\text{for every $z\in K_t$}.
\end{align}
Using \eqref{prop-variation-e4} and $B({\bf0},r_0)\subset K$, we obtain,
for sufficiently small $t$, $B({\bf0},\frac{r_0}{2})\subset K_t$.
Hence, combining this and \eqref{prop-variation-e5},
we further conclude that, for every sufficiently small $t$,
\begin{align}\label{prop-variation-e6}
\xi\cdot\nu_{K_t}(\xi)\ge \frac{r_0}{2}
\quad\text{for every $\xi\in K_t$}.
\end{align}
This, together with \eqref{prop-variation-e3} and
$|x_t-x_0|\lesssim t$ for sufficiently small $t$,
further implies that, to prove the gradient estimate in \eqref{prop-vatiation-e2},
it remains to show that
\begin{align}\label{prop-variation-e7}
|\nu_{K_t}(x_t)-\nu_{K}(x_0)|\lesssim t
\quad\text{for any sufficiently small $t$}.
\end{align}

Now, we prove \eqref{prop-variation-e7}.
Without loss of generality, we may assume that
$x_t=(y_t,f_t(y_t))$
lies on the upper graph of $\partial K_t$, where $y_t\in D_t$.
Define $\widehat{x}_t:=(y_t,f(y_t))\in\partial K$.
Then
\begin{align*}
\nu_{K_t}(x_t)=\frac{(-\nabla f_t(y_t),1)}{\sqrt{1+|\nabla f_t(y_t)|^2}}
\quad\text{and}\quad\nu_{K}(\widehat x_t)=\frac{(-\nabla f(y_t),1)}{\sqrt{1+|\nabla f(y_t)|^2}}.
\end{align*}
Thus, by the definition of $f_t$, we find that
$|\nu_{K_t}(x_t)-\nu_{K}(\widehat x_t)|\lesssim t$.
On the other hand, since $\partial K$ is a compact $C^2$ hypersurface, its Gauss map is
Lipschitz. Therefore,
$$
|\nu_{K}(\widehat x_t)-\nu_K(x_0)|
\lesssim|\widehat x_t-x_0|
\le |\widehat x_t-x_t|+|x_t-x_0|
\lesssim t.
$$
This then implies \eqref{prop-variation-e7} and hence
\eqref{prop-vatiation-e2}. Therefore, we complete the proof of \eqref{prop-variation-e1}.

\medskip

\emph{Step 2.}
For any $t\in[0,2]$,
from the coarea formula (see, for instance,
\cite[p.\,258]{f69}), the gradient formula for Minkowski functionals
[namely \eqref{prop-variation-e3}], a change of variables,
and the fact that $|K_t|=|K|$, it follows that
\begin{align*}
\int_{\mathbb{R}^n}\sigma(u_t(x))\,dx
&=\int_0^\infty\int_{\|x\|_{K_t}=r}
\sigma(\Phi(r))\frac{1}{|\nabla\|x\|_{K_t}|}\,d\mathcal{H}^{n-1}(\omega)\,dr\\
&=\int_0^\infty\int_{\partial K_t}\sigma(\Phi(r))r^{n-1}\omega\cdot \nu_{K_t}(\omega)
\,d\mathcal{H}^{n-1}(\omega)\,dr\\
&=\int_0^\infty \sigma(\Phi(r))r^{n-1}\,dr
\int_{\partial K_t}\omega\cdot \nu_{K_t}(\omega)
\,d\mathcal{H}^{n-1}(\omega)\\
&=n|K_t|\int_0^\infty \sigma(\Phi(r))r^{n-1}\,dr
\quad\text{by the divergence theorem}\\
&=n|K|\int_0^\infty \sigma(\Phi(r))r^{n-1}\,dr \,,
\end{align*}
where $\mathcal{H}^{n-1}$ denotes the $(n-1)$-dimensional Hausdorff measure.
Since the right side is independent of $t$, we obtain, in particular,
\begin{align}
    \label{eq:constundersymm}
    \int_{\mathbb{R}^n}\sigma(u_t(x))\,dx = \int_{\mathbb{R}^n}\sigma(u_0(x))\,dx \,.
\end{align}
Using \eqref{eq:criticality} and \eqref{prop-variation-e1}, we find that,
as $t\to0^+$,
$$J_\sigma(u_t)-J_\sigma(u_0)
=DJ_\sigma(u_0)[u_t-u_0]+o(\|u_t-u_0\|_{\dot W^{1,p}})=o(t),$$
so \eqref{eq:constundersymm} implies
\begin{align}\label{prop-variation-e8}
0=\left.\frac{d}{dt}\right|_{t=0}J_\sigma(u_t)
=\frac1p\left.\frac{d}{dt}\right|_{t=0}\mathcal{E}_p(u_t)^p.
\end{align}

\medskip

\emph{Step 3.}
We now analyze the continuous Steiner symmetrization on each fiber, respectively.
More precisely, for $z\in\mathbb{R}^{n-1}$ and $t\in[0,2]$, define
\begin{align*}
\alpha(z,t):=\max\left\{s\in\mathbb{R}:
\|\nabla u_t\cdot(z,-s)\|_p\le1\right\}\quad\text{and}\quad
\beta(z,t):=\max\left\{s\in\mathbb{R}:
\|\nabla u_t\cdot(z,s)\|_p\le1\right\}.
\end{align*}
By \eqref{prop-variation-e1}, the two convex bodies
$\{\xi\in\mathbb{R}^n:\|\nabla u\cdot\xi\|_p\le 1\}$
and $\{\xi\in\mathbb{R}^{n}:\|\nabla u_t\cdot\xi\|_p\le 1\}$
are comparable, with comparison
constants independent of all sufficiently small $t$.
Hence, for sufficiently small $t$, $\alpha(z,t)$ and $\beta(z,t)$ are uniformly bounded
with respect to $z$ and $t$.
Although these two functions are not necessarily monotone with respect to $t$, we may shift them
in a manner analogous to the symmetrization described above, thereby
obtaining the desired comparison. Indeed, for $z\in\mathbb{R}^{n-1}$
and $t\in[0,2]$, define
\begin{align*}
a(z,t):=\left(1-\frac t2\right)\alpha(z,0)
+\frac t2\beta(z,0)\quad\text{and}\quad
b(z,t):=\left(1-\frac t2\right)\beta(z,0)
+\frac t2\alpha(z,0).
\end{align*}
Note that
\begin{equation}
    \label{eq:aplusb}
    a(z,t)+b(z,t)=\alpha(z,0)+\beta(z,0) \,.
\end{equation}
We now show that, for every sufficiently
small $t$,
\begin{align}
    \label{eq:inequalitiesab}
    \beta(z,t)\ge b(z,t) \qquad\text{and}\qquad \alpha(z,t)\ge a(z,t) \,.
\end{align}
Applying the coarea formula argument again and \eqref{prop-variation-e3}, we obtain, for every
sufficiently small $t$
and every $\xi\in\mathbb{R}^{n}$,
\begin{align*}
\|\nabla u_t\cdot \xi\|_{p}^p
&=\int_0^\infty \int_{\partial K_t}
|\Phi'(r)|^p|\xi\cdot\nu_{K_t}(\omega)|^p
r^{n-1}[\omega\cdot\nu_{K_t}(\omega)]^{1-p}
\,d\mathcal{H}^{n-1}(\omega)\,dr\\
&=A_\Phi\int_{\partial K_t}
|\xi\cdot\nu_{K_t}(\omega)|^p
[\omega\cdot\nu_{K_t}(\omega)]^{1-p}
\,d\mathcal{H}^{n-1}(\omega),
\end{align*}
where $A_\Phi:=\int_0^\infty|\Phi'(r)|^pr^{n-1}\,dr$ and,
here and thereafter, we require $t$ to be sufficiently small because, by
\eqref{prop-variation-e6}, the term $\omega\cdot\nu_{K_t}(\omega)$ is
then positive and away from zero. Recall that the notation $\langle h\rangle$ was introduced in \eqref{eq:langlenotation}.
Thus, from a change of variables, we infer that,
for $(z,z_n)\in\mathbb{R}^{n-1}\times \mathbb{R}$,
\begin{align}\label{prop-variation-e11}
\|\nabla u_t\cdot(z,z_n)\|_p^p
&=A_\Phi\left\{\int_D [\langle f_t\rangle(y)]^{1-p}|z_n-z\cdot\nabla f_t(y)|^p\,dy
+\int_D [\langle g_t\rangle(y)]^{1-p}|z_n-z\cdot\nabla g_t(y)|^p\,dy\right\}\notag\\
&=2A_\Phi\int_D [\langle f_t\rangle(y)]^{1-p}|z_n-z\cdot\nabla f_t(y)|^p\,dy
\quad\text{by $f(\cdot)=g(- \, \cdot)$},
\end{align}
where $\langle f_t\rangle(\cdot)=\langle g_t\rangle(-\,\cdot)$ is also
positive and away from zero by \eqref{prop-variation-e6}.
Therefore, applying this and the elementary inequality
$(a+b)^{1-p}(c+d)^{p}\le a^{1-p}c^p+b^{1-p}d^p$ for
all $a,b\in(0,\infty)$, $c,d\in[0,\infty)$, and $p\in(1,\infty)$
(see \cite[Lemma 8]{lyz00}), we obtain
\begin{align*}
\|\nabla u_t\cdot(z,b(z,t))\|_p^p
&\le 2A_\Phi\int_D \left[\left(1-\frac t2\right)
\langle f\rangle(y)+\frac t2\langle f\rangle(-y) \right]^{1-p}\\
&\qquad\times\left[\left(1-\frac t2\right)|\beta(z,0)-z\cdot\nabla f(y)|
+\frac t2|\alpha(z,0)+z\cdot\nabla f(-y)|
\right]^p\,dy\\
&\le2A_\Phi\left(1-\frac t2\right)\int_D [\langle f\rangle(y)]^{1-p}|\beta(z,0)-z\cdot\nabla f(y)|^p\,dy\\
&\qquad+2A_\Phi\frac t2\int_D [\langle f\rangle(y)]^{1-p}|\alpha(z,0)+z\cdot\nabla f(y)|^p\,dy\\
&=\left(1-\frac t2\right)\|\nabla u_0\cdot(z,\beta(z,0))\|_p^p
+\frac t2\|\nabla u_0\cdot(z,-\alpha(z,0))\|_p^p
\le 1.
\end{align*}
This, combined with the definition of $\beta$,
further implies the first inequality in \eqref{eq:inequalitiesab}. The second one follows similarly.

\medskip

\emph{Step 4.}
We next apply \eqref{prop-variation-e8}
and \eqref{eq:inequalitiesab} to study the first order variation on each fiber.
We define
\[
F(z,t):=\|\nabla u_t\cdot(z,b(z,t))\|_p^p
\]
for appropriate $(z,t)\in\mathbb{R}^{n-1}\times \mathbb{R}$. In this step we prove that
\begin{align}
    \label{eq:limfder}
    \liminf_{t\to0^+}\frac{1-F({\bf 0},t)}{t}=0 \,.
\end{align}

Indeed, using polar coordinates and Tonelli's theorem, we find that,
for every sufficiently small~$t$,
\begin{align*}
\mathcal{E}_p(u_t)^{-n}
&=n\int_{\|\nabla u_t\cdot\xi\|_p\le 1}\,d\xi
=n\int_{\mathbb{R}^{n-1}}\int_{\|\nabla u_t\cdot(z,z_n)\|_p\le 1}\,dz_n\,dz\\
&=n\int_{\mathbb{R}^{n-1}}\left[\beta(z,t)+\alpha(z,t)\right]\,dz,
\end{align*}
which, together with \eqref{eq:aplusb} and \eqref{eq:inequalitiesab}, implies that
\begin{align}\label{prop-variation-e9}
\mathcal{E}_p(u_t)^{-n}-\mathcal{E}_p(u_0)^{-n}
&=n\int_{\mathbb{R}^{n-1}}
\left( \left[\beta(z,t)+\alpha(z,t)\right]
-\left[b(z,t)+a(z,t)\right]\right) dz \notag\\
&\ge n\int_{\mathbb{R}^{n-1}}
\left[\beta(z,t)-b(z,t)\right]\,dz.
\end{align}

Let $U\subset D$ be a fixed compact neighborhood of ${\bf 0}$ in $\mathbb{R}^{n-1}$.
Then, from the definition of $\beta$, we infer that, for every $z\in U$ and
every sufficiently small $t$,
\begin{align*}
0&\le 1-F(z,t)=\|\nabla u_t\cdot(z,\beta(z,t))\|_p^p-
\|\nabla u_t\cdot(z,b(z,t))\|_p^p\\
&=\int_{b(z,t)}^{\beta(z,t)}
\frac{\partial}{\partial s}\|\nabla u_t\cdot(z,s)\|_p^p\,ds\\
&=p\int_{b(z,t)}^{\beta(z,t)}
\int_{\mathbb{R}^n}|\nabla u_t(x)\cdot(z,s)|^{p-2}
\nabla u_t(x)\cdot(z,s)\partial_{x_n}u_t(x)\,dx\,ds\\
&\lesssim \int_{b(z,t)}^{\beta(z,t)}\|\nabla u_t\|_p^p\,ds\lesssim\beta(z,t)-b(z,t),
\end{align*}
where the penultimate step used the uniform boundedness of
$\alpha$ and $\beta$ (and hence of $b$) for sufficiently small $t$,
while the last step follows
from \eqref{prop-variation-e1}.
Combining this, Fatou's lemma, \eqref{prop-variation-e9}, and \eqref{prop-variation-e8},
we conclude that
\begin{align*}
0&\le \int_{\mathbb{R}^{n-1}}
\liminf_{t\to0^+}\frac{1-F(z,t)}t\,dz
\le\liminf_{t\to0^+}\int_{\mathbb{R}^{n-1}}
\frac{1-F(z,t)}{t}\,dz\\
&\lesssim\liminf_{t\to0^+}\int_{\mathbb{R}^{n-1}}\frac{\beta(z,t)-b(z,t)}{t}\,dz
\lesssim\lim_{t\to0^+}\frac{\mathcal{E}_p(u_t)^{-n}-\mathcal{E}_p(u_0)^{-n}}{t}\\
&=-\frac np\mathcal{E}_p(u_0)^{-n-p}
\left.\frac{d}{dt}\right|_{t=0}\mathcal{E}_p(u_t)^p=0.
\end{align*}
Thus, the function $z\mapsto\liminf_{t\to0^+}\frac{1-F(z,t)}{t}$
vanishes at every Lebesgue point and, in
particular, at every point of continuity.

Let us show that this implies \eqref{eq:limfder}.
Indeed, for every $z\in U$, it holds that $F(z,0)=1$ and hence,
for every sufficiently small $t$,
\begin{align}\label{prop-variation-e10}
\frac{1-F(z,t)}{t}=\frac{F(z,0)-F(z,t)}{t}=
-\int_0^1\partial_{x_n}F(z,\theta t)\,d\theta.
\end{align}
Note that
\begin{align*}
\partial_{x_n}F(z,t)=
2A_{\Phi}\int_{D}&
\left[\left(\frac{d}{dt}\langle f_t\rangle^{1-p}(y)\right)
|b(z,t)-z\cdot\nabla f_t(y)|^p\right.\\
&\quad\left.+[\langle f_t\rangle(y)]^{1-p}
\frac{d}{dt}|b(z,t)-z\cdot\nabla f_t(y)|^p\right]
\,dy
\end{align*}
Since $f_t$ and $b(z,t)$ depend affinely on $t$, the map
$s\mapsto \frac{d}{ds}|s|^p=p|s|^{p-2}s$ is continuous for $p>1$, and $\langle f_t\rangle$
is away from zero [by \eqref{prop-variation-e6}],
from \eqref{prop-variation-e10} and the Lebesgue dominated convergence
theorem, we deduce that
$U\ni z\mapsto\lim_{t\to0^+}\frac{1-F(z,t)}{t}=-\partial_{x_n}F(z,0)$ is continuous.
This, in particular, implies \eqref{eq:limfder}, as claimed.

\medskip

\emph{Step 5.}
We can now complete the proof. When $z={\bf 0}$, by definition,
we easily find that, for every $t\in[0,2]$, $\alpha({\bf0},t)
=\beta({\bf0},t)$ and hence $b({\bf0},t)=\beta({\bf0},0)$.
Applying this and \eqref{prop-variation-e11}, we further obtain, for sufficiently small $t$,
\begin{align*}
F({\bf0},t)
&=\|\nabla u_t\cdot({\bf0},b({\bf0},t))\|_p^p
=\|\nabla u_t\cdot({\bf0},\beta({\bf0},0))\|_p^p\\
&=2A_{\Phi}|\beta({\bf0},0)|^p
\int_{D}[\langle f_t\rangle(y)]^{1-p}\,dy \,.
\end{align*}
This, together with \eqref{eq:limfder},
implies that
\begin{align*}
0=\left.\frac{d}{dt}\right|_{t=0}\int_{D}[\langle f_t\rangle(y)]^{1-p}\,dy
=\int_D \frac{1-p}{2}[\langle f\rangle(y)]^{-p}
\langle g-f\rangle(y)\,dy \,.
\end{align*}
By this and $f(\cdot)=g(-\,\cdot)$, we further conclude
\begin{align*}
\int_D \left([\langle g\rangle(y)]^{-p}-[\langle f\rangle(y)]^{-p}\right)
[\langle g\rangle(y)-\langle f\rangle(y)]\,dy=0.
\end{align*}
Observe that $(a^{-p}-b^{-p})(a-b)\le 0$ for all $a,b\in(0,\infty)$.
Hence, it holds that
\[
\langle g\rangle(y)=\langle f\rangle(y)
\qquad\text{for all}\ y\in D \,.
\]
We have ${\bf 0}\in D$ (since ${\bf 0}\in K$) and $f({\bf0})-g({\bf0})=0$
(because $f(\cdot)=g(-\,\cdot)$). By Lemma \ref{linear}, these facts imply that $f-g$ is linear and hence that the chords of $K$ orthogonal
to $\mathbb{R}^{n-1}\times\{0\}$ have coplanar midpoints.
Since we can transfer $\mathbf{e}_n$ to arbitrary $\xi\in\mathbb{S}^{n-1}$,
the midpoints of all chords of $K$
parallel to any prescribed direction lie in a hyperplane transverse
to that direction. Consequently, applying the classical Bertrand--Brunn
theorem (see \cite[p.\,86]{t96} or \cite[p.\,123]{lyz00}),
we find that $K$ is an origin-symmetric
ellipsoid. This completes the proof of Proposition \ref{prop-variation}.
\end{proof}

We now give the proof of Theorem \ref{thm-solu}.

\begin{proof}[Proof of Theorem \ref{thm-solu}]
Recall that $H_f$ is defined in \eqref{eq:defhf}. By \eqref{eq-Lap-Hf}, we are considering the positive solutions of
the critical anisotropic Laplace equation
\begin{align*}
-\Delta_p^{H_f}(f)=S_{\rm aff}^pf^{p^*-1}.
\end{align*}
Noting that $H_f$ is a norm, we have $H_f(\zeta)\lesssim|\zeta|$ for all $\zeta\in\mathbb{R}^n$.
In addition, by the definition of $H_f$, \eqref{eq-holder},
and $f\in\dot{W}^{1,p}$,
we obtain
\begin{align*}
\int_{\mathbb{R}^n}H_f(\nabla f(x))^p\,dx
&=\mathcal{E}_p(f)^{n+p}
\int_{\mathbb{R}^n}\int_{\mathbb{S}^{n-1}}
\|\nabla_\xi f\|_{p}^{-n-p}|\nabla_\xi f(x)|^p\,d\xi\,dx\\
&=\mathcal{E}_p(f)^{n+p}\int_{\mathbb{S}^{n-1}}
\|\nabla_\xi f\|_{p}^{-n}\,d\xi=\mathcal{E}_p(f)^p
\lesssim\|\nabla f\|_{p}^p<\infty.
\end{align*}
Thus, $f$ has a finite $H_f$ anisotropic energy.
Furthermore, since $f$ is positive,
using this and Lemma \ref{lem-solu}, we find that
the Hessian $\nabla^2 (\frac12H_f^2)$ satisfies
condition \cite[(1.11)]{cfr20}. Consequently, by \cite[Theorem 1.1]{cfr20},
there are $c_{n,p},\lambda_0\in(0,\infty)$ and $x_0\in\mathbb{R}^n$ such that,
for all $x\in\mathbb{R}^n$,
\begin{align*}
f(x)=c_{n,p}\lambda_0^{-\frac{n}{p^*}}
\left[1+H_f^*\left(\lambda_0^{-1}(x-x_0)\right)^{p'}\right]^{-\frac{n-p}{p}},
\end{align*}
where $H^*_f(x):=\sup_{\zeta\ne{\bf 0}}\frac{|\zeta\cdot x|}{H_f(\zeta)}$ denotes the dual norm of $H_f$. Rescaling and translating $f$, we may assume that $\lambda_0=1$ and $x_0=0$. Thus, $f=\Phi(\|\cdot\|_K)$ with $K:=\{x\in\mathbb{R}^n:H^*_f(x)\le 1\}$ and $\Phi(r):=(1+t^{p'})^{-\frac{n-p}{p}}$.

Our goal is to apply Proposition
\ref{prop-variation}. We note that $K$ is an origin-symmetric convex body of class $C_+^2$
and ${\bf 0}\in\operatorname{Int}K$ (see, for instance,
\cite[p.\,123]{lyz00}), and that
the function $\Phi$
satisfies \eqref{eq:profile-assumption}.
Moreover, since $f$ is positive and satisfies
$-\Delta_p^{{\rm aff}}(f)=S_{\rm aff}^p f^{p^*-1}$, it is a critical point of $J_\sigma$ with $\sigma(t):=(S_{\rm aff}^p / p^*) |t|^{p^*}$. Thus, the assumptions of Propositions \ref{prop-variation} are satisfied and we infer that $K$ is an origin-symmetric ellipsoid.
Then there is a $(c,\lambda,S,x)\in\mathscr{A}$ with $\lambda\in(0,\infty)$ such that
$f=cT_{\lambda,S,x}U$.
Moreover, apply \eqref{eq:standardminimizernorm} and integration by parts,
we obtain
\begin{align*}
c^pS_{\rm aff}^p
=\mathcal{E}_p(f)^p=\left\langle-\Delta_p^{\rm aff}(f),f\right\rangle
=S_{\rm aff}^p\|f\|_{p^*}^{p^*}=c^{p^*}S_{\rm aff}^p,
\end{align*}
which further implies $c=1$.
Thus, $f=T_{\lambda,S,x}U$ and we complete the present proof.
\end{proof}


\section{Relative compactness}\label{sec-compact}

The two main results of this section give the relative compactness of minimizing sequences of the affine Sobolev inequality and of its critical points.
For $q\in[1,\infty)$, $\lambda\in(0,\infty)$,
$A\in \operatorname{SL}(n)$, and
$x\in {\mathbb R}^n$, define the \emph{general affine transformation}
$\mathcal{T}^{(q)}_{\lambda,A,x}$ by setting,
for every $f\in L^1_{\rm loc}$,
\begin{align}\label{def-tran-g}
{\mathcal T}_{\lambda,A,x}^{(q)} f(\cdot)
:=
\lambda^{-\frac nq}f\left(\lambda^{-1}A^{-1}(\cdot-x)\right).
\end{align}
Note that ${\mathcal T}_{\lambda,A,x}^{(q)}$ is an isometry on $L^q$
and, for every $\lambda\in(0,\infty)$,
every symmetric matrix $S$ satisfying $\operatorname{Tr}{S}=0$, and
every $x\in\mathbb{R}^n$,
\begin{align*}
{\mathcal T}_{\lambda, e^S,x}^{(p^*)}
=
T_{\lambda,S,x},
\end{align*}
where $T_{\lambda,S,x}$ was introduced in \eqref{eq:deftlambdasx}.

Unlike the affine transformations used in the parametrization of the extremal manifold in
\eqref{def-min-equiv}, for a general function one cannot use the polar decomposition to remove
the rotational part of a volume-preserving linear map. Indeed,
such a reduction is possible for the function $U$ only because it is radial.
Therefore, in the compactness argument below, we need to use the full affine action.

\begin{theorem}\label{thm-comp}
Let $p\in[2,n)$. If a sequence
$\{u_k\}_{k\in\mathbb{N}}$ in $\dot{W}^{1,p}$ satisfies
\begin{align*}
    \mathcal{E}_p(u_k)\to S_{\rm aff}
    \qquad\text{and}\qquad
    \|u_k\|_{p^*}\to1
    \qquad\text{as}\ k\to\infty \,,
\end{align*}
then there are $u\in\mathcal{M}_{\rm aff}$, $\{\lambda_k\}_{k\in\mathbb{N}}$
in $(0,\infty)$,
$\{x_k\}_{k\in\mathbb{N}}$
in $\mathbb{R}^n$, and $\{A_k\}_{k\in\mathbb{N}}$ in
$\operatorname{SL}(n)$ such that
\begin{align*}
    \mathcal{T}^{(p^*)}_{\lambda_k,A_k,x_k}u_k \to u
    \qquad\text{in}\ \dot W^{1,p} \ \text{as}\ k\to\infty \,.
\end{align*}
\end{theorem}

With some additional work, one can probably prove this theorem for $p\in(1,2)$ as well, but the case $p\geq 2$ is all we need for Theorem \ref{thm-stab} and it allows for a simple argument at the end of the proof.

As we have already mentioned, the analogue of Theorem \ref{thm-comp} in the case of the classical Sobolev inequality is due to Lions \cite{l85}. Since then, there have been alternative proofs of this result; see, for example, \cite{fl22,dt25}. Yet another approach proceeds by profile decomposition, and it is this approach that we generalize to the affine case.

As an input, we need the profile decomposition in $\dot{W}^{1,p}$, which was established by Jaffard in \cite{j99}.
We first recall some necessary concepts.
Let $p\in(0,\infty)$.
For any $\mathfrak{g}:=(\lambda,x_0)\in(0,\infty)\times\mathbb{R}^n$,
define an operator $\mathfrak{g}$ on $L^1_{\operatorname{loc}}$ by setting,
for all $u\in L^1_{\operatorname{loc}}$,
\begin{align*}
\mathfrak{g}(u;p):=\lambda^{-\frac np}u(\lambda^{-1}[\cdot-x_0]);
\end{align*}
then the inverse mapping $\mathfrak{g}^{-1}$ is given by
\begin{align*}
\mathfrak{g}^{-1}(u;p)=\lambda^{\frac np}u(\lambda\cdot+x_0).
\end{align*}
A sequence $\{\mathfrak{g}_k\}_{k\in\mathbb{N}}=\{(\lambda_k,x_k)\}_{k\in\mathbb{N}}$
is said to \emph{tend to infinity} if
$|\log \lambda_k|+|x_k|\to\infty$ as $k\to\infty$.
Two sequences $\{\mathfrak{g}_{1,k}\}_{k\in\mathbb{N}}$
and $\{\mathfrak{g}_{2,k}\}_{k\in\mathbb{N}}$ taking the above form
are said to be \emph{orthogonal} if $\mathfrak{g}_{1,k}^{-1}\mathfrak{g}_{2,k}$
tends to infinity.
Then we have the following profile decomposition (see \cite{j99}).

\begin{lemma}\label{lem-pdecom}
Let $p\in(1,n)$. If $\{w_k\}_{k\in\mathbb{N}}$
is a bounded sequence in $\dot{W}^{1,p}$, then
there are $\{V_j\}_{j\in\mathbb{N}}$
in $\dot{W}^{1,p}$, $\{g_{j,k}\}_{j,k\in\mathbb{N}}
=\{(\lambda_{j,k},x_{j,k})\}_{j,k\in\mathbb{N}}$
in $(0,\infty)\times\mathbb{R}^n$, and $\{r_{k}^J\}_{J,k\in\mathbb{N}}$
in $\dot W^{1,p}$ such that
\begin{enumerate}[{\rm(i)}]
  \item\label{it1-pdecom} for all $J,k\in\mathbb{N}$,
  \begin{align*}
  w_k=\sum_{j=1}^{J}\mathfrak{g}_{j,k}(V_j;p^*)
  +r_k^{J};
  \end{align*}
  \item\label{it2-pdecom} for any $j_1,j_2\in\mathbb{N}$ with $j_1\ne j_2$,
  $\{\mathfrak{g}_{j_1,k}\}_{k\in\mathbb{N}}$
  and $\{\mathfrak{g}_{j_2,k}\}_{k\in\mathbb{N}}$ are orthogonal;
  \item\label{it3-pdecom} for any $J\in\mathbb{N}$ and $j\in\{1,\ldots,J\}$,
  $\mathfrak{g}_{j,k}^{-1}(r_k^J;p^*)\rightharpoonup0$ in $\dot{W}^{1,p}$
  as $k\to\infty$, and
  \begin{align*}
  \lim_{J\to\infty}\limsup_{k\to\infty}
  \|r_k^J\|_{p^*}=0.
  \end{align*}
  Here and in what follows, the notation $\rightharpoonup$ denotes \emph{weak convergence}.
\end{enumerate}
\end{lemma}

We also need some technical lemmas
that are presented in Appendices \ref{App-B} and \ref{App-C}.

\begin{proof}[Proof of Theorem \ref{thm-comp}]
By the assumption on $\{u_k\}_{k\in\mathbb{N}}$,
we find that $\{\mathcal{E}_p(u_k)\}_{k\in\mathbb{N}}$
is bounded. Consequently, by \eqref{eq-p}, for each $k\in\mathbb{N}$ there is a
$T_k\in\operatorname{SL}(n)$ such that $\|\nabla(u_k\circ T_k)\|_p\lesssim
\mathcal{E}_p(u_k)\lesssim1$.
This shows that $\{w_k\}_{k\in\mathbb{N}}:=\{u_k\circ T_k\}_{k\in\mathbb{N}}$
is a bounded sequence in $\dot{W}^{1,p}$.
Then, applying the profile decomposition, the conclusion of Lemma \ref{lem-pdecom}
holds for $\{w_k\}_{k\in\mathbb{N}}$ here. We use the same notation as
in Lemma \ref{lem-pdecom}.
By \eqref{eq-pdecom},
we obtain
\begin{align}\label{thm-comp-e1}
1=\lim_{k\to\infty}\|u_k\|_{p^*}^{p^*}=\lim_{k\to\infty}\|w_k\|_{p^*}^{p^*}
=\sum_{j=1}^{\infty}\|V_j\|_{p^*}^{p^*}.
\end{align}

Now, we turn to deal with the directional derivative terms.
For any $J,k\in\mathbb{N}$, define
$G_k^J:=\sum_{j=1}^{J}\mathfrak{g}_{j,k}(V_j;p^*)$.
Then, from Lemma \ref{lem-pdecom}\eqref{it1-pdecom} and
the convexity of the function $t\mapsto t^p$, it follows that,
for any $\xi\in\mathbb{S}^{n-1}$,
\begin{align}\label{thm-comp-e2}
\|\nabla_\xi w_k\|_p^p
\ge\|\nabla_\xi G_k^J\|_p^p+
p\int_{\mathbb{R}^n}|\nabla_\xi G_k^J|^{p-2}\nabla_\xi G_k^J
\nabla_\xi r_k^J.
\end{align}
Observe that, since $n(\frac1p - \frac1{p^*}) =1$,
\begin{align}\label{thm-comp-e4}
\nabla_\xi\mathfrak{g}_{j,k}(V_j;p^*)
&=\lambda_{j,k}^{-\frac np}
\nabla_\xi V_j\left(\frac{\cdot-x_{j,k}}{\lambda_{j,k}}\right)
=\mathfrak{g}_{j,k}(\nabla_\xi V_j;p).
\end{align}
Thus, by Lemma \ref{lem-pdecom}\eqref{it2-pdecom} and
\eqref{lem-decou-ee}, we find that, when $k\to\infty$,
$$
\left\|\sum_{j=1}^{J}|\nabla_\xi\mathfrak{g}_{j,k}(V_j;p^*)|^{p-2}
\nabla_\xi\mathfrak{g}_{j,k}(V_j;p^*)-|\nabla_\xi G_k^J|^{p-2}\nabla_\xi G_k^J\right\|_{p'}
\to0.
$$
This, combined with Lemma \ref{lem-pdecom}\eqref{it3-pdecom}, further implies that
$\|\nabla r^J_k\|_{p}$ is uniformly bounded in terms of $k$ and hence,
when $k\to\infty$,
\begin{align}\label{thm-comp-e3}
\int_{\mathbb{R}^n}|\nabla_\xi G_k^J|^{p-2}\nabla_\xi G_k^J
\nabla_\xi r_k^J
=\sum_{j=1}^{J}\int_{\mathbb{R}^n}|\nabla_\xi\mathfrak{g}_{j,k}(V_j;p^*)|^{p-2}
\nabla_\xi\mathfrak{g}_{j,k}(V_j;p^*)\nabla_\xi r_k^J+o(1).
\end{align}
From \eqref{thm-comp-e4} and a change of variables,
it follows that
\begin{align}\label{thm-comp-e5}
\int_{\mathbb{R}^n}|\nabla_\xi\mathfrak{g}_{j,k}(V_j;p^*)|^{p-2}
\nabla_\xi\mathfrak{g}_{j,k}(V_j;p^*)\nabla_\xi r_k^J
&=\int_{\mathbb{R}^n}|\mathfrak{g}_{j,k}(\nabla_\xi V_j;p)|^{p-2}
\mathfrak{g}_{j,k}(\nabla_\xi V_j;p)\nabla_\xi r_k^J\notag\\
&=\int_{\mathbb{R}^n}|\nabla_\xi V_j|^{p-2}
\nabla_\xi V_j\mathfrak{g}_{j,k}^{-1}(\nabla_\xi r_k^J;p)\notag\\
&=\int_{\mathbb{R}^n}|\nabla_\xi V_j|^{p-2}
\nabla_\xi V_j\nabla_\xi\mathfrak{g}_{j,k}^{-1}(r_k^J;p^*).
\end{align}
Note that
$\Phi_{j,\xi}:u\mapsto
\int_{\mathbb{R}^n}|\nabla_\xi V_j|^{p-2}\nabla_\xi V_j\nabla_\xi u$
is a bounded linear functional on $\dot{W}^{1,p}$;
indeed, using H\"older's inequality,
we find that, for any
$u\in\dot{W}^{1,p}$,
\begin{align}\label{thm-comp-e8}
|\Phi_{j,\xi}(u)|
\le \int_{\mathbb{R}^n}|\nabla V_j|^{p-1}|\nabla u|
\le\left(\int_{\mathbb{R}^n}|\nabla V_j|^p\right)^{1-\frac1p}
\left(\int_{\mathbb{R}^n}|\nabla u|^p\right)^{\frac1p}.
\end{align}
Consequently, Lemma \ref{lem-pdecom}\eqref{it3-pdecom},
combined with \eqref{thm-comp-e3} and \eqref{thm-comp-e5},
further implies that, for all $\xi\in\mathbb{S}^{n-1}$ and $J\in\mathbb{N}$,
\begin{align}\label{thm-comp-e6}
\lim_{k\to\infty}\int_{\mathbb{R}^n}|\nabla_\xi G_k^J|^{p-2}\nabla_\xi G_k^J
\nabla_\xi r_k^J=0.
\end{align}
On the other hand, by \eqref{lem-decou-e0} and \eqref{thm-comp-e4}, we conclude that,
when $k\to\infty$,
\begin{align*}
\|\nabla_\xi G_k^J\|_{p}^p
=\left\|\sum_{j=1}^{J}\mathfrak{g}_{j,k}(\nabla_\xi V_j;p)\right\|_p^p
=\sum_{j=1}^{J}\|\nabla_\xi V_j\|_p^p+o(1).
\end{align*}
From this, \eqref{thm-comp-e2}, and \eqref{thm-comp-e6}, we further infer that,
for all $J\in\mathbb{N}$,
when $k\to\infty$,
\begin{align}\label{thm-comp-e7}
\|\nabla_\xi w_k\|_p^p
\ge\sum_{j=1}^{J}\|\nabla_\xi V_j\|_p^p+o(1).
\end{align}

We need the uniformity with respect to $\xi\in\mathbb S^{n-1}$ of this asymptotic lower bound. Note that the remainder term $o(1)$ in \eqref{thm-comp-e7} is
\begin{align*}
\rho_k^J(\xi):=\|\nabla_\xi G_k^J\|_{p}^p
-\sum_{j=1}^{J}\|\nabla_\xi V_j\|_p^p+p
\int_{\mathbb{R}^n}|\nabla_\xi G_k^J|^{p-2}\nabla_\xi G_k^J
\nabla_\xi r_k^J.
\end{align*}
It is easy to verify that $\rho_k^J \in C(\mathbb{S}^{n-1})$.
We have proved that $\rho_k^J\to0$ pointwise as $k\to\infty$.
But, if we aim to further apply the
nonlinear Brezis--Lieb lemma given in Lemma \ref{lem-BL}\eqref{it2-BL}, we still need the
uniformity in $\xi\in\mathbb S^{n-1}$ of this convergence.

Observe that $\mathfrak{g}_{j,k}(\cdot;p^*)$ preserves
$\|\cdot\|_{\dot{W}^{1,p}}$, we find that
$\{G_{k}^J\}_{k\in\mathbb{N}}$ is a bounded sequence
in $\dot{W}^{1,p}$.
Hence, applying H\"older's inequality similar to
\eqref{thm-comp-e8} and applying the uniform boundedness
of $\|\nabla r^J_k\|_{p}$ in terms of $k$,
we can easily conclude that the sequence $\{\rho_k^J\}_{k\in\mathbb{N}}$
is equicontinuous and bounded in $C(\mathbb{S}^{n-1})$.
Then, from the Arzel\`a--Ascoli theorem,
it follows that, for any subsequence $\{\rho_{k_j}\}_{j\in\mathbb{N}}$,
there are a further subsequence $\{\rho^J_{k_{j_\ell}}\}_{\ell\in\mathbb{N}}$
and a function $\rho^J\in C(\mathbb{S}^{n-1})$
such that $\rho^J_{k_{j_\ell}}\to\rho^J$ in $C(\mathbb{S}^{n-1})$ as $\ell\to\infty$.
Noting that $\rho_k^J\to0$ pointwise as $k\to\infty$, we have
$\rho^J\equiv0$.
Then, by the arbitrariness of the above subsequence,
we find that the whole sequence $\{\rho_k^J\}_{k\in\mathbb{N}}$
converges to $0$ in $C(\mathbb{S}^{n-1})$
(that is, converges to $0$ uniformly). This proves the claimed uniformity.

Meanwhile, from \eqref{thm-comp-e1},
it follows that there is a
$J_0\in\mathbb{N}$ such that $V_{J_0}\ne 0$.
Then $\xi\mapsto\|\nabla_\xi V_{J_0}\|_p$
is a norm on $\mathbb{R}^n$ and hence
there is a $C_0\in(0,\infty)$ such that
$\|\nabla_\xi V_{J_0}\|_p\ge C_0$ for all
$\xi\in\mathbb{S}^{n-1}$
(see, for instance, \cite[(3.1)]{n16}).
Consequently,
by the fact that $\{\rho_k^J\}_{k\in\mathbb{N}}$
converges to $0$ uniformly, we conclude that,
for any $\xi\in\mathbb{S}^{n-1}$,
any $J>J_0$, and any sufficiently large $k$,
\begin{align*}
\sum_{j=1}^{J}\|\nabla_\xi V_j\|_p^p+\rho_k^J(\xi)\ge
\frac{C_0^p}{2}.
\end{align*}
This, together with Lemma \ref{lem-BL}\eqref{it2-BL} and \eqref{thm-comp-e7},
further implies that, for all $J>J_0$, when $k\to\infty$,
\begin{align*}
\mathcal{E}_p(w_k)^p
&\ge\left(\int_{\mathbb{S}^{n-1}}\left[\sum_{j=1}^{J}
\|\nabla_\xi V_j\|_p^p+\rho_k^J(\xi)\right]^{-\frac np}\,d\xi\right)^{-\frac pn}\\
&=\left(\int_{\mathbb{S}^{n-1}}\left[\sum_{j=1}^{J}
\|\nabla_\xi V_j\|_p^p\right]^{-\frac np}\,d\xi\right)^{-\frac pn}
+o(1)\\
&\ge\sum_{j=1}^{J}\mathcal{E}_p(V_j)^p+o(1)
\end{align*}
by the reverse Minkowski inequality (Lemma \ref{lem:minkowskireverse}).

Applying this, the assumption on $w_k$, and the arbitrariness of $J$, we then obtain
\begin{align*}
S_{\rm aff}^p
&=\lim_{k\to\infty}\mathcal{E}_p(w_k)^p
\ge\sum_{j=1}^{\infty}\mathcal{E}_p(V_j)^p\\
&\ge S_{\rm aff}^p\sum_{j=1}^{\infty}\|V_j\|_{p^{*}}^{p}
\quad\text{by \eqref{eq-affSob}}\\
&\ge S_{\rm aff}^p
\left(\sum_{j=1}^{\infty}\|V_j\|_{p^*}^{p^*}\right)^{\frac{p}{p^*}}
\quad\text{by the concavity of $t\mapsto t^{\frac{p}{p^*}}$}\\
&=S_{\rm aff}^p\quad\text{by \eqref{thm-comp-e1}}.
\end{align*}
This shows that all the above inequalities should be equalities.
Therefore, due to the strict concavity of $t\mapsto t^{\frac{p}{p^*}}$,
there is a unique non-zero function in $\{V_j\}_{j\in\mathbb{N}}$,
which we denote by $V$.
Then $\mathcal{E}_p(V)=S_{\rm aff}$ and $\|V\|_{p^*}=1$.
This implies $V\in\mathcal{M}_{\rm aff}$.
Moreover, by Lemma \ref{lem-pdecom}\eqref{it1-pdecom},
there are a sequence $\{\mathfrak{g}_k\}_{k\in\mathbb{N}}
=\{(\lambda_k,x_k)\}_{k\in\mathbb{N}}$
in $(0,\infty)\times\mathbb{R}^n$ and a sequence
$\{r_k\}_{k\in\mathbb{N}}$ in $L^{p^*}$ such that,
for every $k\in\mathbb{N}$,
\begin{align*}
w_k=\mathfrak{g}_k(V;p^*)+r_k\quad
\text{and}\quad\mathfrak{g}_k^{-1}(r_k;p^*)\rightharpoonup0\
\text{in $\dot{W}^{1,p}$ as $k\to\infty$}.
\end{align*}
Hence,
\begin{align*}
\mathfrak{g}_k^{-1}(w_k;p^*)=V+\mathfrak{g}_k^{-1}(r_k;p^*).
\end{align*}

Since $\mathfrak{g}_k^{-1}(w_k;p^*)$ takes the form
$\mathcal{T}^{(p^*)}_{\lambda_k,A_k,x_k}u_k$ with some $A_k\in \operatorname{SL}(n)$, we infer that,
if we can show $$\mathfrak{g}_k^{-1}(r_k;p^*)\to0\
\text{in $\dot{W}^{1,p}$ as $k\to\infty$,}$$
then we complete the present proof.
Furthermore, by \eqref{def-min}, the fact that
the mapping $\mathcal{T}^{(p^*)}_{\lambda_k,A_k,x_k}$ preserves both $\mathcal{E}_p(\cdot)$
and $\|\cdot\|_{p^*}$, and the weak convergence,
we can assume that $V=U$. That is,
we only need to prove that, if a sequence $\{\rho\}_{k\in\mathbb{N}}$
in $\dot{W}^{1,p}$ satisfies,
as $k\to\infty$,
$\mathcal{E}_{p}(U+\rho_k)\to S_{\rm aff}$ and
$\rho_k\rightharpoonup0$
in $\dot{W}^{1,p}$, then
$\rho_k\to0$ in $\dot{W}^{1,p}$.

We now show this assertion. Indeed,
applying $p\ge 2$ and Taylor's formula
(see also \cite[Lemma 2.1(ii)]{fz22}),
there is a positive constant $c_p$, depending only on
$p$, such that, for any $x,y\in\mathbb{R}$,
$$|x+y|^p\ge |x|^p+p|x|^{p-2}xy+c_p|y|^p.$$ (This is the point where the assumption $p\geq 2$ enters our proof.)
Thus, for all $\xi\in\mathbb{S}^{n-1}$ and $k\in\mathbb{N}$,
\begin{align*}
\|\nabla_\xi (U+\rho_k)\|_p^p
\ge \|\nabla_\xi U\|_p^p
+p\int_{\mathbb{R}^n}|\nabla_\xi U|^{p-2}
\nabla_\xi U\nabla_\xi \rho_k
+c_p\|\nabla_\xi \rho_k\|_p^p.
\end{align*}
Similar to the analysis of $\{\rho_k^J\}_{k\in\mathbb{N}}$
above (and even easier), we deduce that, as $k\to\infty$,
\begin{align*}
\mathcal{E}_p(U+\rho_k)^p
&\ge\left[\int_{\mathbb{S}^{n-1}}
\left( \|\nabla_\xi U\|_p^p
+c_p\|\nabla_\xi \rho_k\|_p^p\right)^{-\frac np}\,d\xi\right]^{-\frac pn}
+o(1)\\
&\ge S_{\rm aff}^p+o(1)\quad\text{by $\mathcal{E}_p(U)=S_{\rm aff}$}.
\end{align*}
This, combined with the assumption
$\mathcal{E}_p(U+\rho_k)\to S_{\rm aff}$,
further implies that
\begin{align}\label{thm-comp-e10}
\lim_{k\to\infty}\left[\int_{\mathbb{S}^{n-1}}
\left( \|\nabla_\xi U\|_p^p
+c_p\|\nabla_\xi \rho_k\|_p^p\right)^{-\frac np}\,d\xi\right]^{-\frac pn}
=S_{\rm aff}^p.
\end{align}
Since $U$ is radial, we have $\|\nabla_\xi U\|_p^p=a$
with some absolute positive constant $a$ independent of $\xi$.
Then, using the boundedness of $\{\rho_k\}_{k\in\mathbb{N}}$
in $\dot{W}^{1,p}$ and Taylor's formula, we find that,
for any $k\in\mathbb{N}$ and $\xi\in\mathbb{S}^{n-1}$, there is a $t\in(0,1)$
such that
\begin{align}\label{thm-comp-e9}
a^{-\frac np}-
\left(a+c_p\|\nabla_\xi \rho_k\|_p^p\right)^{-\frac np}
=\frac npc_p\left(a+c_pt\|\nabla_\xi \rho_k\|_p^p\right)^{-1-\frac np}
\|\nabla_\xi\rho_k\|_p^p
\gtrsim\|\nabla_\xi\rho_k\|_p^p,
\end{align}
where the implicit positive constant depends only on $n$ and $p$.
In addition, observe that
$S_{\rm aff}=\mathcal{E}_p(U)=a^{\frac1p}|\mathbb{S}^{n-1}|^{-\frac1n}$.
From this and \eqref{thm-comp-e9}, it further follows that,
when $k\to\infty$,
\begin{align*}
\|\nabla\rho_k\|_{p}^p
&\sim\int_{\mathbb{S}^{n-1}}\|\nabla_\xi\rho_k\|_{p}^p\,d\xi
\lesssim\int_{\mathbb{S}^{n-1}}
\left( a^{-\frac np}-\left(a+c_p\|\nabla_\xi \rho_k\|_p^p\right)^{-\frac np}\right) d\xi\\
&=S_{\rm aff}^{-n}-\int_{\mathbb{S}^{n-1}}
\left( \|\nabla_\xi U\|_p^p
+c_p\|\nabla_\xi \rho_k\|_p^p\right)^{-\frac np}\,d\xi\to0
\quad\text{by \eqref{thm-comp-e10}.}
\end{align*}
This shows $\lim_{k\to\infty}\rho_k=0$ in $\dot{W}^{1,p}$
and hence completes the present proof.
\end{proof}

Moreover, this
proof gives the following consequence, which may be of
independent interest.

\begin{corollary}
Let $p\in[2,n)$, and let $U$ be as in \eqref{eq:standardminimizer} and
\eqref{eq:standardminimizernorm}. If a sequence $\{\rho\}_{k\in\mathbb{N}}$
in $\dot{W}^{1,p}$ satisfies that,
as $k\to\infty$,
$$\mathcal{E}_{p}(U+\rho_k)\to S_{\rm aff}\quad\text{and}
\quad\rho_k\rightharpoonup0\
\text{in $\dot{W}^{1,p}$},
$$
then
$\rho_k\to0$ in $\dot{W}^{1,p}$.
\end{corollary}

If, in this statement, the convergence of the affine energy is replaced by
$$
\|\nabla(U+\rho_k)\|_p\to S_{\rm Sob},
$$
then the conclusion follows immediately from the Radon--Riesz property
(or equivalently, the uniform convexity) of $L^p$. The corollary
above can therefore be viewed as an affine analogue of this fact.

The following theorem gives the affine Struwe compactness result
for single bubble. We recall that the affine Euler--Lagrange residual $\delta_{\rm aff}(\cdot)$ is defined in \eqref{eq:elresidual}.

\begin{theorem}\label{thm-critical}
Let $p\in[2,n)$, and let
$\{u_k\}_{k\in\mathbb{N}}$ in $\dot{W}^{1,p}$ satisfy
\begin{align}\label{thm-critical-e1}
\frac12 \, S_{\rm aff}^p\le\mathcal{E}_p(u_k)^p\le\frac32 \, S_{\rm aff}^p
\quad\text{for every}\ k\in\mathbb N
\qquad\text{and}\qquad
\lim_{k\to\infty}\delta_{\rm aff}(u_k)=0.
\end{align}
Then, possibly passing to a subsequence,
there are $u\in\mathcal{M}_{\rm aff}^{(1)}$, $\{\lambda_k\}_{k\in\mathbb{N}}$
in $(0,\infty)$,
$\{x_k\}_{k\in\mathbb{N}}$
in $\mathbb{R}^n$, and $\{A_k\}_{k\in\mathbb{N}}$ in
$\operatorname{SL}(n)$ such that
$\mathcal{T}^{(p^*)}_{\lambda_k,A_k,x_k}u_k$
converges to $u$ in $\dot{W}^{1,p}$ as $k\to\infty$.
\end{theorem}

\begin{proof}
Similar to the proof of Theorem \ref{thm-comp},
by \eqref{eq-p} and Lemma \ref{lem-pdecom},
we find that, for each $k\in\mathbb{N}$, there is
$T_k\in\operatorname{SL}(n)$ such that $\|\nabla(u_k\circ T_k)\|_p\lesssim
\mathcal{E}_p(u_k)$
and the profile decomposition holds for the
sequence $\{w_k\}_{k\in\mathbb{N}}:=\{u_k\circ T_k\}_{k\in\mathbb{N}}$.
In the remainder of this proof, let all the notation
be the same as in Lemma \ref{lem-pdecom}.

In the classical critical $p$-Laplacian argument, once the profile
decomposition has been obtained, condition \eqref{thm-critical-e1}
and the monotonicity of the fixed nonlinear map
$\xi\mapsto |\xi|^{p-2}\xi$
are used to prove that every nonzero $V_j$ in the profile decomposition solves the limiting critical
equation; see \cite{mw10}. However, as pointed out in \eqref{eq-Lap-Hf},
the affine setting contains an additional
sequence-dependent geometric variable.
Consequently, we need to
first extract a limiting directional energy profile.
After this, the rescaled sequence can be treated
as an approximate sequence for a fixed anisotropic functional.

More precisely,
for any $k\in\mathbb{N}$
and $\xi\in\mathbb{S}^{n-1}$,
define $h_k(\xi):=\|\nabla_\xi w_k\|_p$;
then, from \eqref{thm-critical-e1}, it follows that, for every $\xi_1,\xi_2\in\mathbb{S}^{n-1}$,
\begin{align}\label{thm-critical-e2}
|h_k(\xi)|\le\|\nabla w_k\|_p\lesssim S_{\rm aff}
\quad\text{and}\quad|h_k(\xi_1)-h_k(\xi_2)|\le\|\nabla w_k\|_p
|\xi_1-\xi_2|\lesssim S_{\rm aff}|\xi_1-\xi_2|.
\end{align}
Therefore, applying the Arzel\`a--Ascoli theorem, there are $h_\infty\in C(\mathbb{S}^{n-1})$
and a subsequence, still denoted by $\{h_k\}_{k\in\mathbb{N}}$, such that
$h_k\to h_\infty$ in $C(\mathbb{S}^{n-1})$ as $k\to\infty$.
Moreover, for fixed $k\in\mathbb{N}$, define $c_k:=\min\{h_k(\xi):\xi\in\mathbb{S}^{n-1}\}$ and choose $\xi_k\in\mathbb S^{n-1}$ such that $h_k(\xi_k)=c_k$. Then, by \eqref{thm-critical-e2}, we obtain,
for any $\eta\in\mathbb{S}^{n-1}$ such that $|\eta-\xi_k|< \epsilon c_k$,
\begin{align*}
|h_k(\eta)-h_k(\xi_k)|\lesssim|\eta-\xi_k|< \epsilon c_k.
\end{align*}
It follows that we can choose $\epsilon>0$ so small, depending only on $n$ and $p$,
such that $h_k(\eta)\sim c_k$ for $|\eta-\xi_k|<\epsilon c_k$. This, together with \eqref{thm-critical-e1},
further implies that
\begin{align*}
1\gtrsim\mathcal{E}_p(u_k)^{-n}
=\mathcal{E}_k(w_k)^{-n}
=\int_{\mathbb{S}^{n-1}}h_k(\eta)^{-n}\,d\eta
\gtrsim\int_{\{\eta\in\mathbb{S}^{n-1}:|\eta-\xi_k|< \epsilon c_k\}}
c_k^{-n}\,d\eta\sim c_k^{-1}.
\end{align*}
Consequently, $c_k\gtrsim1$. Meanwhile, \eqref{thm-critical-e2} shows that $|h_k|\lesssim1$ for every
$k\in\mathbb{N}$.
Then, by the convergence of $\{h_k\}_{k\in\mathbb{N}}$,
there are $c,C\in(0,\infty)$ such that, for every $\xi\in\mathbb{S}^{n-1}$, $c\le h_\infty(\xi)\le C$.

Define
\begin{align*}
E_\infty:=\left[\int_{\mathbb{S}^{n-1}}h_\infty(\xi)^{-n}\,d\xi\right]^{-\frac1n}
\end{align*}
and
\begin{align*}
H_\infty(\zeta):=
E_{\infty}^{1+\frac np}\left[\int_{\mathbb{S}^{n-1}}h_{\infty}(\xi)^{-n-p}
|\zeta\cdot\xi|^p\,d\xi\right]^{\frac1p}\ \forall\,\zeta\in\mathbb{R}^n.
\end{align*}
As proved in Lemma \ref{lem-solu},
there are positive constants $\lambda$ and $\Lambda$ such that, for all $\zeta,\eta\in\mathbb{R}^n$,
\begin{align}\label{thm-critical-e3}
\lambda|\eta|^2\le
\left\langle\nabla^2\left(\frac{1}{2}H_\infty^2\right)(\zeta) \,\eta,\eta\right\rangle
\le \Lambda|\eta|^2.
\end{align}
Indeed, applying the fact that $h_\infty$ has positive upper and lower bounds on $\mathbb{S}^{n-1}$ and
extending $h_\infty$ to $\mathbb{R}^n$ as a homogeneous function of degree zero, we find that
$K_\infty:=\{\zeta\in\mathbb{R}^n:h_\infty(\zeta)\le 1\}$ is an origin-symmetric convex body
containing the origin in its interior; therefore, by the proof of Lemma \ref{lem-solu}, we find that
$H_\infty$ satisfies estimate
\eqref{thm-critical-e3}.

Next, we show that every nonzero $V_j$ in the profile decomposition
satisfies the critical equation related to
the anisotropic Laplacian $\Delta_p^{H_\infty}u:=\operatorname{div}(\nabla(\frac{H_\infty^p}
{p})(\nabla u))$. We first prove that nonzero $V_j$ exists.
Indeed, from integration by parts and
\eqref{thm-critical-e1}, we deduce that, when $k\to\infty$,
\begin{align*}
\left|\mathcal{E}_p(w_k)^p-S_{\rm aff}^p\|w_k\|_{p^*}^{p^*}\right|
&=\left|\langle \Delta_p^{\rm aff}(w_k)+ S_{\rm aff}^p|w_k|^{p^*-2}w_k,w_k\rangle\right|\\
&\le \delta_{\rm aff}(w_k)\|\nabla w_k\|_p
\lesssim\delta_{\rm aff}(w_k)\to0.
\end{align*}
Consequently, by \eqref{thm-comp-e1} and \eqref{thm-critical-e1}, we have
\begin{align}\label{thm-critical-e6}
\sum_{j=1}^\infty \|V_j\|_{p^*}^{p^*}
=\lim_{k\to\infty} \|w_k\|_{p^*}^{p^*}\in\left[\frac12,\frac32\right].
\end{align}
Then there is at least one $j\in\mathbb{N}$ such that $V_j\ne 0$.
Now, we fix a $j\in\mathbb{N}$ satisfying this property.
By a change of variables and H\"older's inequality, we obtain,
for every $k\in\mathbb{N}$ and $\varphi\in\dot{W}^{1,p}$,
\begin{align}\label{thm-critical-e4}
&\left|\left\langle\Delta_p^{\rm aff}\left(\mathfrak{g}_{j,k}^{-1}(w_k;p^*)\right)
-\Delta_p^{H_\infty}\left(\mathfrak{g}_{j,k}^{-1}(w_k;p^*)\right),\varphi\right\rangle\right|\notag\\
&\quad=\left|\left\langle\Delta_p^{\rm aff}(w_k)-\Delta_p^{H_\infty}(w_k),\mathfrak{g}_{j,k}(\varphi;p^*)
\right\rangle\right|\notag\\
&\quad=\left|\int_{\mathbb{R}^n}\int_{\mathbb{S}^{n-1}}
\left[\mathcal{E}_p(w_k)^{n+p}\|\nabla_\xi w_k\|_p^{-n-p}
-E_\infty^{n+p} h_\infty(\xi)^{-n-p}\right]\right.\notag\\
&\qquad\times
|\nabla_\xi w_k|^{p-2}\nabla_\xi w_k\nabla_\xi(\mathfrak{g}_{j,k}(\varphi;p^*))\Bigg|\notag\\
&\quad\lesssim
\sup_{\xi\in\mathbb{S}^{n-1}}\left|\mathcal{E}_p(w_k)^{n+p}\|\nabla_\xi w_k\|_p^{-n-p}
-E_\infty^{n+p} h_\infty(\xi)^{-n-p}\right|
\|\nabla w_k\|_p^{p-1}\|\nabla\varphi\|_p.
\end{align}
Note that the supremum over $\xi$ in \eqref{thm-critical-e4} tends to zero. Indeed, $h_k(\xi)=\|\nabla_\xi w_k\|_{p}$ tends to $h_\infty(\xi)$ as $k\to\infty$, uniformly for $\xi\in\mathbb S^{n-1}$, and it is bounded away from $0$ uniformly with respect to $\xi$ and $k$. In particular, $\mathcal{E}_p(w_k)^{n+p}\to E_\infty^{n+p}$
as $k\to\infty$. The fact that the supremum in \eqref{thm-critical-e4} tends to zero, together with the boundedness of $\{w_k\}_{k\in\mathbb{N}}$ in $\dot W^{1,p}$, implies that, as $k\to\infty$,
\begin{align*}
    \Delta_p^{\rm aff}\left(\mathfrak{g}_{j,k}^{-1}(w_k;p^*)\right)
-\Delta_p^{H_\infty}\left(\mathfrak{g}_{j,k}^{-1}(w_k;p^*)\right) \to 0
\qquad\text{in}\ \dot W^{-1,p'} \,.
\end{align*}
Combined with \eqref{thm-critical-e1}, we infer that, as $k\to\infty$,
\begin{align}\label{thm-critical-e8}
-\Delta_p^{H_\infty}\left(\mathfrak{g}_{j,k}^{-1}(w_k;p^*)\right)-S_{\rm aff}^p
\left(\mathfrak{g}_{j,k}^{-1}(w_k;p^*)\right)^{p^*-1}\to 0
\qquad\text{in}\ \dot W^{-1,p'} \,.
\end{align}
By the orthogonality and the weak convergence conditions
in the profile decomposition [namely \eqref{it2-pdecom} and \eqref{it3-pdecom}
of Lemma \ref{lem-pdecom}, respectively], we
find that, when $k\to\infty$,
$\mathfrak g_{j,k}^{-1}(w_k;p^*)
\rightharpoonup V_j$ in $\dot W^{1,p}$.
Then we are able to apply the anisotropic version of the localized truncation
argument used in the classical global compactness theory for the
$p$-Laplacian, as in Mercuri and Willem \cite{mw10}, recently established
in \cite[Lemma 3.6]{acg26}. Its key input is the quantitative monotonicity of
$\xi\mapsto
H_\infty(\xi)^{p-1}\nabla H_\infty(\xi),$
which follows from \eqref{thm-critical-e3}. Hence, possibly after passing
to a further subsequence, we obtain
\begin{align}\label{thm-critical-e7}
\nabla\left(\mathfrak g_{j,k}^{-1}(w_k;p^*)\right)
\longrightarrow\nabla V_j\quad \text{and}\quad
\mathfrak g_{j,k}^{-1}(w_k;p^*)
\longrightarrow V_j
\quad
\text{a.\,e. in }\mathbb R^n
\end{align}
and
\begin{align}\label{thm-critical-e5}
-\Delta_p^{H_\infty}(V_j)
=
S_{\rm aff}^p|V_j|^{p^*-2}V_j.
\end{align}
This completes the proof that $V_j$ satisfies
the anisotropic critical equation.

We now show that there is only one such nonzero $V_j$ and that it has a constant sign.
Indeed, if we define
$V_{j,+}:=\max\{0,V_j\}$ and $V_{j,-}:=\max\{0,-V_j\}$, then it holds that
\begin{align*}
\nabla V_{j,+}={\bf 1}_{\{V_j>0\}} \nabla V_j
\quad\text{and}\quad
\nabla V_{j,-}=-{\bf 1}_{\{V_j<0\}}\nabla V_j.
\end{align*}
By this, \eqref{thm-critical-e5}, H\"older's inequality, and the
affine Sobolev inequality \eqref{eq-affSob},
we obtain, for every $\sigma\in\{+,-\}$,
\begin{align*}
S_{\rm aff}^p\|V_{j,\sigma}\|_{p^*}^{p*}
&=\left|\left\langle-\Delta_p^{H_\infty}(V_j),V_{j,\sigma}\right\rangle\right|\\
&=E_\infty^{n+p}\left|
\int_{\mathbb{S}^{n-1}}\int_{\mathbb{R}^n}h_{\infty}(\xi)^{-n-p}
|\nabla_\xi V_j(x)|^{p-2}\nabla_\xi V_j(x)\nabla_\xi V_{j,\sigma}(x)\,dx\,d\xi\right|\\
&=E_\infty^{n+p}
\int_{\mathbb{S}^{n-1}}\int_{\mathbb{R}^n}h_{\infty}(\xi)^{-n-p}
|\nabla_\xi V_{j,\sigma}(x)|^{p}\,dx\,d\xi\\
&\ge E_\infty^{n+p}\left[\int_{\mathbb{S}^{n-1}}h_\infty(\xi)^{-n}\,d\xi\right]^{\frac{n+p}{n}}
\left(\int_{\mathbb{S}^{n-1}}\|\nabla_\xi V_{j,\sigma}\|_p^{-n}\,d\xi\right)^{-\frac{p}{n}}\\
&=\mathcal{E}_p(V_{j,\sigma})^p\ge S_{\rm aff}^p\|V_{j,\sigma}\|_{p^*}^{p}.
\end{align*}
This, combined with $p^*>p$, further implies
$\|V_{j,\sigma}\|_{p^*}\ge 1$ for all nonzero $V_{j,\sigma}$ with $\sigma\in\{+,-\}$.
Hence, from \eqref{thm-critical-e6}, it follows that there is only one such nonzero
$V_{j}$ and it has a constant sign.
By replacing $\{w_k\}_{k\in\mathbb{N}}$ with $\{-w_k\}_{k\in\mathbb{N}}$ if necessary, we may assume that
$V_j\ge 0$. Thus, applying
\eqref{thm-critical-e5} and the strong maximum principle (see,
for instance, \cite[Theorem 2.5.1]{ps07}), we further conclude
$V_j>0$.

Meanwhile, from \eqref{thm-critical-e7} and
the anisotropic Brezis--Lieb lemma
(see \cite[(3.23)]{acg26}), it follows that, when $k\to\infty$,
\begin{align*}
\int_{\mathbb{R}^n}H_\infty\left(\nabla\left(\mathfrak g_{j,k}^{-1}(w_k;p^*)\right)\right)^p
=
\int_{\mathbb{R}^n}H_\infty(\nabla V_j)^p+\int_{\mathbb{R}^n}
H_\infty\left(\nabla\left(\mathfrak g_{j,k}^{-1}(w_k;p^*)-V_j\right)\right)^p+o(1).
\end{align*}
Moreover, by \eqref{thm-critical-e8}, we find that, when $k\to\infty$,
\begin{align*}
\int_{\mathbb{R}^n}H_\infty
\left(\nabla\left(\mathfrak{g}_{j,k}^{-1}(w_k;p^*)\right)\right)^p
&=\left\langle-\Delta_p^{H_\infty}\left(\mathfrak{g}_{j,k}^{-1}(w_k;p^*)\right),
\mathfrak{g}_{j,k}^{-1}(w_k;p^*)\right\rangle\\
&=S_{\rm aff}^p\left\|\mathfrak{g}_{j,k}^{-1}(w_k;p^*)\right\|_{p^*}^{p^*}+o(1)
=S_{\rm aff}^p\|w_k\|_{p^*}^{p^*}+o(1)\\
&=S_{\rm aff}^p\|V_j\|_{p^*}^{p^*}+o(1)\quad\text{by \eqref{thm-critical-e6}}\\
&=\left\langle-\Delta_p^{H_\infty}(V_j),
V_j\right\rangle+o(1)\quad\text{by \eqref{thm-critical-e5}}\\
&=\int_{\mathbb{R}^n}H_\infty(\nabla V_j)^p+o(1).
\end{align*}
This further implies that
\begin{align*}
\lim_{k\to\infty}\int_{\mathbb{R}^n}
H_\infty\left(\nabla\left(\mathfrak g_{j,k}^{-1}(w_k;p^*)-V_j\right)\right)^p=0
\end{align*}
and hence, by the equivalence of norms in Euclidean space, we further have
$\mathfrak g_{j,k}^{-1}(w_k;p^*)\to V_j$ in $\dot{W}^{1,p}$. Then,
using a change of variables, it holds that,
for every $\xi\in\mathbb{S}^{n-1}$,
\begin{align*}
\left|\|\nabla_\xi V_j\|_{p}-\|\nabla_\xi w_k\|_p\right|
=\left|\|\nabla_\xi V_j\|_{p}-\left\|\nabla_\xi \mathfrak g_{j,k}^{-1}(w_k;p^*)\right\|_p\right|
\le\left\|\nabla\left(V_j-\mathfrak g_{j,k}^{-1}(w_k;p^*)\right)\right\|_p\to0
\end{align*}
as $k\to\infty$. This shows that $\|\nabla_\xi V_j\|_{p}=H_\infty$ and, in view of \eqref{thm-critical-e5},
\begin{align*}
-\Delta_p^{\rm aff}(V_j)=-\Delta_p^{H_\infty}(V_j)=S_{\rm aff}^p |V_j| ^{p^*-2}V_j.
\end{align*}
Therefore, by the positivity of $V_j$ and Theorem \ref{thm-solu}, we find that
$V_j\in\mathcal{M}^{(1)}_{\rm aff}$.
Since $\mathfrak{g}_{j,k}^{-1}(w_k;p^*)$ takes the form
$\mathcal{T}^{(p^*)}_{\lambda_k,A_k,x_k}u_k$ for some $A_k\in\operatorname{SL}(n)$, we have completed the present proof.
\end{proof}


\section{Expansions for integral functionals}\label{sec-expan}

The goal of this section is to deal with the expansion
of negative-exponent Lebesgue spaces and then
obtain appropriate expansions for the affine Sobolev energy $\mathcal{E}_p(f)$
and the affine Euler--Lagrange residual $\delta_{\rm aff}$.
This technique was perfected in \cite{fz22}, but some changes are needed because of the negative exponent that appears in the affine Sobolev energy.


\subsection{Expansion for Lebesgue spaces with
negative exponents}

In this subsection, we assume that $(\Omega,\mu)$
is a probability space, and prove the following expansion.

\begin{proposition}\label{prop-neg-Lebe}
Let $\alpha\in(0,\infty)$.
Then there is a positive constant $C_{\alpha}$,
depending only on $\alpha$, such that
\begin{enumerate}[{\rm (i)}]
  \item\label{it0-neg-Lebe} for any $\delta\in(0,\frac12)$, $a\in(0,\infty)$,
and $U\in L^\infty(\Omega)$ such that
$\|U\|_{L^\infty(\Omega)}\le a\delta$,
\begin{align}\label{prop-neg-Lebe-e1}
\left[\int_\Omega(a+U)^{-\alpha}\right]^{-\frac1\alpha}
\ge a+\int_{\Omega}U-
\frac{1+\alpha}{2}a^{-1}
\left[\int_{\Omega}U^2-\left(\int_{\Omega}U\right)^2\right]
-C_\alpha a^{-1}\delta\int_{\Omega}U^2;
\end{align}
  \item\label{it1-neg-Lebe} for any $\delta\in(0,\frac12)$, $a\in(0,\infty)$,
$U\in L^\infty(\Omega)$ satisfying
$\|U\|_{L^\infty(\Omega)}\le a\delta$,
and $V\in L^1(\Omega)$,
\begin{align}\label{prop-neg-Lebe-polarized-e1}
&\left[\int_\Omega(a+U)^{-\alpha}\right]^{-1-\frac1\alpha}
\int_\Omega(a+U)^{-\alpha-1}V
\nonumber\\
&\qquad \ge
\int_\Omega V
-(1+\alpha)a^{-1}
\left[\int_{\Omega}
\left(U-\int_{\Omega}U\right)
\left(V-\int_{\Omega}V\right)
\right]
\nonumber\\
&\qquad\quad
-
C_\alpha a^{-1}\delta
\left[
\int_\Omega |UV|
+
\left(\int_\Omega |U|\right)
\left(\int_\Omega |V|\right)
\right];
\end{align}
  \item\label{it2-neg-Lebe} for any $a\in(0,\infty)$
  and $U\in L^\infty(\Omega)$ satisfying
$\|U\|_{L^\infty(\Omega)}\le \frac a2$,
  \begin{align*}
  \left|\left[\int_\Omega(a+U)^{-\alpha}\right]^{-\frac1\alpha}
-a-\int_{\Omega}U\right|\le C_\alpha a^{-1}\|U\|_{L^\infty(\Omega)}^2.
  \end{align*}
\end{enumerate}
\end{proposition}

\begin{proof}
We first show \eqref{it0-neg-Lebe}.
For any $t\in(0,\infty)$, define
$\Psi(t):=\int_{\Omega}(a+tU)^{-\alpha}$
and $\Phi:=\Psi^{-\frac{1}{\alpha}}$.
Note that $$\left[\int_{\Omega}(a+U)^{-\alpha}\right]^{-\frac{1}{\alpha}}=\Phi(1).$$
From this and Taylor's formula, we infer that there is a
$\xi\in[0,1]$ such that
\begin{align}\label{prop-neg-Lebe-e2}
\left[\int_{\Omega}(a+tU)^{-\alpha}\right]^{-\frac{1}{\alpha}}
=\Phi(1)=\Phi(0)
+\Phi'(0)+\frac12\Phi''(0)
+\frac16\Phi'''(\xi).
\end{align}
Now, we calculate derivatives of both $\Phi$ and $\Psi$.
Indeed, for any $k\in\mathbb{N}$ and $t\in(0,\infty)$,
\begin{align*}
\Psi^{(k)}(t)=
(-1)^k\prod_{j=0}^{k-1}
(\alpha+j-1)\int_{\Omega}(a+tU)^{-\alpha-k}U^k,
\end{align*}
\begin{align*}
\Phi'=-\frac1\alpha\Psi^{-1-\frac1\alpha}\Psi',\quad
\Phi''=\frac{1}{\alpha}\left(1+\frac1\alpha\right)
\Psi^{-2-\frac1\alpha}\Psi'^2
-\frac{1}{\alpha}\Psi^{-1-\frac1\alpha}\Psi'',
\end{align*}
and
\begin{align}\label{prop-neg-Lebe-e4}
\Phi'''=-\frac{1}{\alpha}\left(1+\frac1\alpha\right)
\left(2+\frac1\alpha\right)
\Psi^{-3-\frac1\alpha}\Psi'^3+\frac{2}{\alpha}
\left(1+\frac1\alpha\right)\Psi^{-2-\frac1\alpha}
\Psi'\Psi''
-\frac1\alpha\Psi^{-1-\frac1\alpha}\Psi'''.
\end{align}
Hence, we find that
the first three terms on the right-hand side
of \eqref{prop-neg-Lebe-e1} are exactly
$\Phi(0)$, $\Phi'(0)$, and $\frac12\Phi''(0)$, respectively.
By this and \eqref{prop-neg-Lebe-e2}, to prove \eqref{prop-neg-Lebe-e1},
it suffices to show
\begin{align}\label{prop-neg-Lebe-e3}
\left\|\Phi'''\right\|_{L^\infty[0,1]}\lesssim a^{-1}\delta\int_{\Omega}U^2.
\end{align}
For any $t\in[0,1]$, since $\delta\in(0,\frac12)$, from
$(1-\delta)a\le|a+tU|\le (1+\delta)a$ ,
it follows that $|a+tU|\in(\frac12a,\frac32a)$.
Therefore, for any $t\in[0,1]$, $|\Psi(t)|\sim a^{-\alpha}$
and, for any $k\in\mathbb{N}$,
\begin{align}\label{prop-neg-Lebe-e5}
\left|\Psi^{(k)}(t)\right|
&\lesssim a^{-\alpha-k}\int_{\Omega}|U|^k
\le
\begin{cases}
\displaystyle{a^{-\alpha-1}}\int_{\Omega}|U|&\text{if $k=1$},\\
\displaystyle{a^{-\alpha-2}\delta^{k-2}}\int_{\Omega}U^2
&\text{if $k\ge2$}.
\end{cases}
\end{align}
Using these, we obtain, for any $t\in[0,1]$,
\begin{align*}
\left|\Psi(t)^{-3-\frac1\alpha}\Psi'(t)^3
\right|
&\lesssim a^{1+3\alpha}\left(a^{-\alpha-1}\int_{\Omega}|U|\right)^3
=a^{-2}\left(\int_{\Omega}|U|\right)^3\\
&\le a^{-2}\int_{\Omega}|U|\int_{\Omega}|U|^2
\quad\text{by H\"older's inequality}\\
&\le a^{-1}\delta\int_{\Omega}|U|^2\quad\text{by assumption
$\|U\|_{L^\infty}\le a\delta$}.
\end{align*}
Similarly (and even more simply), we deduce that
the remaining two terms
on the right-hand side of \eqref{prop-neg-Lebe-e4} can also
be controlled by $a^{-1}\delta\int_{\Omega}|U|^2$.
Consequently, \eqref{prop-neg-Lebe-e3} holds.
We then complete the proof of \eqref{prop-neg-Lebe-e1} and hence \eqref{it0-neg-Lebe}.

We next prove \eqref{it1-neg-Lebe}.
For any $t\in(0,\infty)$, define
$\Theta(t):=\int_\Omega(a+tU)^{-\alpha-1}V$
and $\Upsilon:=\Psi^{-1-\frac1\alpha}\Theta$.
Then
\begin{align*}
\left[\int_\Omega(a+U)^{-\alpha}\right]^{-1-\frac1\alpha}
\int_\Omega(a+U)^{-\alpha-1}V
=
\Upsilon(1).
\end{align*}
By Taylor's formula, there is a $\xi_1\in[0,1]$ such that
\begin{align*}
\Upsilon(1)
=
\Upsilon(0)+\Upsilon'(0)+\frac12\Upsilon''(\xi_1).
\end{align*}
Now, we calculate derivatives of the above functions.
Indeed, for any $k\in\mathbb{N}$ and $t\in(0,\infty)$,
it holds that
\begin{align}\label{prop-neg-Lebe-polarized-e12}
\Theta^{(k)}(t)
=(-1)^{k}\prod_{j=1}^{k}(\alpha+j)
\int_{\Omega}(a+tU)^{-\alpha-k-1}U^kV,
\end{align}
\begin{align*}
\Upsilon'
=
-\left(1+\frac1\alpha\right)
\Psi^{-2-\frac1\alpha}\Psi'(t)\Theta
+
\Psi^{-1-\frac1\alpha}\Theta',
\end{align*}
and
\begin{align*}
\Upsilon''
&=
\left(1+\frac1\alpha\right)
\left(2+\frac1\alpha\right)
\Psi^{-3-\frac1\alpha}\Psi'^2\Theta-
\left(1+\frac1\alpha\right)
\Psi^{-2-\frac1\alpha}\Psi''\Theta
\nonumber\\
&\quad
-
2\left(1+\frac1\alpha\right)
\Psi(t)^{-2-\frac1\alpha}\Psi'\Theta'
+
\Psi(t)^{-1-\frac1\alpha}\Theta''.
\end{align*}
Therefore,
the first two terms on the right-hand side
of \eqref{prop-neg-Lebe-polarized-e1} are exactly
$\Upsilon(0)$ and $\Upsilon'(0)$, respectively.
Then it remains to show
\begin{align}\label{prop-neg-Lebe-polarized-e6}
\|\Upsilon''\|_{L^\infty[0,1]}
\lesssim a^{-1}\delta\left[
\int_\Omega |UV|
+\left(\int_\Omega |U|\right)
\left(\int_\Omega |V|\right)
\right].
\end{align}
Indeed, applying \eqref{prop-neg-Lebe-e5}, we find that
\begin{align}\label{prop-expan-lap-e11}
&\left|\Psi^{-3-\frac1\alpha}\Psi'^2\Theta
\right|+\left|\Psi^{-2-\frac1\alpha}\Psi''\Theta\right|\notag\\
&\quad\lesssim\left(a^{-\alpha}\right)^{-3-\frac1\alpha}
\left(a^{-\alpha-1}\int_{\Omega}|U|\right)^2
\int_{\Omega}a^{-\alpha-1}|V|
+\left(a^{-\alpha}\right)^{-2-\frac1\alpha}
a^{-\alpha-1}\delta\int_{\Omega}|U|
\int_{\Omega}a^{-\alpha-1}|V|\notag\\
&\quad\lesssim a^{-1}\delta\int_{\Omega}|U|
\int_{\Omega}|V|.
\end{align}
On the other hand, from \eqref{prop-neg-Lebe-polarized-e12},
we infer that, for every $k\in\mathbb{N}$ and $t\in[0,1]$,
\begin{align*}
|\Theta^{(k)}(t)|
&\lesssim a^{-\alpha-k-1}\int_{\Omega}|U^kV|
\quad\text{by $a+tU\sim a$}\\
&\le \begin{cases}
\displaystyle{a^{-\alpha-1}}\delta\int_{\Omega}|V|&\text{if $k=1$},\\
\displaystyle{a^{-\alpha-2}\delta^{k-1}}\int_{\Omega}|UV|
&\text{if $k\ge2$}
\end{cases}
\quad\text{by $\|U\|_{L^\infty(\Omega)}\le a\delta$}.
\end{align*}
Applying this, \eqref{prop-neg-Lebe-e5},
and a similar argument to that used in the estimation of \eqref{prop-expan-lap-e11},
we obtain
\begin{align*}
\left|\Psi^{-2-\frac1\alpha}\Psi'\Theta'\right|
+\left|\Psi^{-1-\frac1\alpha}\Theta''\right|
\lesssim\operatorname{RHS}\eqref{prop-neg-Lebe-polarized-e6}.
\end{align*}
This then finishes the proof of \eqref{prop-neg-Lebe-polarized-e6}
and hence \eqref{it1-neg-Lebe}.

Finally, we prove \eqref{it2-neg-Lebe}. The proof proceeds along very similar lines
to that of \eqref{it0-neg-Lebe}.
Indeed, by an argument similar to that used
in \eqref{prop-neg-Lebe-e3}, we conclude that
\begin{align*}
\left\|\Phi''\right\|_{L^\infty(\Omega)}
\lesssim \left(a^{-\alpha}\right)^{-2-\frac1\alpha}
\left(a^{-\alpha-1}\int_{\Omega}|U|\right)^2
+\left(a^{-\alpha}\right)^{-1-\frac1\alpha}
a^{-\alpha-2}\delta\int_{\Omega}|U|^2
\le a^{-1}\|U\|_{L^\infty(\Omega)}^2.
\end{align*}
This, together with
Taylor's formula for $\Phi$ with second order remainder,
gives the desired estimate in \eqref{it2-neg-Lebe}.
Then the present proof is complete.
\end{proof}


\subsection{Local expansion for affine Sobolev energy}\label{subsec-expan}

In this subsection, we establish the expansion of
the affine Sobolev energy
near a given function $f\in\dot{W}^{1,p}$,
which is important in the proof of Theorem \ref{thm-stab}.

It will suffice to consider the case where $f$ is \emph{radial}. Then $\|\nabla_\xi f\|_p$ is independent of $\xi$ and hence
$\mathcal{E}_p(f)=|\mathbb{S}^{n-1}|^{-\frac1n}
\|\nabla_\xi f\|_{p}$.

For $g\in\dot{W}^{1,p}$ and $\xi\in\mathbb{S}^{n-1}$, define
\begin{align}\label{def-L1}
L^{(1)}_\xi(f,g):=\int_{\mathbb{R}^n}|\nabla_\xi f|^{p-2}
\nabla_\xi f\nabla_\xi g,
\end{align}
\begin{align}\label{def-var}
\operatorname{Var}(f,g):=\int_{\mathbb{S}^{n-1}}\left[
L^{(1)}_\xi(f,g)-\fint_{\mathbb{S}^{n-1}}
L^{(1)}_u(f,g)\,du\right]^2 d\xi,
\end{align}
and
\begin{align}\label{def-L2}
L^{(2)}_\xi(f,g):=&\int_{\mathbb{R}^n}
|\nabla_\xi f|^{p-2}(\nabla_\xi g)^2\notag\\
&\qquad+(p-2)\frac{\min\{|\nabla_\xi (f+g)|,|\nabla_\xi f|\}^{p-1}}{|\nabla_\xi f|}
\left[|\nabla_\xi(f+g)|-|\nabla_\xi f|\right]^2,
\end{align}
where, if $|\nabla_\xi f|=0$, then define
$\frac{\min\{|\nabla_\xi (f+g)|,|\nabla_\xi f|\}^{p-1}}{|\nabla_\xi f|}:=0$.
It is clear that, for every $i\in\{1,2\}$ and $\xi\in\mathbb{S}^{n-1}$,
\begin{align}\label{eq-L-dila}
 L^{(i)}_\xi(cf,g)=c^p L^{(1)}_\xi(f,c^{-1}g).
\end{align}
Moreover,
by the triangle inequality, for any $\xi\in\mathbb{S}^{n-1}$,
\begin{align*}
\frac{\min\{|\nabla_\xi (f+g)|,|\nabla_\xi f|\}^{p-1}}{|\nabla_\xi f|}
\left[|\nabla_\xi(f+g)|-|\nabla_\xi f|\right]^2
\le|\nabla_\xi f|^{p-2}|\nabla_\xi g|^2;
\end{align*}
therefore, if $p\ge 2$, then using H\"older's inequality,
we obtain
\begin{align}\label{eq-L2}
L^{(2)}_\xi(f,g)\le
(p-1)\int_{\mathbb{R}^{n-1}}|\nabla_\xi f|^{p-2}
|\nabla_\xi g|^2
\le(p-1)\|\nabla_\xi f\|_p^{p-2}\|\nabla_\xi g\|_p^2.
\end{align}

Then we have the following local expansion of $\mathcal{E}_p(f+g)$.

\begin{proposition}\label{prop-expan-local}
Let $p\ge 2$. Then there are positive constants $c_{n,p}$,
$C_{n,p}$, and $\widetilde{C}_{n,p}$,
depending only on $n$ and $p$, such that, for any
$\varepsilon\in(0,c_{n,p})$ and $g\in\dot{W}^{1,p}$
satisfying
\begin{align}\label{prop-expan-local-e1}
\sup_{\xi\in\mathbb{S}^{n-1}}
\left\{
\left[
\frac{L^{(2)}_\xi(f,g)}
{\|\nabla_\xi f\|_p^p}
\right]^{\frac12}
+
\frac{\|\nabla_\xi g\|_{p}^p}
{\|\nabla_\xi f\|_p^p}
\right\}\le \varepsilon,
\end{align}
it holds that
\begin{align}\label{prop-expan-local-e2}
\mathcal{E}_p(f+g)^p
&\ge\mathcal{E}_p(f)^p+p|\mathbb{S}_{n-1}|^{-1-\frac pn}
\int_{\mathbb{S}^{n-1}}\left[L^{(1)}_\xi(f,g)
+\frac{1-C_{n,p}\varepsilon}{2}L^{(2)}_\xi(f,g)\right]
\,d\xi\notag\\
&\qquad-(n+p)|\mathbb{S}^{n-1}|^{-1-\frac{2p}{n}}
\mathcal E_p(f)^{-p}
\operatorname{Var}(f,g)
+\widetilde C_{n,p}\varepsilon\|\nabla g\|_{p}^p.
\end{align}
\end{proposition}

\begin{remark}\label{rem-3.4}
Note that $f$ is radial and $\|\nabla_\xi g\|_p\le \|\nabla g\|_p|\xi|$ for all
$\xi\in\mathbb{S}^{n-1}$.
By these and \eqref{eq-L2}, we conclude that, if $\varepsilon$ is sufficiently
small, then, for all $g\in\dot{W}^{1,p}$ such that $\|\nabla g\|_p<\varepsilon\|\nabla f\|_p$,
it holds that
$\sup_{\xi\in\mathbb{S}^{n-1}}\frac{\|\nabla_\xi g\|_p}{\|\nabla_\xi f\|_p}\lesssim\varepsilon$ and
hence, for each $\xi\in\mathbb{S}^{n-1}$,
\begin{align*}
\left[
\frac{L^{(2)}_\xi(f,g)}
{\|\nabla_\xi f\|_p^p}
\right]^{\frac12}
+
\frac{\|\nabla_\xi g\|_{p}^p}
{\|\nabla_\xi f\|_p^p}
&\lesssim\frac{\|\nabla_\xi g\|_p}{\|\nabla_\xi f\|_p}+
\left(\frac{\|\nabla_\xi g\|_p}{\|\nabla_\xi f\|_p}\right)^p\lesssim\varepsilon
\end{align*}
with the implicit positive constants depending only on $n$ and $p$.
This shows that, if $f+g$ is sufficiently close to $f$ in $\dot{W}^{1,p}$,
then the expansion \eqref{prop-expan-local-e2} holds,
which is precisely what is required for the proof of Theorem \ref{thm-stab}.
Nevertheless, we choose to keep the present assumption \eqref{prop-expan-local-e1}
to better highlight the parameter dependencies.
\end{remark}

\begin{proof}[Proof of Proposition \ref{prop-expan-local}]
We need to use several pointwise inequalities from Figalli--Zhang \cite{fz22}, which are exactly
\cite[Lemma 2.1(ii) and (2.4)]{fz22}:
For any $\theta\in[0,\infty)$, there is a positive
constant $\widetilde{c}_{p,\theta}$ (if $\theta=0$, then
choose $\widetilde{c}_{p,\theta}:=0$), depending only on $p$ and $\theta$,
such that, for every $x,y\in\mathbb R$,
\begin{align}\label{eq:pointwise-FZ}
  |x+y|^p
  &\ge |x|^p+p|x|^{p-2}xy\notag\\
  &\quad+\frac{(1-\theta)p}{2}\left[|x|^{p-2}|y|^{2}
  +(p-2)\frac{\min\{|x+y|,|x|\}^{p-1}}{|x|}\left(|x+y|-|x|\right)^2\right]
  +\widetilde{c}_{p,\theta}|y|^p.
\end{align}
We first apply this inequality with $\theta=0$ and then obtain,
for any $\xi\in\mathbb{S}^{n-1}$,
\begin{align*}
\|\nabla_\xi(f+g)\|_{p}^p
\ge\|\nabla_\xi f\|_p^p
+p\left[L^{(1)}_{\xi}(f,g)+\frac12L^{(2)}_\xi(f,g)\right]
=:a+U.
\end{align*}
Observe that, by H\"older's inequality and \eqref{prop-expan-local-e1},
we conclude that
\begin{align}\label{prop-expan-local-e3}
\left[L^{(1)}_\xi(f,g)\right]^2
\le\int_{\mathbb{R}^n}|\nabla_\xi f|^p\int_{\mathbb{R}^n}
|\nabla_\xi f|^{p-2}|\nabla_\xi g|^2\le aL^{(2)}_\xi(f,g)
\le a^2\varepsilon^2.
\end{align}
This implies
$\|U\|_{L^\infty(\mathbb{S}^{n-1})}
\le ap\varepsilon(1+\frac12\varepsilon).$
Hence, by Proposition \ref{prop-neg-Lebe},
if $p\varepsilon(1+\frac12\varepsilon)<\frac12$, then there is a
$C_1$, depending only on $n$ and $p$, such that
\begin{align}\label{prop-expan-local-e4}
&\mathcal{E}_p(f+g)^p\notag\\
&\quad\ge\mathcal{E}_p(f)^p+p|\mathbb{S}_{n-1}|^{-1-\frac pn}
\int_{\mathbb{S}^{n-1}}\left[L^{(1)}_\xi(f,g)
+\frac{1}{2}L^{(2)}_\xi(f,g)\right]\,d\xi\notag\\
&\qquad-\frac{p(n+p)}{2}|\mathbb{S}^{n-1}|^{-\frac{2p}{n}}
\mathcal{E}_p(f)^{-p}\notag\\
&\quad\qquad\times
\left\{\fint_{\mathbb{S}^{n-1}}
\left[L^{(1)}_\xi(f,g)+\frac12 L^{(2)}_\xi(f,g)\right]^2
-\left[\fint_{\mathbb{S}^{n-1}}
L^{(1)}_\xi(f,g)+\frac12 L^{(2)}_\xi(f,g)\right]^2\right\}\notag\\
&\qquad-C_1p\varepsilon\left(1+\frac12\varepsilon\right)
|\mathbb{S}^{n-1}|^{-1-\frac{2p}{n}}\mathcal{E}_p(f)^{-p}
\int_{\mathbb{S}^n}\left[L^{(1)}_\xi(f,g)+\frac12 L^{(2)}_\xi(f,g)\right]^2.
\end{align}
Using \eqref{prop-expan-local-e3} and
the penultimate inequality in \eqref{prop-expan-local-e1}, we obtain
\begin{align*}
\left[L^{(1)}_\xi(f,g)+\frac12 L^{(2)}_\xi(f,g)\right]^2
\le \left[\sqrt{aL^{(2)}_\xi(f,g)}
+\frac12\varepsilon\sqrt{aL^{(2)}_\xi(f,g)}\right]^2
\le aL^{(2)}_{\xi}(f,g)\left(1+\frac12\varepsilon\right)^2.
\end{align*}
This further implies
\begin{align*}
&p\varepsilon\left(1+\frac12\varepsilon\right)
|\mathbb{S}^{n-1}|^{-1-\frac{2p}{n}}\mathcal{E}_p(f)^{-p}
\int_{\mathbb{S}^n}\left[L^{(1)}_\xi(f,g)+\frac12 L^{(2)}_\xi(f,g)\right]^2
\lesssim p|\mathbb{S}^{n-1}|^{-1-\frac pn}
\varepsilon\int_{\mathbb{S}^{n-1}}L^{(2)}_\xi(f,g),
\end{align*}
where the implicit positive constant depends only on $n$ and $p$.

On the other hand,
\begin{align*}
&\fint_{\mathbb{S}^{n-1}}
\left[L^{(1)}_\xi(f,g)+\frac12 L^{(2)}_\xi(f,g)\right]^2
-\left[\fint_{\mathbb{S}^{n-1}}
L^{(1)}_\xi(f,g)+\frac12 L^{(2)}_\xi(f,g)\right]^2\\
&\quad=\fint_{\mathbb{S}^{n-1}}
\left[L^{(1)}_\xi(f,g)+\frac12 L^{(2)}_\xi(f,g)
-\fint_{\mathbb{S}^{n-1}}L^{(1)}_\xi(f,g)+\frac12 L^{(2)}_\xi(f,g)\right]^2\\
&\quad=\fint_{\mathbb{S}^{n-1}}
\left[L^{(1)}_\xi(f,g)-\fint_{\mathbb{S}^{n-1}}L^{(1)}_\xi(f,g)\right]^2
+\frac14\fint_{\mathbb{S}^{n-1}}
\left[L^{(2)}_\xi(f,g)-\fint_{\mathbb{S}^{n-1}}L^{(2)}_\xi(f,g)\right]^2\\
&\qquad-\fint_{\mathbb{S}^{n-1}}
\left[L^{(1)}_\xi(f,g)-\fint_{\mathbb{S}^{n-1}}L^{(1)}_\xi(f,g)\right]
\left[L^{(2)}_\xi(f,g)-\fint_{\mathbb{S}^{n-1}}L^{(2)}_\xi(f,g)\right].
\end{align*}
From \eqref{prop-expan-local-e1}, we deduce that
\begin{align*}
\fint_{\mathbb{S}^{n-1}}
\left[L^{(2)}_\xi(f,g)-\fint_{\mathbb{S}^{n-1}}L^{(2)}_\xi(f,g)\right]^2
\le\fint_{\mathbb{S}^{n-1}}\left[L_\xi^{(2)}(f,g)\right]^2
\le |\mathbb{S}^{n-1}|^{\frac pn-1}\mathcal{E}_p(f)^p\varepsilon^2
\int_{\mathbb{S}^{n-1}}L^{(2)}_\xi(f,g).
\end{align*}
Moreover, by \eqref{prop-expan-local-e3}, we find that,
for any $\xi\in\mathbb{S}^{n-1}$,
\begin{align}\label{prop-expan-local-e9}
\left|L^{(1)}_{\xi}(f,g)-\fint_{\mathbb{S}^{n-1}}L^{(1)}_\xi(f,g)\right|
\le 2\|L^{(1)}_\xi(f,g)\|_{L^\infty(\mathbb{S}^{n-1})}
\le 2a\varepsilon.
\end{align}
Thus, it holds that
\begin{align*}
&\left|\fint_{\mathbb{S}^{n-1}}
\left[L^{(1)}_\xi(f,g)-\fint_{\mathbb{S}^{n-1}}L^{(1)}_\xi(f,g)\right]
\left[L^{(1)}_\xi(f,g)-\fint_{\mathbb{S}^{n-1}}L^{(1)}_\xi(f,g)\right]\right|\\
&\quad\le 4a\varepsilon\fint_{\mathbb{S}^{n-1}}L^{(1)}_\xi(f,g)
=4|\mathbb{S}^{n-1}|^{\frac pn-1}\mathcal{E}_p(f)^p\varepsilon
\int_{\mathbb{S}^{n-1}}L^{(1)}_\xi(f,g).
\end{align*}
Incorporating the above estimates into \eqref{prop-expan-local-e4}
yields that, for sufficiently small $\varepsilon$,
\begin{align}\label{prop-expan-local-e5}
&\mathcal{E}_p(f+g)^p\notag\\
&\quad\ge\mathcal{E}_p(f)^p+p|\mathbb{S}_{n-1}|^{-1-\frac pn}
\int_{\mathbb{S}^{n-1}}\left[L^{(1)}_\xi(f,g)
+\frac{1-C_2\varepsilon}{2}L^{(1)}_\xi(f,g)\right]\,d\xi\notag\\
&\qquad-\frac{p(n+p)}{2}|\mathbb{S}^{n-1}|^{-1-\frac{2p}{n}}
\mathcal{E}_p(f)^{-p}
\fint_{\mathbb{S}^{n-1}}
\left[L^{(1)}_\xi(f,g)-\fint_{\mathbb{S}^{n-1}}L^{(1)}_\xi(f,g)\right]^2,
\end{align}
where $C_2$ depends only on $n$ and $p$.

To complete the proof, we still need to add term $C\|\nabla g\|_{p}^p$
on the right-hand side of \eqref{prop-expan-local-e5}. To achieve this,
we use \eqref{eq:pointwise-FZ} with $\theta=1$ and then conclude that
there is a positive constant $C_3$, depending only on $p$, such that,
for any $\xi\in\mathbb{S}^{n-1}$,
\begin{align*}
\|\nabla_\xi(f+g)\|_{p}^p
\ge\|\nabla_\xi f\|_p^p+pL^{(1)}_{\xi}(g)+C_3\|\nabla_\xi g\|_p^p
=:a+\widetilde{U}.
\end{align*}
From \eqref{prop-expan-local-e3} again, we infer that
$\|\widetilde U\|_{L^\infty(\mathbb{S}^{n-1})}\le a\varepsilon(p+C_3)$.
Consequently, if $\varepsilon(p+C_3)<\frac12$, then, applying
Proposition \ref{prop-neg-Lebe} again,
we find that there is a $C_4$, depending only on $n$ and $p$, such that
\begin{align}\label{prop-expan-local-e6}
\mathcal{E}_p(f+g)^p
&\ge\mathcal{E}_p(f)^p
+|\mathbb{S}^{n-1}|^{-1-\frac pn}
\int_{\mathbb{S}^{n-1}}\left[pL^{(1)}_\xi(f,g)+C_3\|\nabla_\xi g\|_p^p\right]\notag\\
&\quad-\frac{1+\frac np}{2}|\mathbb{S}^{n-1}|^{-\frac{2p}n}
\mathcal{E}_p(f)^{-p}\notag\\
&\qquad\times\left\{\fint_{\mathbb{S}^{n-1}}
\left[pL^{(1)}_\xi(f,g)+C_3\|\nabla_\xi g\|_{p}^p\right]^2
-\left[\fint_{\mathbb{S}^{n-1}}
pL^{(1)}_\xi(f,g)+C_3\|\nabla_\xi g\|_{p}^p\right]^2\right\}\notag\\
&\quad-C_4\varepsilon(p+C_3)|\mathbb{S}^{n-1}|^{-1-\frac{2p}{n}}
\mathcal{E}_p(f)^{-p}\int_{\mathbb{S}^{n-1}}
\left[pL^{(1)}_\xi(f,g)+C_3\|\nabla_\xi g\|_{p}^p\right]^2\notag\\
&\quad=:\mathcal{E}_p(f)^p
+|\mathbb{S}^{n-1}|^{-1-\frac pn}
\int_{\mathbb{S}^{n-1}}\left[pL^{(1)}_\xi(f,g)+C_3\|\nabla_\xi g\|_p^p\right]
-I_1(g)-I_2(g).
\end{align}
By Tonelli's theorem, we conclude that
\begin{align}\label{eq-lp}
\int_{\mathbb{S}^{n-1}}\|\nabla_\xi g\|_p^p\,d\xi=
\alpha_{n,p}^p\|\nabla g\|_p^p,
\end{align}
where $\alpha_{n,p}$ is as in \eqref{eq-alpha}.
Therefore, we only need to show that, for sufficiently small $\varepsilon$,
\begin{align}\label{prop-expan-local-e7}
I_1(g)+I_2(g)&\le \frac{p(n+p)}{2}|\mathbb{S}^{n-1}|^{-1-\frac{2p}{n}}
\mathcal{E}_p(f)^{-p}
\fint_{\mathbb{S}^{n-1}}
\left[L^{(1)}_\xi(f,g)-\fint_{\mathbb{S}^{n-1}}L^{(1)}_\xi(f,g)\right]^2\notag\\
&\quad+C_5\varepsilon\int_{\mathbb{S}^{n-1}}
\left[L^{(1)}_\xi(f,g)+\|\nabla_\xi g\|_{p}^p\right],
\end{align}
where the positive constant $C_5$ depends only on $n$ and $p$.
Indeed, if \eqref{prop-expan-local-e7} holds, then
\eqref{prop-expan-local-e6} implies that there are
two positive constants $C_6$ and $C_7$, depending only on $n$ and $p$,
such that
\begin{align}\label{prop-expan-local-e8}
&\mathcal{E}_p(f+g)^p\notag\\
&\quad\ge\mathcal{E}_p(f)^p+p|\mathbb{S}_{n-1}|^{-1-\frac pn}
\int_{\mathbb{S}^{n-1}}\left[L^{(1)}_\xi(f,g)
-\frac{C_6\varepsilon}{2}L^{(1)}_\xi(f,g)\right]\,d\xi\notag\\
&\qquad-\frac{p(n+p)}{2}|\mathbb{S}^{n-1}|^{-1-\frac{2p}{n}}
\mathcal{E}_p(f)^{-p}
\fint_{\mathbb{S}^{n-1}}
\left[L^{(1)}_\xi(f,g)-\fint_{\mathbb{S}^{n-1}}L^{(1)}_\xi(f,g)\right]^2
+C_7\|\nabla g\|_p^p.
\end{align}
Letting $(1-\varepsilon)\times\eqref{prop-expan-local-e5}+\varepsilon\times\eqref{prop-expan-local-e8}$,
we obtain the desired estimate \eqref{prop-expan-local-e2}.

Hence, to complete the proof, we now prove \eqref{prop-expan-local-e7}.
We begin with the simpler term $I_2(g)$.
Indeed, using \eqref{prop-expan-local-e3} and
the penultimate inequality in \eqref{prop-expan-local-e1}, we obtain
\begin{align*}
\int_{\mathbb{S}^{n-1}}\left[pL^{(1)}_\xi(f,g)+C_3\|\nabla_\xi g\|_{p}^p\right]^2
&\le 2\int_{\mathbb{S}^{n-1}}\left\{p^2\left[L^{(1)}_\xi(f,g)\right]^2
+C_3^2\|\nabla_\xi g\|_{p}^{2p}\right\}\\
&\le 2\int_{\mathbb{S}^{n-1}}\left[p^2aL^{(1)}_\xi(f,g)
+C_3^2\varepsilon^2a\|\nabla_\xi g\|_p^p\right]\\
&=2(p^2+C_3^2\varepsilon^2)|\mathbb{S}^{n-1}|^{\frac pn}
\mathcal{E}_p(f)^p
\int_{\mathbb{S}^{n-1}}\left[L^{(1)}_\xi(f,g)+\|\nabla_\xi g\|_p^p\right].
\end{align*}
This further implies that, for sufficiently small $\varepsilon$,
\begin{align*}
I_2(g)\lesssim\varepsilon
\int_{\mathbb{S}^{n-1}}\left[L^{(1)}_\xi(f,g)+\|\nabla_\xi g\|_p^p\right].
\end{align*}
Next, we deal with $I_1(g)$.
Note that
\begin{align*}
&\fint_{\mathbb{S}^{n-1}}
\left[pL^{(1)}_\xi(f,g)+C_3\|\nabla_\xi g\|_{p}^p\right]^2
-\left[\fint_{\mathbb{S}^{n-1}}
pL^{(1)}_\xi(f,g)+C_3\|\nabla_\xi g\|_{p}^p\right]^2\\
&\quad=p^2\fint_{\mathbb{S}^{n-1}}
\left[L^{(1)}_\xi(f,g)-\fint_{\mathbb{S}^{n-1}}L^{(1)}_\xi(f,g)\right]^2
+C_3^2\fint_{\mathbb{S}^{n-1}}
\left[\|\nabla_\xi g\|_p^p-\fint_{\mathbb{S}^{n-1}}\|\nabla_\xi g\|_p^p\right]^2\\
&\qquad+2pC_3\fint_{\mathbb{S}^{n-1}}
\left[L^{(1)}_\xi(f,g)-\fint_{\mathbb{S}^{n-1}}L^{(1)}_\xi(f,g)\right]
\left[\|\nabla_\xi g\|_p^p-\fint_{\mathbb{S}^{n-1}}\|\nabla_\xi g\|_p^p\right].
\end{align*}
By \eqref{prop-expan-local-e1} and \eqref{prop-expan-local-e9}, we conclude that
\begin{align*}
\fint_{\mathbb{S}^{n-1}}
\left[\|\nabla_\xi g\|_p^p-\fint_{\mathbb{S}^{n-1}}\|\nabla_\xi g\|_p^p\right]^2
\le \fint_{\mathbb{S}^{n-1}}\|\nabla_\xi g\|_{p}^{2p}\le
|\mathbb{S}^{n-1}|^{\frac pn-1}\mathcal{E}_p(f)^p
\varepsilon^2\int_{\mathbb{S}^{n-1}}\|\nabla_\xi g\|_{p}^p
\end{align*}
and
\begin{align*}
&\left|\fint_{\mathbb{S}^{n-1}}
\left[L^{(1)}_\xi(f,g)-\fint_{\mathbb{S}^{n-1}}L^{(1)}_\xi(f,g)\right]
\left[\|\nabla_\xi g\|_p^p-\fint_{\mathbb{S}^{n-1}}\|\nabla_\xi g\|_p^p\right]\right|\\
&\quad\le 4a\varepsilon\fint_{\mathbb{S}^{n-1}}\|\nabla_\xi g\|_{p}^p
=4|\mathbb{S}^{n-1}|^{-\frac pn-1}\mathcal{E}_p(f)^{p}\varepsilon
\int_{\mathbb{S}^{n-1}}\|\nabla_\xi g\|_p^p.
\end{align*}
Combining these estimates, we obtain \eqref{prop-expan-local-e7}
for sufficiently small $\varepsilon$,
which completes the proof of the present proposition.
\end{proof}


\subsection{Local expansion for affine $p$-Laplacian}\label{subsec-expan2}

Next, we consider the expansion of the affine $p$ Laplacian, which is important in the proof of Theorem \ref{thm-bubble}.

Recall that, to study the stability of critical points of the classical Sobolev inequality,
Liu and Zhang \cite[Lemma 2.1(2)]{lz25} derived the following
expansion: If $p\ge 2$, then there is a
$c_p\in(0,1)$, depending only on $p$, such that,
for any $\theta\in [0,\infty)$, there is a positive
constant $\widetilde{c}_{p,\theta}$
(if $\theta=0$, then
choose $\widetilde{c}_{p,\theta}:=0$), depending only on $p$ and $\theta$,
such that, for every $x,y\in\mathbb R$,
\begin{align}\label{eq:pointwise-LZ}
  |x+y|^{p-2}(x+y)y
  &\ge |x|^{p-2}xy\notag\\
  &\quad+(1-\theta)
  \frac{\min\{|x|,\max\{c_{p}^{\frac1{p-1}}|x|,|x+y|\}\}^{p-1}}{|x|}|y|^2\notag\\
  &\quad+(1-\theta)(p-2)
  \frac{\min\{|x+y|,|x|\}^{p-1}}{|x|}\left(|x+y|-|x|\right)^2
  +\widetilde{c}_{p,\theta}|y|^p.
\end{align}

As in the previous subsection we assume that $f\in \dot W^{1,p}$ is \emph{radial}. For any $g\in\dot{W}^{1,p}$ and $\xi\in\mathbb{S}^{n-1}$, define
\begin{align}\label{eq-L2w}
\widetilde{L}^{(2)}_\xi(f,g)
:=&\int_{\mathbb{R}^n}
\frac{\min\{|\nabla_\xi f|,\max\{c_{p}^{\frac1{p-1}}|\nabla_\xi f|,|\nabla_\xi(f+g)|\}\}^{p-1}}
{|\nabla_\xi f|}|\nabla_\xi g|^2\notag\\
&\quad+(p-2)\frac{\min\{|\nabla_\xi (f+g)|,|\nabla_\xi f|\}^{p-1}}{|\nabla_\xi f|}
\left[|\nabla_\xi(f+g)|-|\nabla_\xi f|\right]^2,
\end{align}
where $c_{p}$ is as in \eqref{eq:pointwise-LZ}.
It is clear that $\widetilde{L}^{(2)}$ has the same
dilation property as in \eqref{eq-L-dila}
and $\widetilde{L}^{(2)}_\xi(f,g)\le L^{(2)}_\xi(f,g)$.

\begin{proposition}\label{prop-expan-lap}
Let $p\ge 2$. Then there are positive constants $c_{n,p}$,
$C_{n,p}$, and $\widetilde C_{n,p}$,
depending only on $n$ and $p$, such that, for any
$\varepsilon\in(0,c_{n,p})$ and $g\in\dot W^{1,p}$ satisfying
\begin{align}\label{prop-expan-lap-e1}
\sup_{\xi\in\mathbb{S}^{n-1}}
\left[
\left(
\frac{\widetilde{L}^{(2)}_\xi(f,g)}
{\|\nabla_\xi f\|_p^p}
\right)^{\frac12}
+
\frac{\|\nabla_\xi g\|_{p}^p}
{\|\nabla_\xi f\|_p^p}
\right]\le \varepsilon,
\end{align}
it holds that
\begin{align}\label{prop-expan-lap-e2}
-\int_{\mathbb{R}^n}
\Delta_p^{\rm aff}(f+g)g
&\ge
-\int_{\mathbb{R}^n}\Delta_p^{\rm aff}(f)g
+(1-C_{n,p}\varepsilon)|\mathbb{S}^{n-1}|^{-1-\frac pn}
\int_{\mathbb{S}^{n-1}}\widetilde{L}^{(2)}_\xi(f,g)\,d\xi\notag\\
&\qquad-(n+p)|\mathbb{S}^{n-1}|^{-1-\frac{2p}{n}}
\mathcal E_p(f)^{-p}
\operatorname{Var}(f,g)
+\widetilde C_{n,p}\varepsilon\|\nabla g\|_{p}^p.
\end{align}
\end{proposition}

\begin{remark}\label{rem-3.3}
Similar to Remark \ref{rem-3.4},
if $\varepsilon$ is sufficiently
small, then, for all $g\in\dot{W}^{1,p}$ such that $\|\nabla g\|_p<\varepsilon\|\nabla f\|_p$,
it holds that the left-hand side of $\eqref{prop-expan-lap-e1}\lesssim\varepsilon$
with the implicit positive constants depending only on $n$ and $p$.
This shows that, if $f+g$ is sufficiently close to $f$ in $\dot{W}^{1,p}$,
then the expansion \eqref{prop-expan-lap-e2} holds.
\end{remark}

\begin{proof}[Proof of Proposition \ref{prop-expan-lap}]
We choose $\theta=0$ and $\theta=1$ in \eqref{eq:pointwise-LZ} to
derive two estimates, and then take a convex combination of them.

Indeed, first, letting $\theta=0$ in \eqref{eq:pointwise-LZ}, we find that,
for any $\xi\in\mathbb{S}^{n-1}$,
\begin{align*}
\int_{\mathbb{R}^n}|\nabla_\xi (f+g)|^{p-2}
\nabla_\xi(f+g)\nabla_\xi g
&\ge L^{(1)}_\xi(f,g)+\widetilde{L}^{(2)}_\xi(f,g)=:V_0(\xi).
\end{align*}
By this and the definition of the affine Laplacian, we obtain
\begin{align*}
-\int_{\mathbb{R}^n}\Delta_p^{\rm aff}(f+g)g
\ge \mathcal{E}_p(f+g)^{n+p}
\int_{\mathbb{S}^{n-1}}\|\nabla_\xi (f+g)\|_p^{-n-p}
V_0(\xi)\,d\xi.
\end{align*}
Define $a:=\|\nabla_\xi f\|_p^p$
and $U(\xi):=\|\nabla_\xi (f+g)\|_p^p-\|\nabla_\xi f\|_p^p$ for
all $\xi\in\mathbb{S}^{n-1}$;
since $f$ is radial, $a$ is independent of the choice of $\xi$.
Then
\begin{align}\label{prop-expan-lap-e3}
-\int_{\mathbb{R}^n}\Delta_p^{\rm aff}(f+g)g
\ge \left[\int_{\mathbb{S}^{n-1}}(a+U)^{-\frac np}
\right]^{-1-\frac pn}
\int_{\mathbb{S}^{n-1}}(a+U)^{-1-\frac np}V_0.
\end{align}
To apply Proposition \ref{prop-neg-Lebe}\eqref{it1-neg-Lebe},
we now estimate $\|U\|_{L^\infty(\mathbb{S}^{n-1})}.$
From the definition of $\widetilde{L}^{(2)}$ and $p\ge 2$,
we deduce that, for all $\xi\in\mathbb{S}^{n-1}$,
\begin{align*}
\widetilde{L}^{(2)}_\xi(f,g)\ge c_{p}^{-1}\int_{\mathbb{R}^n}
|\nabla_\xi f|^{p-2}(\nabla_\xi g)^2,
\end{align*}
where $c_p$ is as in \eqref{eq:pointwise-LZ}.
This, together with H\"older's inequality and
\eqref{prop-expan-lap-e1},
further implies
\begin{align}\label{prop-expan-lap-e4}
\left[L^{(1)}_\xi (f,g)\right]^2
\le \int_{\mathbb{R}^n}|\nabla_\xi f|^p\int_{\mathbb{R}^n}
|\nabla_\xi f|^{p-2}|\nabla_\xi g|^2\le
c_{p}^{-1}a \widetilde{L}^{(2)}_\xi(f,g)\le c_p^{-1} a^2\varepsilon^2.
\end{align}
Moreover, by Taylor's formula and $p\ge 2$ again, we obtain, for any $x,y\in\mathbb{R}$,
there is a $t\in[0,1]$ such that
\begin{align*}
\left||x+y|^p-|x|^p-p|x|^{p-2}xy\right|
=p(p-1)|x+ty|^{p-2}y^2
\lesssim
|x|^{p-2}y^2+|y|^p,
\end{align*}
where the implicit positive constant depends only on $p$.
Thus, if let $R(\xi):=U(\xi)-pL^{(1)}_\xi(f,g)$
for every $\xi\in\mathbb{S}^{n-1}$, then, from
\eqref{prop-expan-lap-e1} and \eqref{prop-expan-lap-e4},
it follows that, for $\varepsilon\in(0,1)$,
\begin{align}\label{prop-expan-lap-e5}
|R(\xi)|\lesssim
\left[
\widetilde{L}^{(2)}_\xi(f,g)+\|\nabla_\xi g\|_p^p
\right]\lesssim a\varepsilon
\quad\text{and}\quad
|U(\xi)|\lesssim a\varepsilon,
\end{align}
where the implicit positive constants depend only on $p$.
Therefore, for sufficiently small $\varepsilon$,
using \eqref{prop-expan-lap-e3} and Proposition \ref{prop-neg-Lebe} with $\alpha=\frac np$,
we obtain
\begin{align}\label{prop-expan-lap-e6}
-\int_{\mathbb{R}^n}\Delta_p^{\rm aff}(f+g)g
&\ge
|\mathbb{S}^{n-1}|^{-1-\frac pn}
\int_{\mathbb{S}^{n-1}}V_0(\xi)\,d\xi\notag\\
&\quad
-\left(1+\frac np\right)
|\mathbb{S}^{n-1}|^{-1-\frac pn}a^{-1}
\int_{\mathbb{S}^{n-1}}
\left(U-\fint_{\mathbb{S}^{n-1}}U\right)
\left(V_0-\fint_{\mathbb{S}^{n-1}}V_0\right)\notag\\
&\quad
-C_{1}\varepsilon|\mathbb{S}^{n-1}|^{-1-\frac pn}a^{-1}
\left[
\int_{\mathbb{S}^{n-1}}|UV_0|
+
\left(\int_{\mathbb{S}^{n-1}}|U|\right)
\left(\fint_{\mathbb{S}^{n-1}}|V_0|\right)
\right]\notag\\
&=:|\mathbb{S}^{n-1}|^{-1-\frac pn}
\int_{\mathbb{S}^{n-1}}\left[L^{(1)}_\xi(f,g)+\widetilde{L}^{(2)}_\xi(f,g)\right]\,d\xi
-I_1(g)-I_{2}(g),
\end{align}
where $C_{1}$ is a positive constant depending only on $n$ and $p$.

We first deal with $I_1$. Note that
\begin{align*}
&\int_{\mathbb{S}^{n-1}}
\left(U-\fint_{\mathbb{S}^{n-1}}U\right)
\left(V_0-\fint_{\mathbb{S}^{n-1}}V_0\right)d\xi \\
&\quad =
p\int_{\mathbb{S}^{n-1}}
\left[
L^{(1)}_\xi(f,g)
-\fint_{\mathbb{S}^{n-1}}L^{(1)}_u(f,g)\,du
\right]^2\,d\xi \\
&\qquad
+
\int_{\mathbb{S}^{n-1}}
\left[
\widetilde{L}^{(2)}_\xi(f,g)
-\fint_{\mathbb{S}^{n-1}}\widetilde{L}^{(2)}_u(f,g)\,du
\right]\\
&\qquad\quad\times\left\{
pL^{(1)}_\xi(f,g)+R(\xi)
-\fint_{\mathbb{S}^{n-1}}\left[pL^{(1)}_u(f,g)+R(u)\right]\,du\right\}\,d\xi \\
&\qquad
+
\int_{\mathbb{S}^{n-1}}
\left[
R(\xi)-\fint_{\mathbb{S}^{n-1}}R(u)\,du
\right]
\left[
L^{(1)}_\xi(f,g)
-\fint_{\mathbb{S}^{n-1}}L^{(1)}_u(f,g)\,du
\right]d\xi .
\end{align*}
Observe that the first term on the right-hand side
is exactly $p\operatorname{Var}(f,g)$.
We now estimate the last two terms. By \eqref{prop-expan-lap-e4}
and \eqref{prop-expan-lap-e5}, we find that
$|pL^{(1)}_\xi(f,g)+R(\xi)|\lesssim a\varepsilon$,
which, combined with $L^(2)(f,g)\ge 0$, further implies
\begin{align*}
&\left|\int_{\mathbb{S}^{n-1}}
\left[
\widetilde{L}^{(2)}_\xi(f,g)
-\fint_{\mathbb{S}^{n-1}}\widetilde{L}^{(2)}_u(f,g)\,du
\right]
\left\{
pL^{(1)}_\xi(f,g)+R(\xi)
-\fint_{\mathbb{S}^{n-1}}\left[pL^{(1)}_u(f,g)+R(u)\right]\,du\right\}\,d\xi\right| \\
&\quad\lesssim a\varepsilon
\int_{\mathbb{S}^{n-1}}
\left|
\widetilde{L}^{(2)}_\xi(f,g)
-\fint_{\mathbb{S}^{n-1}}\widetilde{L}^{(2)}_u(f,g)\,du
\right|\,d\xi\lesssim a\varepsilon
\int_{\mathbb{S}^{n-1}}\widetilde{L}^{(2)}_\xi(f,g)\,d\xi ,
\end{align*}
where the implicit positive constants depend only on $n$ and $p$.
Similarly, one has
\begin{align*}
&\left|
\int_{\mathbb{S}^{n-1}}
\biggl[
R(\xi)-\fint_{\mathbb{S}^{n-1}}R(u)\,du
\biggr]
\biggl[
L^{(1)}_\xi(f,g)
-\fint_{\mathbb{S}^{n-1}}L^{(1)}_u(f,g)\,du
\biggr]d\xi
\right|\\
&\quad\lesssim a\varepsilon
\int_{\mathbb{S}^{n-1}}
\left|
R(\xi)-\fint_{\mathbb{S}^{n-1}}R(u)\,du
\right|d\xi\lesssim a\varepsilon
\int_{\mathbb{S}^{n-1}}
\left[
\widetilde{L}^{(2)}_\xi(f,g)+\|\nabla_\xi g\|_p^p
\right]d\xi .
\end{align*}
Combining the above estimates and the fact
$\mathcal{E}_p(f)^p=|\mathbb{S}^{n-1}|^{-\frac pn}a$
(since $f$ is radial), we obtain,
for sufficiently small $\varepsilon$,
\begin{align}\label{prop-expan-lap-e7}
|I_1(g)|&\le
(n+p)|\mathbb{S}^{n-1}|^{-1-\frac {2p}n}\operatorname{Var}(f,g)
+C_{2}\varepsilon
\int_{\mathbb{S}^{n-1}}
\left[
\widetilde{L}^{(2)}_\xi(f,g)+\|\nabla_\xi g\|_p^p
\right]d\xi,
\end{align}
where $C_{2}$ is a positive constant depending only on $n$ and $p$.
This is the desired estimate of $I_1$.
On the other hand, using the similar argument and the
good pointwise estimates of $L^{(1)}_\xi$,
$\widetilde{L}^{(2)}_\xi$, and $R$ obtained above,
we can also find that, for sufficiently small $\varepsilon$,
\begin{align*}
|I_2(g)|\lesssim
\varepsilon
\int_{\mathbb{S}^{n-1}}
\left[\widetilde{L}^{(2)}_\xi(f,g)+\|\nabla_\xi g\|_p^p\right]\,d\xi
\end{align*}
with the implicit positive constant depending only on $n$ and $p$.
This, together with \eqref{prop-expan-lap-e6}, \eqref{prop-expan-lap-e7},
and \eqref{eq-lp},
further implies that, for sufficiently small $\varepsilon$,
\begin{align*}
-\int_{\mathbb{R}^n}\Delta_p^{\rm aff}(f+g)g
&\ge
|\mathbb{S}^{n-1}|^{-1-\frac pn}
\int_{\mathbb{S}^{n-1}}L^{(1)}_\xi(f,g)\,d\xi
+(1-C_{3}\varepsilon)|\mathbb{S}^{n-1}|^{-1-\frac pn}
\int_{\mathbb{S}^{n-1}}\widetilde{L}^{(2)}_\xi(f,g)\,d\xi\notag\\
&\quad-(n+p)|\mathbb{S}^{n-1}|^{-1-\frac {2p}n}
\mathcal{E}_p(f)^{-p}\operatorname{Var}(f,g)
-C_{4}\varepsilon\|\nabla g\|_p^p,
\end{align*}
where $C_{3}$ and $C_4$ are positive constants depending only on $n$ and $p$.
Furthermore, since
$f$ is radial, from the fact $\mathcal{E}_p(f)^p=|\mathbb{S}^{n-1}|^{-\frac pn}a$
and the definition of the affine Laplacian, we infer that
\begin{align}\label{prop-expan-lap-e12}
-\int_{\mathbb{R}^n}\Delta_p^{\rm aff}(f)g
&=\mathcal{E}_p(f)^{n+p}\int_{\mathbb{S}^{n-1}}
\|\nabla_\xi f\|_p^{n+p}\int_{\mathbb{R}^n}
|\nabla_\xi f|^{p-2}\nabla_\xi f\nabla_\xi g\notag\\
&=|\mathbb{S}^{n-1}|^{-1-\frac pn}
\int_{\mathbb{S}^{n-1}}L^{(1)}_\xi(f,g)\,d\xi.
\end{align}
Thus, it holds that, for sufficiently small $\varepsilon$,
\begin{align}\label{prop-expan-lap-e9}
-\int_{\mathbb{R}^n}\Delta_p^{\rm aff}(f+g)g
&\ge
-\int_{\mathbb{R}^n}\Delta_p^{\rm aff}(f)g
+(1-C_{3}\varepsilon)|\mathbb{S}^{n-1}|^{-1-\frac pn}
\int_{\mathbb{S}^{n-1}}\widetilde{L}^{(2)}_\xi(f,g)\,d\xi\notag\\
&\qquad-(n+p)|\mathbb{S}^{n-1}|^{-1-\frac {2p}n}
\mathcal{E}_p(f)^{-p}\operatorname{Var}(f,g)
-C_{4}\varepsilon\|\nabla g\|_p^p.
\end{align}

We next apply the same argument with $\theta=1$. In this case,
\eqref{eq:pointwise-LZ} shows that there is a
positive constant $\widetilde{c}_p$, depending only
on $p$, such that
\begin{align*}
\int_{\mathbb{R}^n}|\nabla_\xi (f+g)|^{p-2}
\nabla_\xi(f+g)\nabla_\xi g
\ge
L^{(1)}_\xi(f,g)+\widetilde{c}_p\|\nabla_\xi g\|_p^p.
\end{align*}
Repeating the argument used above with
$V_0$ replaced by $L^{(1)}_\xi(f,g)+\widetilde{c}_p\|\nabla_\xi g\|_p^p$,
one has, for sufficiently small $\varepsilon$,
\begin{align}\label{prop-expan-lap-e10}
-\int_{\mathbb{R}^n}\Delta_p^{\rm aff}(f+g)g
&\ge
-\int_{\mathbb{R}^n}\Delta_p^{\rm aff}(f)g
-C_5\varepsilon|\mathbb{S}^{n-1}|^{-1-\frac pn}
\int_{\mathbb{S}^{n-1}}\widetilde{L}^{(2)}_\xi(f,g)\notag\\
&\qquad-(n+p)|\mathbb{S}^{n-1}|^{-1-\frac {2p}n}\int_{\mathbb{S}^{n-1}}
\mathcal{E}_p(f)^{-p}\operatorname{Var}(f,g)+C_6\|\nabla g\|_p^p,
\end{align}
where $C_{5}$ and $C_6$ are positive constants depending only on $n$ and $p$.
Now, letting $(1-\varepsilon)\times\eqref{prop-expan-lap-e9}+
\varepsilon\times\eqref{prop-expan-lap-e10}$,
we obtain the desired estimate \eqref{prop-expan-lap-e2}.
Then the present proof is complete.
\end{proof}

Moreover, to show the stability of critical points,
we also need the following upper estimate.

\begin{proposition}\label{prop-expan-E}
Let $p\ge 2$. Then there are positive constants $c_{n,p}$ and
$C_{n,p}$,
depending only on $n$ and $p$, such that, for any
$\varepsilon\in(0,c_{n,p})$ and $g\in\dot W^{1,p}$ satisfying
\begin{align}\label{prop-expan-E-e1}
\sup_{\xi\in\mathbb{S}^{n-1}}
\frac{\|\nabla_\xi g\|_{p}}
{\|\nabla_\xi f\|_p}\le \varepsilon,
\end{align}
it holds that
\begin{align}\label{prop-expan-E-e2}
\left|\mathcal{E}_p(f+g)^p
-\mathcal{E}_p(f)
+p\int_{\mathbb{R}^n}\Delta_p^{\rm aff}(f)g\right|
\le C_{n,p}\left(\|\nabla g\|_{p}^p
+\int_{\mathbb{R}^n}
|\nabla f|^{p-2}|\nabla g|^2\right).
\end{align}
\end{proposition}

\begin{proof}
We aim to apply Proposition \ref{prop-neg-Lebe}\eqref{it2-neg-Lebe}.
As in the proof of Proposition \ref{prop-expan-lap},
define $a:=\|\nabla_\xi f\|_p^p$ (which is independent of $\xi$)
and $U(\xi):=\|\nabla_\xi (f+g)\|_p^p-\|\nabla_\xi f\|_p^p$ for
all $\xi\in\mathbb{S}^{n-1}$.
From Taylor's formula, we deduce that, for every $x,y\in\mathbb{R}$,
there is a $t\in[0,1]$ such that
\begin{align*}
\left||x+y|^p-|x|^p\right|
=p|x+ty|^{p-1}|y|\lesssim
(|x|^{p-1}+|y|^{p-1})|y|,
\end{align*}
where the implicit positive constant depends only on $n$ and $p$.
Therefore, by H\"older's inequality and \eqref{prop-expan-E-e1},
it holds that, for any $\xi\in\mathbb{S}^{n-1}$,
\begin{align}\label{prop-expan-E-e3}
|U(\xi)|&\lesssim\int_{\mathbb{R}^n}
\left[|\nabla_\xi f|^{p-1}|\nabla_\xi g|
+|\nabla_\xi g|^p\right]\notag\\
&\le \|\nabla_\xi f\|_p^{p-1}\|\nabla_\xi g\|_p
+\varepsilon^p\|\nabla_\xi f\|_p^{p}
\le(\varepsilon+\varepsilon^p)a
\end{align}
with the implicit positive constant depending only on $n$ and $p$.
Then, using this and Proposition \ref{prop-neg-Lebe}\eqref{it2-neg-Lebe},
we conclude that, if $\varepsilon$ is sufficiently small, then
\begin{align}\label{prop-expan-E-e4}
\left|\mathcal{E}_p(f+g)^p
-\mathcal{E}_p(f)-
|\mathbb{S}^{n-1}|^{-1-\frac pn}
\int_{\mathbb{S}^{n-1}}U(\xi)\,d\xi\right|
&\lesssim a^{-1}\|U\|_{L^\infty(\mathbb{S}^{n-1})}^2.
\end{align}
Since $f$ is radial, we infer that $a=\|\nabla_\xi f\|_p^p\sim\|\nabla f\|_p$.
Combining this, the first inequality in \eqref{prop-expan-E-e3},
and H\"older's inequality,
we obtain,
for sufficiently small $\varepsilon$,
\begin{align}\label{prop-expan-E-e5}
\operatorname{LHS}\eqref{prop-expan-E-e4}
&\lesssim \|\nabla f\|_p^{-p}
\left(\int_{\mathbb{R}^n}
\left[|\nabla_\xi f|^{p-1}|\nabla_\xi g|
+|\nabla_\xi g|^p\right]\right)^{2}\notag\\
&\lesssim \|\nabla f\|_p^{-p}
\left(\|\nabla f\|_p^{p}
\int_{\mathbb{R}^n}|\nabla_\xi f|^{p-2}|\nabla_\xi g|^2+\|\nabla_\xi g\|_p^{2p}\right)\notag\\
&\le\|\nabla g\|_p^p
+\int_{\mathbb{R}^n}|\nabla f|^{p-2}|\nabla g|^2.
\end{align}
On the other hand, using \eqref{prop-expan-lap-e12}, we find that,
for sufficiently small $\varepsilon$,
\begin{align*}
&\left|p\int_{\mathbb{R}^n}\Delta_p^{\rm aff}(f)g
+|\mathbb{S}^{n-1}|^{-1-\frac pn}
\int_{\mathbb{S}^{n-1}}U(\xi)\,d\xi\right|\\
&\quad=\left||\mathbb{S}^{n-1}|^{-1-\frac pn}
\int_{\mathbb{S}^{n-1}}\left[U(\xi)-p
L^{(1)}_\xi(f,g)\right]\,d\xi\right|\\
&\quad\lesssim\int_{\mathbb{S}^{n-1}}
\left[
\widetilde{L}^{(2)}_\xi(f,g)+\|\nabla_\xi g\|_p^p
\right]\,d\xi\quad\text{by \eqref{prop-expan-lap-e5}}\\
&\quad\lesssim\|\nabla g\|_p^p
+\int_{\mathbb{R}^n}|\nabla f|^{p-2}|\nabla g|^2
\quad\text{by \eqref{eq-L2}}.
\end{align*}
This, together with \eqref{prop-expan-E-e5},
gives \eqref{prop-expan-E-e2} and hence completes the present proof.
\end{proof}


\section{Affine spectral gap inequality}\label{sec-spectral}

The main objective of this section is to establish an affine spectral
gap inequality. We first introduce some notation.
Recall that $U(x):=c_0(1+|x|^{p'})^{\frac{p-n}{p}}$
for $x\in\mathbb{R}^n$ satisfying $\|U\|_{p^*}=1$.
Then $U$ satisfies
\begin{align}\label{eq-pU}
-\Delta_pU:=-\operatorname{div}\left(|\nabla U|^{p-2}\nabla U\right)
={S}_{\rm Sob}^{p}U^{p^*-1};
\end{align}
see \cite{a76,t76,cnv04}.

For any positive $V=c_0T_{\lambda_0,S,x_0}U\in\mathcal{M}_{\rm aff}$ with
affine transformation parameters $(c_0,\lambda_0,S,x_0)\in\mathscr{A}$,
define the \emph{weighted Lebesgue space} $L^{2}(V^{p^*-2})$ by setting
\begin{align*}
L^2(V^{p^*-2}):=
\left\{f:\|f\|_{L^{2}(V^{p^*-2})}:=
\left(\int_{\mathbb{R}^n}|f|^2V^{p^*-2}\right)^{\frac12}<\infty\right\}.
\end{align*}
For any
$\varphi,\psi\in L^2(V^{p^*-2})$, we say that
$\varphi\perp \psi$ in $L^2(V^{p^*-2})$ if
\begin{align*}
\langle\varphi,\psi\rangle:=
\int_{\mathbb{R}^n}V^{p^*-2}\varphi\psi=0;
\end{align*}
moreover, define the \emph{tangent space}
$T_V\mathcal{M}_{\rm aff}$ of $\mathcal{M}_{\rm aff}$ at $V$
by setting
\begin{align}\label{eq-tangent}
T_V\mathcal{M}_{\rm aff}:=
\left\{V,\ \partial_i V,\ \frac{n-p}{p}V+x\cdot\nabla V,\
\langle Bx,\nabla V\rangle:i\in\{1,\ldots,n\},\ B=B^{T},\ \operatorname{Tr}B=0\right\}.
\end{align}
Let the \emph{notation} $C_{{\rm c},V}^1$ denote
the set of compactly supported functions of class $C^1$
that are constant in a neighborhood of $x_0$.
Define the \emph{weighted Sobolev space} $\dot{W}^{1,2}(|\nabla V|^{p-2})$
to be the closure of $C_{{\rm c},{V}}^1$ with respect to
\begin{align*}
\|\varphi\|_{\dot{W}^{1,2}(|\nabla V|^{p-2})}
:=\left(\int_{\mathbb{R}^n}|\nabla \varphi|^2|\nabla V|^{p-2}\right)^{\frac12}.
\end{align*}

We now recall the classical spectral gap inequality,
which was established in \cite{fn19,fz22}.
For simplicity, in the remainder of this section, we assume $V=U$,
with the understanding that the orthogonality condition
$\perp$ holds in $L^2(U^{p^*-2})$ and will not be explicitly mentioned hereafter.
Define the \emph{linearized $p$-Laplacian} $\mathcal{L}_U$ by setting,
for any $\varphi\in\dot{W}^{1,2}(|\nabla U|^{p-2})$,
\begin{align*}
\mathcal{L}_U(\varphi):=-\operatorname{div}
\left(|\nabla U|^{p-2}\nabla\varphi+[p-2]
|\nabla U|^{p-4}[\nabla U\cdot\nabla\varphi]\nabla U\right).
\end{align*}
Then one has the following spectral gap inequality for $\mathcal{L}_U$
(see \cite[Proposition 3.6]{fz22}).

\begin{lemma}\label{fz-spectral}
There is a $\delta\in(0,\infty)$,
depending only on $n$ and $p$, such that, for any function $\varphi\in \dot{W}^{1,2}(|\nabla U|^{p-2})$
satisfying $\varphi\perp\{U,\ \partial_i U,\ \frac{n-p}{p}U+x\cdot\nabla U:i=1,\ldots,n\}$,
\begin{align*}
\langle \mathcal{L}_U\varphi,\varphi\rangle
\ge\left[(p^*-1){S}_{\rm Sob}^p+\delta\right]
\int_{\mathbb{R}^n}U^{p^*-2}\varphi^2\,dx.
\end{align*}
\end{lemma}

In this section, we prove the following affine spectral gap inequality.

\begin{theorem}\label{thm-spectral}
There is a positive constant $\delta\in(0,\infty)$,
depending only on $n$ and $p$, such that, for any $\varphi\in \dot{W}^{1,2}(|\nabla U|^{p-2})$
satisfying $\varphi\perp T_U\mathcal{M}_{\rm aff}$,
\begin{align}\label{thm-spectral-e0}
&(p-1)|\mathbb{S}^{n-1}|^{-1-\frac pn}
\int_{\mathbb{S}^{n-1}}\int_{\mathbb{R}^n}
|\nabla_\xi U|^{p-2}(\nabla_\xi\varphi)^2\,dx\,d\xi
-(n+p)|\mathbb{S}^{n-1}|^{-1-\frac{2p}{n}}
\mathcal{E}_p(U)^{-p}\operatorname{Var}(U,\varphi)\notag\\
&\qquad\ge\left[(p^*-1)S_{\rm aff}^p+\delta\right]
\int_{\mathbb{R}^n}U^{p^*-2}\varphi^2\,dx,
\end{align}
where $L^{(1)}_\xi$ is as in \eqref{def-L1} with
$f$ replaced by $U$.
\end{theorem}

To show this, we need the spherical harmonic functions, which
are homogeneous harmonic polynomials restricted to the
sphere. To be exact, for any $\ell\in\mathbb{Z}_+$,
$Y_{\ell}$ is called an \emph{$\ell$ order spherical harmonic function}
if it is $\ell$ homogeneous [that is, for any $\lambda\in(0,\infty)$,
$Y_\ell(\lambda\cdot)=\lambda^\ell Y_\ell(\cdot)$]
polynomial such that
\begin{align}\label{eq-sphere}
-\Delta_{\mathbb S^{n-1}}Y_{\ell,m}
  =\mu_\ell Y_{\ell,m},\quad \mu_\ell:=\ell(\ell+n-2),
\end{align}
where $\Delta_{\mathbb{S}^{n-1}}$ denotes the Laplace--Beltrami operator
on $\mathbb{S}^{n-1}$.
This family of functions forms an orthonormal basis for $L^2(\mathbb{S}^{n-1})$.
Moreover, we denote by the \emph{notation} $\mathcal{H}_\ell(\mathbb{S}^{n-1})
:=\{Y_{\ell,m}\}_m$ the set of all $\ell$ order spherical harmonic functions.
For more properties about spherical harmonic functions,
we refer to \cite{dx13}.
We first establish the following identities.

\begin{lemma}\label{lem-L2}
Let $\alpha_{n,p}$ be as in \eqref{eq-alpha}.
\begin{enumerate}[{\rm(i)}]
  \item\label{lem-L2-it1} For any $\varphi\in\dot{W}^{1,2}(|\nabla U|^{p-2})$,
\begin{align}\label{lem-L2-e0}
\int_{\mathbb{S}^{n-1}}\int_{\mathbb{R}^n}
|\nabla_\xi U|^{p-2}(\nabla_\xi\varphi)^2\,dx\,d\xi
=\frac{\alpha_{n,p}^p}{p-1}\langle\mathcal{L}_U\varphi,\varphi\rangle.
\end{align}
  \item\label{lem-L2-it2} For any $Y_{\ell,m}\in\mathcal{H}_\ell(\mathbb{S}^{n-1})$
and any differentiable function $f$ on $(0,\infty)$,
\begin{align*}
\int_{\mathbb{S}^{n-1}}\int_{\mathbb{R}^n}
|\nabla_\xi U|^{p-2}(\nabla_\xi [fY_{\ell,m}])^2\,dx\,d\xi
=\frac{\alpha_{n,p}^p}{p-1}
\int_{0}^{\infty}|\partial_r U|^{p-2}
r^{n-1}\left[(p-1)f'^2+\frac{\mu_\ell}{r^2}f^2\right]\,dr.
\end{align*}
\end{enumerate}
\end{lemma}

To prove this lemma, we first calculate
a spherical integral explicitly.

\begin{lemma}\label{lem-wu}
If $p\in(1,\infty)$, then,
for any $u\in\mathbb{R}^n$ and $\omega\in\mathbb{S}^{n-1}$,
\begin{align}\label{lem-wu-e0}
\int_{\mathbb{S}^{n-1}}|\omega\cdot\xi|^{p-2}
|u\cdot\xi|^2\,d\xi=\alpha_{n,p}^p
\left[|u\cdot\omega|^2+\frac{|u|^2-|u\cdot\omega|^2}{p-1}\right],
\end{align}
where $\alpha_{n,p}$ is as in \eqref{eq-alpha}.
\end{lemma}

\begin{proof}
For $u:=(u_1,\ldots,u_n)\in\mathbb{R}^n$ and $\omega\in\mathbb{S}^{n-1}$, define
\begin{align*}
F(\omega,u):=\int_{\mathbb{S}^{n-1}}|\omega\cdot\xi|^{p-2}
|u\cdot\xi|^2\,d\xi.
\end{align*}
Choose $A\in O(n)$ such that $A\omega=e_n:=(0,\ldots,0,1)$.
Then, by a change of variables, we conclude that
\begin{align}\label{lem-wu-e1}
F(\omega,u)
&=\int_{\mathbb{S}^{n-1}}
\left|e_n\cdot(A\xi)\right|^{p-2}
\left|(Au)\cdot(A\xi)\right|^2\,d\xi\notag\\
&=\int_{\mathbb{S}^{n-1}}|\xi_n|^{p-2}
|(Au)\cdot\xi|^2\,d\xi=F(e_n,Au).
\end{align}
Therefore, we first consider the case $\omega=e_n$. In this case,
\begin{align}\label{lem-wu-e2}
F(e_n,u)
&=\int_{\mathbb{S}^{n-1}}
|\xi_n|^{p-2}\left(\sum_{i=1}^n u_i\xi_i\right)^2\,d\xi
=\sum_{i=1}^{n-1}u_i^2\int_{\mathbb{S}^{n-1}}
|\xi_n|^{p-2}\xi_i^2\,d\xi
+u_n^2\int_{\mathbb{S}^{n-1}}|\xi_n|^p\,d\xi\notag\\
&=\sum_{i=1}^{n-1}u_i^2\int_{\mathbb{S}^{n-1}}
|\xi_n|^{p-2}\xi_i^2\,d\xi+u_n^2\alpha_{n,p}^p.
\end{align}
Observe that, for any $i\in\{1,\ldots,n-1\}$,
$$\int_{\mathbb{S}^{n-1}}|\xi_n|^{p-2}\xi_i^2\,d\xi=
\int_{\mathbb{S}^{n-1}}|\xi_n|^{p-2}\xi_1^2\,d\xi.$$
Indeed, choose $B\in O(n)$ such that
$Be_i=e_1$ and $Be_n=e_n$; then using a change of variables,
we immediately find that this claim holds.
Hence,
\begin{align*}
(n-1)\int_{\mathbb{S}^{n-1}}|\xi_n|^{p-2}\xi_1^2\,d\xi
&=\sum_{i=1}^{n-1}\int_{\mathbb{S}^{n-1}}|\xi_n|^{p-2}\xi_i^2\,d\xi\\
&=\int_{\mathbb{S}^{n-1}}|\xi_n|^{p-2}(1-\xi_n^2)\,d\xi
=\alpha_{n,p-2}^{p-2}-\alpha_{n,p}^p.
\end{align*}
Since we assume that $p\in(1,\infty)$,
from \eqref{eq-alpha} and the basic property $\Gamma(\alpha+1)=\alpha\Gamma(\alpha)$
of the gamma function, we infer that
$\alpha_{n,p-2}^{p-2}=\frac{n+p-2}{p-1}\alpha_{n,p}^{p}$.
This then implies, for any $i\in\{1,\ldots,n-1\}$,
\begin{align*}
\int_{\mathbb{S}^{n-1}}|\xi_n|^{p-2}\xi_i^2\,d\xi
=\frac{1}{n-1}\left(\alpha_{n,p-2}^{p-2}-\alpha_{n,p}^p\right)
=\frac{1}{p-1}\alpha_{n,p}^p.
\end{align*}
Combining this and \eqref{lem-wu-e2}, we further obtain
\begin{align*}
F(e_n,u)=\alpha_{n,p}^p
\left[u_n^2+\frac{1}{p-1}(|u|^2-u_n^2)\right],
\end{align*}
which shows \eqref{lem-wu-e0} for $u=e_n$.
Furthermore, using this, the fact
$(Au)_n=(Au)\cdot e_n=u\cdot(A^Te_n)=u\cdot\omega$, and
\eqref{lem-wu-e1}, we conclude \eqref{lem-wu-e0} for general $u$.
This completes the proof.
\end{proof}

\begin{proof}[Proof of Lemma \ref{lem-L2}]
We first show \eqref{lem-L2-it1}. Using polar coordinates, we obtain
\begin{align}\label{eq-nabla}
\nabla=\partial_r \omega+\frac1r\nabla_{\mathbb{S}^{n-1}},
\quad\omega=\frac{x}{|x|}\in\mathbb{S}^{n-1}.
\end{align}
Since $U$ is radial, it follows that
$\nabla U=\partial_r U\omega$, $\nabla\varphi=\partial_r\varphi\omega
+\frac1r\nabla_{\mathbb{S}^{n-1}}\varphi$, and
$\omega\perp\nabla_{\mathbb{S}^{n-1}}\varphi$.
This further implies
\begin{align}\label{lem-L2-e1}
\langle\mathcal{L}_U\varphi,\varphi\rangle
&=\int_{\mathbb{R}^n}|\nabla U|^{p-2}
(\nabla\varphi)^2+(p-2)|\nabla U|^{p-4}
(\nabla U\cdot\nabla\varphi)^2\notag\\
&=\int_{0}^{\infty}\int_{\mathbb{S}^{n-1}}
\left\{|\partial_rU|^{p-2}
\left[|\partial_r\varphi|^2+\frac{1}{r^2}
|\nabla_{\mathbb{S}^{n-1}}\varphi|^2\right]
+(p-2)|\partial_r U|^{p-4}(\partial_r U\partial_r\varphi)^2\right\}
r^{n-1}\,d\omega\,dr\notag\\
&=\int_{0}^{\infty}|\partial_r U|^{p-2}r^{n-1}
\int_{\mathbb{S}^{n-1}}\left[(p-1)|\partial_r\varphi|^2+\frac{1}{r^2}
|\nabla_{\mathbb{S}^{n-1}}\varphi|^2\right]\,d\omega\,dr.
\end{align}
Moreover, note that $\nabla_\xi U=\partial_r U(\omega\cdot\xi)$.
Applying polar coordinates for the inner $\int_{\mathbb{R}^n}$ and
applying Lemma \ref{lem-wu}, we obtain
\begin{align*}
\operatorname{LHS}\eqref{lem-L2-e0}
&=\int_{0}^{\infty}|\partial_r U|^{p-2}r^{n-1}
\int_{\mathbb{S}^{n-1}}\int_{\mathbb{S}^{n-1}}
|\omega\cdot\xi|^{p-2}|\nabla\varphi\cdot\xi|^2\,d\omega\,d\xi\,dr\\
&=\alpha_{n,p}^p\int_{0}^{\infty}|\partial_rU|^{p-2}r^{n-1}
\int_{\mathbb{S}^{n-1}}\left(|\partial_r\varphi|^2+\frac{1}{p-1}
\left[|\nabla\varphi|^2-
\left|\nabla\varphi\cdot(\partial_r\varphi\omega)\right|^2\right]\right)\,d\omega\,dr\\
&=\alpha_{n,p}^p\int_{0}^{\infty}|\partial_rU|^{p-2}r^{n-1}
\int_{\mathbb{S}^{n-1}}\left(|\partial_r\varphi|^2+\frac{1}{(p-1)r^2}
|\nabla_{\mathbb{S}^{n-1}}\varphi|^2\right)\,d\omega\,dr\\
&\qquad\qquad\text{by \eqref{eq-nabla} and $\omega\perp\nabla_{\mathbb{S}^{n-1}}\varphi$}\\
&=\frac{\alpha_{n,p}^p}{p-1}\langle\mathcal{L}_U\varphi,\varphi\rangle
\quad\text{by \eqref{lem-L2-e1}}\\
&=\operatorname{RHS}\eqref{lem-L2-e0}.
\end{align*}
This completes the proof of \eqref{lem-L2-it1}.

Once we prove \eqref{lem-L2-it1}, \eqref{lem-L2-it2} is a
direct consequence of the representation of $\mathcal{L}_U$ in polar coordinates
(see \cite[(3.10)]{fn19}).
For the convenience of the reader, we also give some details.
Indeed, if $\varphi=fY_{\ell,m}$, then
$\partial_r\varphi=f'Y_{\ell,m}\omega$ and
$\nabla_{\mathbb{S}_{n-1}}\varphi=f\nabla_{\mathbb{S}^{n-1}}Y_{\ell,m}$.
Combining this, the facts that $\int_{\mathbb{S}^{n-1}}Y_{\ell,m}^2=1$ and
\begin{align*}
\int_{\mathbb{S}^{n-1}}|\nabla_{\mathbb{S}^{n-1}}Y_{\ell,m}|^2
&=-\int_{\mathbb{S}^{n-1}}(\Delta_{\mathbb{S}^{n-1}}Y_{\ell,m})Y_{\ell,m}\\
&=\mu_\ell\int_{\mathbb{S}^{n-1}}Y_{\ell,m}^2\quad\text{by \eqref{eq-sphere}}\\
&=\mu_\ell,
\end{align*}
and \eqref{lem-L2-e1}, we conclude that
\begin{align*}
\langle\mathcal{L}_U(fY_{\ell,m}),fY_{\ell,m}\rangle
=\int_{0}^{\infty}|\partial_r U|^{p-2}
r^{n-1}\left[(p-1)f'^2+\frac{\mu_\ell}{r^2}f^2\right]\,dr.
\end{align*}
This, together with \eqref{lem-L2-it1},
proves \eqref{lem-L2-it2},
which completes the proof of Lemma \ref{lem-L2}.
\end{proof}

Recall that $L^{(1)}_\xi$ is defined as in \eqref{def-L1}
with $f$ replaced by $U$.
We have the following conclusion.

\begin{lemma}\label{lem-L1}
Let $\ell\in\mathbb{Z}_+$.
Then there is a $\beta_{\ell,p}\in\mathbb{R}$
such that
\begin{align}\label{lem-L1-e0}
|\beta_{\ell,p}|=
\begin{cases}
\alpha_{n,p}^p&\text{if $\ell=0$},\\
{\displaystyle\alpha_{n,p}^p
\prod_{j=0}^{k-1}\left|\frac{2j-p}{2j+n+p}\right|}&
\text{if $\ell=2k$ for $k\in\mathbb{N}$},\\
0&\text{otherwise}
\end{cases}
\end{align}
and, for any $Y_{\ell,m}\in\mathcal{H}_\ell(\mathbb{S}^{n-1})$,
any $\xi\in\mathbb{S}^{n-1}$, and any differentiable function $f$ on $(0,\infty)$,
\begin{align}\label{lem-L1-ee}
L^{(1)}_{\xi}(fY_{\ell,m})
=\beta_{\ell,p}\int_{0}^{\infty}|\partial_rU|^{p-2}
\partial_rUr^{n-1}
\left[f'(r)+\frac{\mu_\ell}{pr}f(r)\right]\,dr
Y_{\ell,m}(\xi).
\end{align}
\end{lemma}

\begin{proof}
Let $\varphi:=fY_{\ell,m}$.
Then, by \eqref{eq-nabla}, we find that
$\nabla\varphi=f'Y_{\ell,m}\omega+\frac{f}r\nabla_{\mathbb{S}^{n-1}}Y_{\ell,m}$.
This further implies that, for any $\xi\in\mathbb{S}^{n-1}$,
\begin{align}\label{lem-L1-e1}
L^{(1)}_\xi(\varphi)
&=\int_{0}^{\infty}|\partial_rU|^{p-2}
\partial_rUr^{n-1}f'\int_{\mathbb{S}^{n-1}}
|\omega\cdot\xi|^pY_{\ell,m}\,d\omega\,dr\notag\\
&\quad+\int_{0}^{\infty}|\partial_rU|^{p-2}
\partial_rUr^{n-1}\frac{f}{r}\int_{\mathbb{S}^{n-1}}
|\omega\cdot\xi|^{p-2}(\omega\cdot\xi)\nabla_{\mathbb{S}^{n-1}}
Y_{\ell,m}\cdot\xi\,d\omega\,dr.
\end{align}
Note that, if we consider $|\omega\cdot\xi|^p$ as a function of
$\omega\in\mathbb{R}^n$, then $\nabla|\omega\cdot\xi|^p
=p|\omega\cdot\xi|^{p-2}(\omega\cdot\xi)\xi$.
This, combined with \eqref{eq-nabla} and the fact
$\nabla_{\mathbb{S}^{n-1}}Y_{\ell,m}\perp\omega$, further implies, for any $\xi\in\mathbb{S}^{n-1}$,
\begin{align*}
(\nabla_{\mathbb{S}^{n-1}}|\omega\cdot\xi|^p)
\cdot\nabla_{\mathbb{S}^{n-1}}Y_{\ell,m}
=(\nabla|\omega\cdot\xi|^p)\cdot\nabla_{\mathbb{S}^{n-1}}Y_{\ell,m}
=p|\omega\cdot\xi|^{p-2}(\omega\cdot\xi)\nabla_{\mathbb{S}^{n-1}}Y_{\ell,m}\cdot\xi.
\end{align*}
Hence, from this and integration by parts, we deduce that,
 for any $\xi\in\mathbb{S}^{n-1}$,
\begin{align*}
&\int_{\mathbb{S}^{n-1}}|\omega\cdot\xi|^{p-2}
(\omega\cdot\xi)\nabla_{\mathbb{S}^{n-1}}Y_{\ell,m}\cdot\xi\,d\omega\\
&\quad=\frac{1}{p}\int_{\mathbb{S}^{n-1}}
(\nabla_{\mathbb{S}^{n-1}}|\omega\cdot\xi|^p)
\cdot\nabla_{\mathbb{S}^{n-1}}Y_{\ell,m}\,d\omega
=-\frac1p\int_{\mathbb{S}^{n-1}}|\omega\cdot\xi|^p
\Delta_{\mathbb{S}^{n-1}}Y_{\ell,m}\\
&\quad=\frac{\mu_\ell}{p}\int_{\mathbb{S}^{n-1}}|\omega\cdot\xi|^p
Y_{\ell,m}\,d\omega\quad\text{by \eqref{eq-sphere}}.
\end{align*}
Using this and \eqref{lem-L1-e1}, we obtain
\begin{align*}
L^{(1)}_\xi(\varphi)=\int_{0}^{\infty}|\partial_rU|^{p-2}
\partial_rUr^{n-1}
\left[f'(r)+\frac{\mu_\ell}{pr}f(r)\right]\,dr
\int_{\mathbb{S}^{n-1}}|\omega\cdot\xi|^p
Y_{\ell,m}\,d\omega.
\end{align*}

Therefore, to complete the proof, we only need to
calculate $\int_{\mathbb{S}^{n-1}}|\omega\cdot\xi|^p
Y_{\ell,m}\,d\omega$.
Indeed, for $\ell=0$, since $Y_{0,m}$ is a constant, then
$$\int_{\mathbb{S}^{n-1}}|\omega\cdot\xi|^p
Y_{0,m}\,d\omega=\alpha_{n,p}^pY_{0,m}.$$
Moreover, if $\ell=2k$ for some $k\in\mathbb{N}$,
then, from \cite[p.\,3054]{r26}, we infer that
there is a $\beta_{\ell,p}$ satisfying \eqref{lem-L1-e0}
such that $$\int_{\mathbb{S}^{n-1}}|\omega\cdot\xi|^p
Y_{\ell,m}\,d\omega=\beta_{\ell,p}Y_{\ell,m}.$$
Finally, for an odd number $\ell\in\mathbb{N}$,
applying the Funk--Hecke formula (see \cite[Theorem 1.2.9]{dx13}),
we obtain $\int_{\mathbb{S}^{n-1}}|\omega\cdot\xi|^p
Y_{\ell,m}\,d\omega=\beta_{\ell,p}Y_{\ell,m}$ and
\begin{align*}
\beta_{\ell,p}=|\mathbb{S}^{n-2}|\int_{-1}^{1}
|t|^p\frac{C_\ell^{\frac{n-2}{2}}(t)}{C_\ell^{\frac{n-2}{2}}(1)}
(1-t^2)^{\frac{n-3}{2}}\,dt,
\end{align*}
where $C^{\alpha}_\ell$ is the Gegenbauer polynomial.
Noting that Gegenbauer polynomials have the property:
$C_\ell^\alpha(-t)=(-1)^\ell C_\alpha^\lambda(t)$
(see, for instance, Szeg\"o \cite[p.\,80, (4.7.4)]{s75}),
it then follows that
$\beta_{\ell,p}=0$ when $\ell\ge1$ is odd.
Hence, we complete the proof.
\end{proof}

Therefore, if we expand
$\varphi\in \dot{W}^{1,2}(|\nabla U|^{p-2})$ in terms of spherical harmonics as
\begin{align*}
\varphi=\sum_{\ell=0}^{\infty}\sum_{m}
\left(\int_{\mathbb{S}^{n-1}}\varphi Y_{\ell,m}\right)
Y_{\ell,m}=:\sum_{\ell=0}^{\infty}\sum_{m}\varphi_{\ell,m},
\end{align*}
then, from the orthogonality of $\{Y_{\ell,m}\}$, it follows that
$$\langle \mathcal{L}_U\varphi,\varphi\rangle=\sum_{\ell=0}^{\infty}\sum_{m}
\langle \mathcal{L}_U\varphi_{\ell,m},\varphi_{\ell,m}\rangle.$$
On the other hand, by \eqref{lem-L1-ee} and the
fact that $Y_{\ell,m}$ has zero mean for all $\ell\ge1$, we find that
\begin{align*}
\int_{\mathbb{S}^{n-1}}L^{(1)}_u(\varphi)\,du
=\int_{\mathbb{S}^{n-1}}L^{(1)}_u(\varphi_{0,m})\,du.
\end{align*}
Moreover, $Y_{0,m}$ is a constant. Consequently,
by \eqref{lem-L1-ee} and the orthogonality of $\{Y_{\ell,m}\}$ again,
we have
\begin{align}\label{eq-L1}
\operatorname{Var}(U,\varphi)
&=\sum_{\ell=1}^\infty
\sum_{m}\int_{\mathbb{S}^{n-1}}\left[L^{(1)}_\xi(\varphi_{\ell,m})\right]^2\,d\xi
=\sum_{\ell=0}^{\infty}\sum_{m}\operatorname{Var}(U,
\varphi_{\ell,m}).
\end{align}
This, combined with \eqref{lem-L2-e0}, further implies
\begin{align*}
\operatorname{LHS}\eqref{thm-spectral-e0}
=\sum_{\ell=0}^{\infty}\sum_{m}R_{U,\ell,m}(\varphi),
\end{align*}
where
\begin{align*}
R_{U,\ell,m}(\varphi):=&\,
(p-1)|\mathbb{S}^{n-1}|^{-1-\frac pn}
\int_{\mathbb{S}^{n-1}}\int_{\mathbb{R}^n}
|\nabla_\xi U|^{p-2}(\nabla_\xi\varphi_{\ell,m})^2\,dx\,d\xi\notag\\
&\qquad-(n+p)|\mathbb{S}^{n-1}|^{-1-\frac{2p}{n}}
\mathcal{E}_p(U)^{-p}\operatorname{Var}(U,\varphi_{\ell,m}).
\end{align*}
Thus, to show Theorem \ref{thm-spectral},
we only need to estimate $R_{U,\ell,m}$ for all $\ell$ and $m$.

\begin{proposition}\label{prop-spectral}
Let $\ell\in\mathbb{Z}_+$ and $\varphi\in\dot{W}^{1,2}(|\nabla U|^{p-2})$.
Then there is a positive constant $\delta\in(0,\infty)$,
depending only on $n$ and $p$, such that the following statements hold.
\begin{enumerate}[{\rm(i)}]
  \item\label{it0-spectral} If $\ell=0$ and
  $\varphi\perp\{U,\ \frac{n-p}{p}U+x\cdot\nabla U\}$,
 then
 \begin{align}\label{eq-spectral}
 R_{U,\ell,m}(\varphi)\ge \left[(p^*-1)S_{\rm aff}^p+\delta\right]
 \int_{\mathbb{R}^{n}}U^{p^*-2}\varphi_{\ell,m}^2\,dx.
 \end{align}
 \item\label{it1-spectral} If $\ell=1$ and
 $\varphi\perp\{\partial_i U:i=1,\ldots,n\}$, then \eqref{eq-spectral} holds.
 \item\label{it2-spectral} If $\ell=2$ and
 $\varphi\perp\{\langle Bx,\nabla U\rangle:B=B^{T},\ \operatorname{Tr}B=0\}$,
 then \eqref{eq-spectral} holds.
 \item\label{it3-spectral} If $\ell\ge3$, then
 \eqref{eq-spectral} holds for all $\varphi$.
\end{enumerate}
\end{proposition}

\begin{proof}[Proofs of \eqref{it0-spectral} and \eqref{it1-spectral} of
Proposition \ref{prop-spectral}]
These assertions are
the direct consequence of the classical spectral gap
of $\mathcal{L}_U$ (namely Lemma \ref{fz-spectral}).
Indeed, from \eqref{eq-L1} and Lemma \ref{lem-L1}, we infer that,
for $\ell=0$ and $\ell=1$,
$\operatorname{Var}(U,\varphi_{\ell,m})=0$.
Therefore, by \eqref{lem-L2-e0} and \eqref{eq-relaSob}, we have
\begin{align}\label{it0-spectral-e1}
R_{U,\ell,m}(\varphi)
=|\mathbb{S}_{n-1}|^{-1-\frac pn}\alpha_{n,p}
\langle \mathcal{L}_U(\varphi_{\ell,m}),\varphi_{\ell,m}\rangle
=\left(\frac{S_{\rm aff}}{S_{\rm Sob}}\right)^p
\langle \mathcal{L}_U(\varphi_{\ell,m}),\varphi_{\ell,m}\rangle.
\end{align}

We first consider $\ell=0$.
Note that both $U$ and $\frac{n-p}{p}U+x\cdot\nabla U$ are radial and hence
are orthogonal to $Y_{\ell,m}$ for all $\ell\ge1$. From this and
the orthogonality assumption on $\varphi$ in \eqref{it0-spectral},
it follows that $\varphi_{0,m}\perp\{U,\ \frac{n-p}{p}U+x\cdot\nabla U\}$;
moreover, since $\varphi_{0,m}$ is radial,
we deduce that $\varphi_{0,m}\perp\{\partial_i U:i=1,\ldots,n\}$ automatically.
Combining these, \eqref{it0-spectral-e1}, and Lemma \ref{fz-spectral},
we then obtain \eqref{eq-spectral} for $\ell=0$.

Next, we deal with $\ell=1$. Observe that,
for any $i\in\{1,\ldots,n\}$, $\partial_i U=\nabla U\cdot e_i=\partial_rU(\omega\cdot e_i)$,
which is exactly a constant multiple of some first order spherical harmonic function
on $\mathbb{S}^{n-1}$. Thus, $\partial_i U$ is orthogonal to
$\varphi_{\ell,m}$ for all $\ell\ne 1$.
Combining this and the orthogonality assumption on $\varphi$ in \eqref{it1-spectral},
we find that
$\varphi_{1,m}\perp\{\partial_i U:i=1,\ldots,n\}$;
on the other hand, as mentioned above,
$\varphi_{1,m}\perp\{U,\frac{n-p}{p}U+x\cdot\nabla U\}$ automatically.
Hence, applying \eqref{it0-spectral-e1} and Lemma \ref{fz-spectral} again,
we conclude that \eqref{eq-spectral} with $\ell=1$ holds.
This then completes the proofs of \eqref{it0-spectral}
and \eqref{it1-spectral} of Proposition \ref{prop-spectral}.
\end{proof}

To prove the remaining part of Proposition \ref{prop-spectral},
we still need some additional lemmas. Indeed, from polar coordinates
(see Lemmas \ref{lem-L2} and \ref{lem-L1}), we can
deduce the following radial representation: for all $\ell\in\mathbb{N}$
and $\varphi=fY_{\ell,m}$ with a differentiable $f$ on $(0,\infty)$,
\begin{align}\label{eq-radialR}
R_{U,\ell,m}(\varphi)
&=\left(\frac{S_{\rm aff}}{S_{\rm Sob}}\right)^p
\int_{0}^{\infty}|\partial_r U|^{p-2}r^{n-1}
\left[(p-1)f'^2+\frac{\mu_\ell}{r^2}f^2\right]\,dr\notag\\
&\quad-b_{\ell,p}\left(\int_{0}^{\infty}
|\partial_r U|^{p-2}\partial_r U
r^{n-1}\left[f'+\frac{\mu_\ell}{pr}f\right]\,dr\right)^2,
\end{align}
where $b_{\ell,p}:=(p+n)|\mathbb{S}^{n-1}|^{-1-\frac{2p}{n}}\mathcal{E}_p(U)^{-p}
\beta_{\ell,p}^2$ and $\beta_{\ell,p}$ is as in Lemma \ref{lem-L1}.
Moreover, as $U$ is radial, we can regard $U$ as
a function on $(0,\infty)$ and write $\partial_r U$ simply as $U'$ in what follows for convenience.
Define the \emph{operator} $L^{(\ell)}_U$ on $L^2(U^{p^*-2}r^{n-1},(0,\infty))$
by setting, for any differentiable $f$ on $(0,\infty)$,
\begin{align*}
L_{U}^{(\ell)}(f):=
\left(\frac{S_{\rm aff}}{S_{\rm Sob}}\right)^p
U^{2-p^*}\left\{-(p-1)r^{1-n}\left(f'|U'|^{p-2}r^{n-1}\right)'
+\frac{\mu_\ell}{r^2}|U'|^{p-2}f\right\}
-(p^*-1)S_{\rm aff}^p f.
\end{align*}
This is exactly a multiple of the Sturm--Liouville type operator
appearing in the classical spectral analysis of $\mathcal{L}_U$
(see \cite{fn19,fz22}).
Indeed, we need the following properties of $L^{(0)}_U$;
see \cite[Proposition 3.1]{fn19} for $p\ge 2$ and see
\cite[Appendix B]{fz22} for other cases.

\begin{lemma}\label{lem-fzspe}
For any $p\in(1,n)$,
it holds that
\begin{align*}
L^{(0)}_U(U)=-(p^*-p)S_{\rm aff}^pU\quad\text{and}\quad
L^{(0)}_U\left(\frac{n-p}{p}U+rU'\right)=0.
\end{align*}
\end{lemma}

In addition, if we define
\begin{align*}
\psi_\ell:=U^{2-p^*}
\left[-r^{1-n}\left(|U'|^{p-2}U'r^{n-1}\right)'
+\frac{\mu_\ell}{pr}|U'|^{p-2}U'\right],
\end{align*}
then, by integration by parts, we find that
\begin{align*}
\left(\int_{0}^{\infty}
|\partial_r U|^{p-2}\partial_r U
r^{n-1}\left[f'+\frac{\mu_\ell}{pr}f\right]\,dr\right)^2
=\langle f,\psi_\ell\rangle_0^2,
\end{align*}
here and thereafter, for any $f,g\in L^2(U^{p^*-2}r^{n-1},(0,\infty))$,
define
\begin{align*}
\langle f,g\rangle_0:=\int_{0}^{\infty}fgU^{p^*-2}r^{n-1}.
\end{align*}
Applying Lemma \ref{lem-fzspe}, we
immediately obtain the following property of $L^{(\ell)}_U$
for $\ell\ge 1$, which is important to show
the remaining part of Proposition \ref{prop-spectral}.

\begin{lemma}\label{lem-psiz}
For any $\ell\in\mathbb{N}$,
\begin{align*}
L^{(\ell)}_U(rU')
=p\left(\frac{S_{\rm aff}}{S_{\rm Sob}}\right)^p
\psi_\ell.
\end{align*}
\end{lemma}

\begin{proof}
Using Lemma \ref{lem-fzspe} and a direct calculation, we find that,
for any $\ell\in\mathbb{N}$,
\begin{align}\label{lem-psiz-e1}
L^{(\ell)}_{U}(rU')&=
L^{(0)}_U(rU')+\left(\frac{S_{\rm aff}}{S_{\rm Sob}}\right)^p
\left(U^{2-p^*}\frac{\mu_\ell}{r}|U'|^{p-2}U'\right)\notag\\
&=p\left(\frac{S_{\rm aff}}{S_{\rm Sob}}\right)^p
U^{2-p^*}\left({S}_{\rm Sob}^pU^{p^*-1}+\frac{\mu_\ell}{pr}|U'|^{p-2}
U'\right).
\end{align}
On the other hand, \eqref{eq-pU} and \eqref{eq-nabla} prove that
\begin{align*}
{S}_{\rm Sob}^pU^{p^*-1}&=-\operatorname{div}
\left(|\nabla U|^{p-2}\nabla U\right)\\
&=-\left(\partial_r\omega+\frac{1}{r}\nabla_{\mathbb{S}^{n-1}}\right)
\cdot\left(|U'|^{p-2}U'\right)\\
&=\left(|U'|^{p-2}U'\right)'+\frac{n-1}{r}|U'|^{p-2}U'
\quad\text{by $-\operatorname{div}_{\mathbb{S}^{n-1}}\omega
:=\nabla_{\mathbb{S}^{n-1}}\cdot\omega=n-1$}\\
&=r^{1-n}\left(|U'|^{p-2}U'r^{n-1}\right)'.
\end{align*}
Combining this and \eqref{lem-psiz-e1},
we complete the present proof.
\end{proof}

We also need the following identity.

\begin{lemma}\label{lem-psizin}
For any $\ell\in\mathbb{N}$,
\begin{align*}
\langle rU',\psi_\ell\rangle_0
=\left(1+\frac{\mu_\ell-n}{p}\right)
|\mathbb{S}^{n-1}|^{-1} S^{p}_{\rm Sob}.
\end{align*}
\end{lemma}

\begin{proof}
Noting that $(rU')'=U'+rU''$, we have
\begin{align*}
\langle rU',\psi_\ell\rangle_0
&=\int_{0}^{\infty}
|U'|^{p-2}U'r^{n-1}
\left[\left(1+\frac{\mu_\ell}{p}U'+rU''\right)\right]\,dr\\
&=\left(1+\frac{\mu_\ell}{p}\right)
\int_{0}^{\infty}|U'|^{p}r^{n-1}\,dr+\int_{0}^{\infty}
|U'|^{p-2}U'U''r^n\,dr.
\end{align*}
Since $|U'|^{p-2}U'U''=\frac1p(|U'|^p)'$, from
$|U'|^p(0)=|U'|^p(\infty)=0$ and integration by parts,
we infer that
\begin{align*}
\int_{0}^{\infty}|U'|^{p-2}U'U''r^{n}\,dr
=-\frac np\int_{0}^{\infty}
|U'|^{p}r^{n-1}\,dr.
\end{align*}
Hence, by the facts that $U$ is radial and $U$ is an extremal function of
the classical Sobolev inequality (that is, $\|\nabla U\|_p={S}_{\rm Sob}\|U\|_{p^*}$),
we further obtain
\begin{align*}
\langle rU',\psi_\ell\rangle_0
=\left(1+\frac{\mu_\ell-n}{p}\right)
\int_{0}^{\infty}|U'|^{p}r^{n-1}\,dr
=\left(1+\frac{\mu_\ell-n}{p}\right)
|\mathbb{S}^{n-1}|^{-1}{S}^{p}_{\rm Sob}.
\end{align*}
This completes the present proof.
\end{proof}

Also, we need the following lemma to deal with
the rank-one operator in \eqref{eq-radialR}.

\begin{lemma}\label{lem-rankone}
Let $\mathcal{H}$ be a real Hilbert space, $Q$ be a symmetric
bilinear form on $\mathcal{H}$, and $L\in\mathcal{H}^*$.
If, for any $f\in\mathcal{H}$,
\begin{align*}
Q(f,f)\ge L(f)^2,
\end{align*}
where the equality holds if $f=f_0$,
and there is a $\delta\in(0,\infty)$ such that, for any $f\in \mathcal{H}$,
$Q(f,f)\ge\delta\|f\|_{\mathcal{H}}^2$,
then, for any $f$ satisfying $f\perp f_0$,
\begin{align*}
Q(f,f)\ge L(f)^2+\delta\|f\|_{\mathcal{H}}^2.
\end{align*}
\end{lemma}

\begin{proof}
We first claim that, for any $f\in\mathcal{H}$,
$Q(f,f_0)=L(f)L(f_0)$.
Indeed, using the assumption in the present lemma, we obtain,
for any $t\in\mathbb{R}$ and $f\in\mathcal{H}$,
\begin{align*}
0\le Q(tf+f_0,tf+f_0)-|L(tf+f_0)|^2
=t^{2}\left[Q(f,f)-|L(f)|^2\right]
+2t\left[Q(f,f_0)-L(f)L(f_0)\right].
\end{align*}
Hence, from the arbitrariness of $t\in\mathbb{R}$,
we infer the above claim.
Now, for any $f,g\in\mathcal{H}$, define
$G(f,g):=Q(f,g)-\delta\langle f,g\rangle_{\mathcal{H}}$.
Then, using the assumption in the present lemma and
Cauchy--Schwarz inequality for bilinear forms, we conclude that,
for any $f\in\mathcal{H}$,
\begin{align}\label{lem-rankone-e1}
G(f,f)G(f_0,f_0)\ge G(f,f_0)^2.
\end{align}
Note that $G(f_0,f_0)=Q(f_0,f_0)-\delta\|f_0\|_{\mathcal{H}}^2
\le Q(f_0,f_0)=L(f_0)^2$
and, by the above claim, it holds that, for any $f\perp f_0$,
\begin{align*}
G(f,f_0)^2=Q(f,f_0)^2=L(f)^2L(f_0)^2.
\end{align*}
Combining these and \eqref{lem-rankone-e1},
we then complete the proof.
\end{proof}

Now, we are able to show \eqref{it2-spectral} and \eqref{it3-spectral}
of Proposition \ref{prop-spectral}.

\begin{proof}[Proof of Proposition \ref{prop-spectral}\eqref{it2-spectral}]
We first analyse the operator $L^{(2)}_U$.
Since many of the arguments can be generalized to arbitrary $\ell \ge 2$
and will be needed later, we choose to present them in a general form here.
Indeed, from the spectral analysis of $L^{(\ell)}_U$
in \cite{fn19,fz22}, we infer that the first eigenvalue
of $L^{(1)}_U$ is $0$ and hence the first eigenvalue of $L^{(\ell)}_U$,
$\ell\ge 2$, is positive.
Therefore, using the Rayleigh quotient characterization,
we conclude that there is a
$\delta\in(0,\infty)$, depending only on $n$ and $p$, such that,
for all $\ell\ge2$ and $f\in L^2(U^{p^*-2}r^{n-1},(0,\infty))$,
\begin{align}\label{it2-spectral-e1}
\langle L^{(\ell)}_Uf,f\rangle_0\ge \delta\langle f,f\rangle_0.
\end{align}
Hence, $L^{(\ell)}_U$ is invertible and,
if we apply the Cauchy--Schwarz inequality on
the inner product $\langle\cdot,\cdot\rangle^\sim
:=\langle L^{(\ell)}_U\cdot,\cdot\rangle_0$, then we find that,
for all $f,g\in L^2(U^{p^*-2}r^{n-1},(0,\infty))$,
\begin{align}\label{it2-spectral-e2}
\langle L^{(\ell)}_Uf,g\rangle_0^2\le
\langle L^{(\ell)}_Uf,f\rangle_0\langle L^{(\ell)}_Ug,g\rangle_0.
\end{align}
Letting $g:=(L^{(\ell)}_U)^{-1}\psi_\ell$ and
using Lemmas \ref{lem-psiz} and \ref{lem-psizin}, we obtain, for all
$f\in L^2(U^{p^*-2}r^{n-1},(0,\infty))$,
\begin{align}\label{it2-spectral-e3}
\langle f,\psi_\ell\rangle_0^2&\le
\langle \psi_\ell,(L^{(\ell)}_U)^{-1}\psi_\ell\rangle_0
\langle L^{(\ell)}_Uf,f\rangle_0\notag\\
&=\frac{\mu_\ell+p-n}{p^2}\left(\frac{{S}_{\rm Sob}}{S_{\rm aff}}\right)^p
|\mathbb{S}^{n-1}|^{-1}{S}_{\rm Sob}^p
\langle L^{(\ell)}_Uf,f\rangle_0.
\end{align}
In particular, if $\ell=2$, then, from Lemma \ref{lem-L1} and
\eqref{eq-relaSob}, we infer that
$$\frac{\mu_\ell+p-n}{p^2}\left(\frac{{S}_{\rm Sob}}{S_{\rm aff}}\right)^p
|\mathbb{S}^{n-1}|^{-1}{S}_{\rm Sob}^p=b_{2,p}^{-1}.$$
This shows that, for all $f\in L^2(U^{p^*-2}r^{n-1},(0,\infty))$,
\begin{align}\label{it2-spectral-e4}
\langle L^{(2)}_Uf,f\rangle_0\ge b_{2,p}
\langle f,\psi_2\rangle_0^2.
\end{align}
Moreover, by Lemma \ref{lem-psiz}, we find that
$g:=(L^{(\ell)}_U)^{-1}\psi_\ell$ and $rU'$ are linearly dependent.
Thus, from the equality condition for the Cauchy--Schwarz inequality
\eqref{it2-spectral-e2}, it follows that
the `$\ge$' in \eqref{it2-spectral-e4} becomes an `$=$'
when $f=rU'$.
Using this, \eqref{it2-spectral-e1}, \eqref{it2-spectral-e4},
and Lemma \ref{lem-rankone}, we further conclude that,
if $f\perp rU'$ in  $L^2(U^{p^*-2}r^{n-1},(0,\infty))$,
then
\begin{align}\label{it2-spectral-e5}
\langle L^{(2)}_Uf,f\rangle_0\ge \delta\langle f,f\rangle_0
+b_{2,p}\langle f,\psi_2\rangle_0^2.
\end{align}

Now, we turn to $R_{U,2,m}$. Observe that,
for any symmetric matrix $B$ satisfying $\operatorname{Tr}B=0$,
$\langle Bx,\nabla U\rangle=rU'\langle B\omega,\omega\rangle$,
while $\langle B\omega,\omega\rangle$ is precisely the
form of the second order spherical harmonic function.
Consequently, for any $\ell\ne 2$, $\varphi_{\ell,m}$
is orthogonal to $\{\langle Bx,\nabla U\rangle:B=B^{T},\ \operatorname{Tr}B=0\}$
automatically. This, combined with the
orthogonality assumption on $\varphi$ in \eqref{it2-spectral},
further implies that $\int_{\mathbb{S}^{n-1}}\varphi Y_{\ell,m}$
is orthogonal to $rU'$.
Applying this and \eqref{it2-spectral-e5},
we find that \eqref{eq-spectral} with $\ell=2$ holds.
This completes the present proof.
\end{proof}

\begin{proof}[Proof of Proposition \ref{prop-spectral}\eqref{it3-spectral}]
By Lemma \ref{lem-L1}, when $\ell$ is odd, $b_{\ell,p}=0$.
Hence, if $\ell\ge 3$ is odd, then \eqref{it2-spectral-e1}
is equivalent to \eqref{eq-spectral},
which completes the proof when $\ell$ is odd.
Next, we consider the case where $\ell\ge 4$ is even.
In this case, from \eqref{it2-spectral-e3}, it follows that,
for all $f\in L^2(U^{p^*-2}r^{n-1},(0,\infty))$,
\begin{align*}
\langle L^{(\ell)}_Uf,f\rangle_0
-b_{\ell,p}
\langle f,\psi_\ell\rangle_0^2
\ge
\left[1-\frac{\mu_\ell+p-n}{p^2}
\left(\frac{{S}_{\rm Sob}}{S_{\rm aff}}\right)^p
|\mathbb{S}^{n-1}|^{-1}{S}_{\rm Sob}^p
b_{\ell,p}\right]\langle L^{(\ell)}_Uf,f\rangle_0.
\end{align*}
Therefore, if we can prove that, for all even $\ell\ge 4$,
\begin{align*}
\lambda_{\ell,p}:=
\frac{\mu_\ell+p-n}{p^2}
\left(\frac{{S}_{\rm Sob}}{S_{\rm aff}}\right)^p
|\mathbb{S}^{n-1}|^{-1}{S}_{\rm Sob}^p
b_{\ell,p}
\end{align*}
has a upper bound strictly smaller than $1$,
then, applying \eqref{it2-spectral-e1}, we find that there is a
$\widetilde{\delta}\in(0,\infty)$, depending only on $n$ and $p$,
such that, for all $f\in L^2(U^{p^*-2}r^{n-1},(0,\infty))$,
\begin{align*}
\langle L^{(\ell)}_Uf,f\rangle_0-b_{\ell,p}\langle f,\psi_\ell\rangle_0^2
\ge \widetilde{\delta}\langle f,f\rangle_0,
\end{align*}
which is equivalent to \eqref{eq-spectral} and hence
completes the proof.

Now, we estimate $\lambda_{\ell,p}$. Indeed,
by \eqref{eq-relaSob} and Lemma \ref{lem-L1}, we have,
for any $k\in\mathbb{N}$,
\begin{align*}
\lambda_{2k,p}=\frac{n+p}{p^2}(\mu_{2k}+p-n)
\prod_{j=0}^{k-1}\left(\frac{2j-p}{2j+n+p}\right)^2.
\end{align*}
Then
\begin{align*}
\frac{\lambda_{2k+2,p}}{\lambda_{2k,p}}-1
&=\frac{\mu_{2k+2}+p-n}{\mu_{2k}+p-n}
\left(\frac{2k-p}{2k+n+p}\right)^2-1\\
&=\frac{(4k+n)[n(n+p) - 2k(2k+n)(n+2p-2)]}
{(\mu_{2k}+p-n)(2k+n+p)^2}<0.
\end{align*}
Therefore, $\{\lambda_{2k,p}\}_{k\ge1}$ is strictly decreasing and hence,
for any $k\ge 2$,
$\lambda_{2k,p}< \lambda_{2,p}=1.$
This is the desired upper bound of $\lambda_{\ell,p}$ and
hence completes the present proof.
\end{proof}

Finally, applying Theorem \ref{thm-spectral}, we obtain
the following weighted spectral gap inequality, which
matches the expansions given in Subsection \ref{subsec-expan} well
and plays a key role in the proofs of Theorems \ref{thm-stab}
and \ref{thm-bubble} later.

\begin{theorem}\label{thm-spectral-wei}
Let $p\in[2,n)$ and $\delta$ be as in Theorem \ref{thm-spectral}.
If $C_0\in(0,\infty)$ and $\gamma\in(0,\delta)$, then there is a
$\rho_0\in(0,\infty)$ such that, for any $\rho\in(0,\rho_0)$ and
$\varphi\in\dot{W}^{1,p}$
satisfying $\|\varphi\|_{\dot{W}^{1,p}}<\rho$ and
$\varphi\perp T_U\mathcal{M}_{\rm aff}$,
\begin{align*}
&(1-C_0\rho)|\mathbb{S}^{n-1}|^{-1-\frac pn}
\int_{\mathbb{S}^{n-1}}\widetilde{L}^{(2)}_\xi(U,\varphi)
\,d\xi-(n+p)|\mathbb{S}^{n-1}|^{-1-\frac{2p}{n}}
\mathcal{E}_p(U)^{-p}\operatorname{Var}(U,\varphi)\notag\\
&\qquad\ge\left[(p^*-1)S_{\rm aff}^p+\gamma\right]
\int_{\mathbb{R}^n}U^{p^*-2}\varphi^2\,dx
+\rho\int_{\mathbb{R}^n}
|\nabla U|^{p-2}(\nabla\varphi)^2,
\end{align*}
where $\operatorname{Var}$ and $\widetilde{L}^{(2)}$ are as,
respectively, in \eqref{def-var} and \eqref{eq-L2w}.
\end{theorem}

\begin{remark}\label{rem-spectral}
If $L^{(2)}$ is as in \eqref{def-L2}, then, from the obvious pointwise inequality
$L^{(2)}_\xi\ge \widetilde{L}^{(2)}_\xi$ for all
$\xi\in\mathbb{S}^{n-1}$, we infer that the conclusion of
Theorem \ref{thm-spectral-wei} also holds with
$\widetilde{L}^{(2)}$ replaced by $L^{(2)}$.
\end{remark}

\begin{proof}[Proof of Theorem \ref{thm-spectral-wei}]
We show this theorem by contradiction.
If this theorem does not hold, then there are a sequence
$\{\rho_j\}_{j\in\mathbb{N}}$ in $(0,\frac1{2C_0})$
such that $\lim_{j\to\infty}\rho_j=0$ and a sequence
$\{\varphi_j\}_{j\in\mathbb{N}}$ in
$\dot{W}^{1,p}$ such that, for any $j\in\mathbb{N}$,
$\|\varphi_j\|_{\dot{W}^{1,p}}<\rho_j$,
$\varphi_j\perp T_U\mathcal{M}_{\rm aff}$,
but
\begin{align}\label{eq-spectral-1}
&(1-C_0\rho_j)|\mathbb{S}^{n-1}|^{-1-\frac pn}
\int_{\mathbb{S}^{n-1}}\widetilde{L}^{(2)}_\xi(U,\varphi_j)
\,d\xi-(n+p)|\mathbb{S}^{n-1}|^{-1-\frac{2p}{n}}
\mathcal{E}_p(U)^{-p}\operatorname{Var}(U,\varphi_j)\notag\\
&\qquad<\left[(p^*-1)S_{\rm aff}^p+\gamma\right]
\int_{\mathbb{R}^n}U^{p^*-2}\varphi_j^2\,dx
+\rho_j\int_{\mathbb{R}^n}
|\nabla U|^{p-2}(\nabla\varphi_j)^2.
\end{align}
For any $j\in\mathbb{N}$, let
\begin{align*}
\varepsilon_j:=
\sqrt{(1-C_0\rho_j)|\mathbb{S}^{n-1}|^{-1-\frac pn}
\int_{\mathbb{S}^{n-1}}\widetilde{L}^{(2)}_\xi(U,\varphi_j)}.
\end{align*}
Then $\varepsilon_j\in(0,\infty)$. Indeed, $\varepsilon_j<\infty$
is easily concluded by \eqref{eq-L2}.
On the other hand, since $p\ge 2$,
we deduce that, for any
$\xi\in\mathbb{S}^{n-1}$,
\begin{align}\label{eq-spectral-2}
\widetilde{L}^{(2)}_{\xi}(U,\varphi_j)
\gtrsim\int_{\mathbb{R}^n}|\nabla_\xi U|^{p-2}
|\nabla_\xi\varphi_j|^2;
\end{align}
combining this and Lemmas \ref{lem-L2} and
\ref{fz-spectral}, we further obtain
\begin{align*}
\varepsilon_j^2
\gtrsim\sqrt{1-C_0\rho_j}
\langle\mathcal{L}_U\varphi_j,\varphi_j\rangle
\gtrsim\sqrt{1-C_0\rho_j}\int_{\mathbb{R}^n}U^{p^*-2}\varphi_j^2.
\end{align*}
Hence, if $\varepsilon_j=0$, then $\varphi_j=0$, which contradicts
\eqref{eq-spectral-1}.

Since $\varepsilon_j\in(0,\infty)$, for any $j\in\mathbb{N}$, we can define
$\widehat{\varphi}_j:=\frac{\varphi_j}{\varepsilon_j}$.
Then, for every $j\in\mathbb{N}$,
by Lemma \ref{lem-L2}\eqref{lem-L2-it1}, the assumptions $p\ge 2$
and $\rho_j<\frac{1}{2C_0}$, and \eqref{eq-spectral-2}, we find that
\begin{align*}
\|\widehat{\varphi}_j\|_{\dot{W}^{1,2}(|\nabla U|^{p-2})}^2
&\lesssim\int_{\mathbb{S}^{n-1}}\int_{\mathbb{R}^{n-1}}
|\nabla_\xi U|^{p-2}(\nabla_\xi \widehat{\varphi}_j)^2
\lesssim\frac{1}{1-C_0\rho_j}\lesssim1 .
\end{align*}
This, together with the compact embedding between
$\dot{W}^{1,2}(|\nabla U|^{p-2})$ and $L^2(U^{p^*-2})$
(see \cite[Proposition 3.2]{fz22}), implies that
there are a
subsequence, still denoted by $\{\widehat{\varphi}_{j}\}_{j\in\mathbb{N}}$,
and a function $\varphi_0\in\dot{W}^{1,2}(|\nabla U|^{p-2})$
such that, when $j\to\infty,$
$\widehat{\varphi}_{j}\rightharpoonup\varphi_0$
in $\dot{W}^{1,2}(|\nabla U|^{p-2})$ and
$\widehat{\varphi}_{j}\to\varphi_0$ in $L^2(U^{p^*-2})$.
Moreover, we have the following two observations
on $\{\varphi_j\}_{j\in\mathbb{N}}$.

First, since $\|\nabla\varphi_j\|_p\to0$ as $j\to\infty$,
we can find a subsequence $\{\varphi_{j_k}\}_{k\in\mathbb{N}}$
such that $|\nabla\varphi_{j_k}|\to0$ almost everywhere
as $k\to\infty$.
Second, for a fixed $R\in(1,\infty)$, if we define
$K_R:=B({\bf0},R)\setminus B({\bf0},\frac1R)$,
then, from the structure of $U$, it follows that, for any
$j\in\mathbb{N}$,
\begin{align*}
\int_{K_R}|\nabla\widehat{\varphi}_j|^2
\sim\int_{K_R}|\nabla U|^{p-2}|\nabla\widehat{\varphi}_j|^2
\le \|\widehat{\varphi}_j\|_{\dot{W}^{1,2}(|\nabla U|^{p-2})}
\lesssim 1,
\end{align*}
where the implicit positive constants
may depend on $R$ but are independent of $j$.
Consequently, by the reflexive of $L^2(K_R;\mathbb{R}^n)$,
there are a sequence $\{\widehat{\varphi}_{j_k}\}_{k\in\mathbb{N}}$
and a function $u\in L^2(K_R;\mathbb{R}^n)$ such that
$\nabla\widehat{\varphi}_{j_k}\rightharpoonup
u$ in $L^2(K_R;\mathbb{R}^n)$ as $k\to\infty$.
This, combined with the fact
$\widehat{\varphi}_{j}\to\varphi_0$ in $L^2(U^{p^*-2})$,
further implies that, for any $\Phi\in C_{\rm c}^1(K_R;\mathbb{R}^n)$,
\begin{align*}
\int_{K_R}u\cdot\Phi=\lim_{k\to\infty}\int_{K_R}\nabla\widehat{\varphi}_{j_k}\cdot\Phi
=-\lim_{k\to\infty}\int_{K_R}\widehat{\varphi}_{j_k}\operatorname{div}
\Phi=-\int_{K_R}\varphi_0\operatorname{div}\Phi.
\end{align*}
This shows $u=\nabla\varphi_0$ and
hence $\nabla\widehat{\varphi}_{j_k}\rightharpoonup
\nabla\varphi_0$ in $L^2(K_R;\mathbb{R}^n)$ as $k\to\infty$.

In summary, for a fixed $R\in(1,\infty)$, we can choose a subsequence, still denoted by
$\{\varphi_j\}_{j\in\mathbb{N}}$, such that, when
$j\to\infty$,
\begin{enumerate}[{\rm(i)}]
  \item\label{it1-varphi} $\widehat{\varphi}_{j}\rightharpoonup\varphi_0$
in $\dot{W}^{1,2}(|\nabla U|^{p-2})$ and
$\widehat{\varphi}_{j}\to\varphi_0$ in $L^2(U^{p^*-2})$;
  \item\label{it2-varphi} $|\nabla\varphi_j|\to0$ almost everywhere in $\mathbb{R}^n$;
  \item\label{it3-varphi} $\nabla\widehat{\varphi}_{j}\rightharpoonup
\nabla\varphi_0$ in $L^2(K_R;\mathbb{R}^n)$.
\end{enumerate}
Note that, for any $j\in\mathbb{N}$,
\begin{align}\label{eq-spectral-3}
&\frac{(1-C_0\rho_j)\int_{\mathbb{S}^{n-1}}\widetilde{L}^{(2)}_{\xi}(U,\varphi_j)\,d\xi}{\varepsilon_j^2}\notag\\
&\quad\ge(1-C_0\rho_j)\int_{\mathbb{S}^{n-1}}\int_{\mathbb{R}^n}
\omega_{j,\xi}^{(1)}|\nabla_\xi\widehat{\varphi}_j|^2\notag\\
&\quad\quad+(p-2)(1-C_0\rho_j)\int_{\mathbb{S}^{n-1}}
\int_{\mathbb{R}^n}\omega_{j,\xi}^{(2)}\left[
\frac{|\nabla_\xi(U+\varphi_j)|-|\nabla_\xi U|}{\varepsilon_j}\right]^2\notag\\
&\quad=:{\rm I}_j+(p-2){\rm II}_j,
\end{align}
where $$\omega_{j,\xi}^{(1)}:=\frac{\min\{|\nabla_\xi U|,\max\{c_p^{\frac{1}{p-1}}
|\nabla_\xi (U+\varphi_j)|\}\}^{p-1}}
{|\nabla_\xi U|}$$
and
$$
\omega_{j,\xi}^{(2)}:=\frac{\min\{|\nabla_\xi (U+\varphi_j)|,|\nabla_\xi U|\}^{p-1}}
{|\nabla_\xi U|}$$
with $c_p$ as in \eqref{eq:pointwise-LZ}. We now deal with the terms
${\rm I}_j$ and ${\rm II}_j$ respectively. Since their underlying arguments are largely analogous but ${\rm II}_j$ is structurally more complicated, we shall present the detailed estimates
for ${\rm II}_j$ first, and then briefly point out the necessary modifications for ${\rm I}_j$.
To this end, define $\Psi_j:=\widehat{\varphi}_j-\varphi_0$.
Then, from the Newton--Leibniz formula, we infer that,
for any $j\in\mathbb{N}$,
\begin{align*}
{\rm II}_{j}
&=(1-C_0\rho_j)\int_{\mathbb{S}^{n-1}}
\int_{K_R}\omega_{j,\xi}^{(2)}
 \left[\int_{0}^{1}\frac{\nabla_\xi(U+t\varphi_j)}
{|\nabla_\xi(U+t\varphi_j)|}\nabla_\xi\widehat{\varphi}_j\,dt\right]^2\\
&\ge (1-C_0\rho_j)\int_{\mathbb{S}^{n-1}}
\int_{K_R}\omega_{j,\xi}^{(2)} \left[\int_{0}^{1}\frac{\nabla_\xi(U+t\varphi_j)}
{|\nabla_\xi(U+t\varphi_j)|}\,dt\right]^2
\left[(\nabla_\xi\varphi_0)^2
+2\nabla_\xi\varphi_0\nabla_\xi\Psi_j\right]\\
&=:{\rm II}_{j,1}+2{\rm II}_{j,2}.
\end{align*}
Applying \eqref{it2-varphi} and
the Lebesgue domination convergence theorem, we conclude that,
when $j\to\infty$, $\omega_{j,\xi}^{(2)}\to|\nabla_\xi U|^{p-2}$
and  $\int_{0}^{1}\frac{\nabla_\xi(U+t\varphi_j)}{|\nabla_\xi(U+t\varphi_j)|}\,dt\to
\frac{\nabla_\xi U}{|\nabla_\xi U|}$
almost everywhere in $\mathbb{R}^n$.
Hence, by this, $|\omega_{j,\xi}^{(2)}|\le|\nabla_\xi U|^{p-2}$,
the Lebesgue domination convergence theorem, and
$\lim_{j\to\infty}\rho_j=0$, we find that
\begin{align}\label{eq-spectral-5}
\lim_{j\to\infty}{\rm II}_{j,1}=
\int_{\mathbb{S}^{n-1}}\int_{K_R}
|\nabla_\xi U|^{p-2}(\nabla_\xi \varphi_0)^2.
\end{align}
On the other hand, for any $j\in\mathbb{N}$ and $\xi\in\mathbb{S}^{n-1}$,
it holds that
\begin{align*}
|A_{j,\xi}|:=
\left|\omega_{j,\xi}^{(2)}\int_{0}^{1}
\frac{\nabla_\xi(U+t\varphi_j)}{|\nabla_\xi(U+t\varphi_j)|}
(\nabla_\xi\varphi_0)\xi\right|
\le|\nabla_\xi U|^{p-2}|\nabla_\xi\varphi_0|
\lesssim|\nabla\varphi_0|,
\end{align*}
where the implicit positive constant is independent of $j$ and
$\xi$. From this and the structure of $U$,
we deduce that $A_{j,\xi}\in L^2(K_R;\mathbb{R}^n)$
and $\|A_{j,\xi}\|_{L^2(K_R;\mathbb{R}^n)}\lesssim1$
with the implicit positive constant independent of $j$ and $\xi$.
Therefore, applying \eqref{it3-varphi}, we find that, for any $\xi\in\mathbb{S}^{n-1}$,
$\int_{K_R}\nabla\Psi_j\cdot A_{j,\xi}\to0$ as $j\to\infty$.
Furthermore, since
\begin{align*}
\left|\int_{K_R}\nabla\Psi_j\cdot A_{j,\xi}\right|
\le\left(\int_{K_R}|\nabla\Psi_j|^2\right)^{\frac12}
\left(\int_{K_R}|A_{j,\xi}|^2\right)^{\frac12}\lesssim1
\end{align*}
uniformly in terms of both $j$ and $\xi$, from the Lebesgue
domination convergence theorem and the assumption
$\lim_{j\to\infty}\rho_j=0$, we infer that
\begin{align*}
\lim_{j\to\infty}
{\rm II}_{j,2}=\lim_{j\to\infty}(1-C_0\rho_j)\int_{\mathbb{S}^{n-1}}\int_{K_R}
\nabla\Psi_j\cdot A_{j,\xi}=0.
\end{align*}
This, together with
\eqref{eq-spectral-5}, further implies that
\begin{align}\label{eq-spectral-9}
\liminf_{j\to\infty}{\rm II}_{j}\ge
\int_{\mathbb{S}^{n-1}}\int_{K_R}
|\nabla_\xi U|^{p-2}(\nabla_\xi \varphi_0)^2,
\end{align}
which is the desired estimate of ${\rm II}_j$.
For ${\rm I}_j$, repeating the same argument:
expanding $\widehat{\varphi}_j=\varphi_0+\Psi_j$
and noting that $\omega_{j,\xi}^{(1)}\to |\nabla_\xi U|^{p-2}$
as $j\to\infty$, we can also obtain
\begin{align*}
\liminf_{j\to\infty}{\rm I}_j
\ge\int_{\mathbb{S}^{n-1}}\int_{K_R}
|\nabla_\xi U|^{p-2}|\nabla_\xi \varphi_0|^2.
\end{align*}
This, combined with \eqref{eq-spectral-3} and
\eqref{eq-spectral-9}, further implies that
\begin{align}\label{eq-spectral-6}
\liminf_{j\to\infty}
\frac{(1-C_0\rho_j)\int_{\mathbb{S}^{n-1}}\widetilde{L}^{(2)}_{\xi}(U,\varphi_j)\,d\xi}{\varepsilon_j^2}
\ge(p-1)\int_{\mathbb{S}^{n-1}}\int_{K_R}
|\nabla_\xi U|^{p-2}(\nabla_\xi\varphi_0)^2.
\end{align}

Next, we prove that, when $j\to\infty$,
\begin{align}\label{eq-spectral-7}
\operatorname{Var}(U,\widehat{\varphi}_j)
\to\operatorname{Var}(U,\varphi_0).
\end{align}
Indeed, by Minkowski's inequality, we obtain
\begin{align*}
\left|\operatorname{Var}(U,\widehat{\varphi}_j)^{\frac12}
-\operatorname{Var}(U,{\varphi}_0)^{\frac12}\right|
&\le\operatorname{Var}(U,\widehat{\varphi}_j-\varphi_0)^{\frac12}\\
&\le\left[\int_{\mathbb{S}^{n-1}} L^{(1)}_{\xi}(U,\widehat{\varphi}_j
-\varphi_0)^2\,d\xi\right]^{\frac12}
=\left[\int_{\mathbb{S}^{n-1}}L^{(1)}_\xi(U,\Psi_j)^2\,d\xi\right]^{\frac12}.
\end{align*}
Note that $L^{(1)}_\xi(U,\cdot)$ is a bounded linear functional
on $\dot{W}^{1,2}(|\nabla U|^{p-2})$. Indeed, from
H\"older's inequality, it follows that, for any
$f\in\dot{W}^{1,2}(|\nabla U|^{p-2})$,
\begin{align*}
|L^{(1)}_\xi(U,f)|
\le\int_{\mathbb{R}^n}|\nabla U|^{p-1}|\nabla f|
\le\left(\int_{\mathbb{R}^n}|\nabla U|^p\right)^{\frac12}
\left(\int_{\mathbb{R}^n}|\nabla U|^{p-2}|\nabla f|^2\right)^{\frac12}
\sim\|f\|_{\dot{W}^{1,2}(|\nabla U|^{p-2})},
\end{align*}
where the positive equivalence constants are independent of $\xi$.
Hence, using this, \eqref{it1-varphi}, and the Lebesgue
domination theorem, we obtain
$\int_{\mathbb{S}^{n-1}}L^{(1)}_\xi(U,\Psi_j)^2\,d\xi\to0$ as $j\to\infty$,
which shows \eqref{eq-spectral-7}. Then,
combining \eqref{eq-spectral-1}, \eqref{eq-spectral-6},
\eqref{eq-spectral-7}, \eqref{it1-varphi},
and the boundedness of $\{\widehat{\varphi}_j\}_{j\in\mathbb{N}}$
in $\dot{W}^{1,2}(|\nabla U|^{p-2})$,
we conclude that
\begin{align}\label{eq-spectral-8}
&(p-1)|\mathbb{S}^{n-1}|^{-1-\frac pn}
\int_{\mathbb{S}^{n-1}}\int_{K_R}
|\nabla_\xi U|^{p-2}(\nabla_\xi \varphi_0)^2
-(n+p)|\mathbb{S}^{n-1}|^{-1-\frac{2p}{n}}
\mathcal{E}_p(U)^{-p}\operatorname{Var}(U,\varphi_0)\notag\\
&\qquad\le\left[(p^*-1)S_{\rm aff}^p+\gamma\right]
\int_{\mathbb{R}^n} U^{p^*-2}\varphi_0^2.
\end{align}
On the other hand, by \eqref{it1-varphi} and
the orthogonality assumptions on $\{\varphi_j\}_{j\in\mathbb{N}}$,
we find that $\varphi_0\perp T_U\mathcal{M}_{\rm aff}$; hence,
applying Theorem \ref{thm-spectral} and \eqref{eq-spectral-8}
together with the arbitrariness of $R$, we further obtain
\begin{align*}
\left[(p^*-1)S_{\rm aff}^p+\gamma\right]
\int_{\mathbb{R}^n} U^{p^*-2}\varphi_0^2
\ge \left[(p^*-1)S_{\rm aff}^p+\delta\right]
\int_{\mathbb{R}^n} U^{p^*-2}\varphi_0^2.
\end{align*}
Since $\gamma<\delta$, it holds that $\varphi_0=0$ almost everywhere.
However, \eqref{eq-spectral-1} also shows that,
for any $j\in\mathbb{N}$,
\begin{align*}
&1-(n+p)|\mathbb{S}^{n-1}|^{-1-\frac{2p}{n}}
\mathcal{E}_p(U)^{-p}\operatorname{Var}(U,\widehat{\varphi}_j)\\
&\quad<\left[(p^*-1)S_{\rm aff}^p+\gamma\right]
\int_{\mathbb{R}^n} U^{p^*-2}\widehat\varphi_j^2
+\rho_j\int_{\mathbb{R}^n}
|\nabla U|^{p-2}(\nabla\widehat\varphi_j)^2,
\end{align*}
which, combined with \eqref{eq-spectral-7}, \eqref{it1-varphi},
and the boundedness of $\{\widehat{\varphi}_j\}_{j\in\mathbb{N}}$
in $\dot{W}^{1,2}(|\nabla U|^{p-2})$ again,
further implies that, if $\varphi_0=0$, then
$1\le 0$. It is impossible!
This then proves the present theorem.
\end{proof}


\section{Proofs of Theorems \ref{thm-stab} and \ref{thm-bubble}}\label{sec-proof}

In this section, we show Theorems \ref{thm-stab} and \ref{thm-bubble} following
the proofs of \cite[Theorem 1.1]{fz22} and \cite[Theorem 1.2]{lz25}, respectively.


\subsection{Proof of Theorem \ref{thm-stab}}

We need the following two lemmas
about the compactness and the orthogonality, respectively.

\begin{lemma}\label{lem-min}
If $p\in[2,n)$, then, for any $\varepsilon\in(0,\infty)$,
there is a positive constant $\delta$ such that,
for any $f\in\dot{W}^{1,p}$ satisfying
\begin{align*}
{\mathcal{E}_p(f)}^p
-S_{\rm aff}^p\|f\|_{p^*}^p<\delta \, \mathcal{E}_p(f)^p,
\end{align*}
it holds that
\begin{align*}
\inf_{(c,\lambda,S,x)\in\mathscr{A}}\|\nabla(T_{\lambda,S,x}f-cU)\|_{p}^p
<\varepsilon \, \mathcal{E}_p(f)^p.
\end{align*}
\end{lemma}

\begin{proof}
Assume that the desired conclusion does not true.
Then, by homogeneity, there is a positive constant $\varepsilon_0$
and a sequence $\{u_k\}_{k\in\mathbb{N}}$ in $\dot{W}^{1,p}$
such that $\mathcal{E}_p(u_k)=1$ for every $k\in\mathbb{N}$
and
\begin{align}\label{lem-min-e3}
\lim_{k\to\infty}\left[{\mathcal{E}_p(u_k)}^p
-S_{\rm aff}^p\|u_k\|_{p^*}^p\right]=0,
\end{align}
but, for every $k\in\mathbb{N}$,
\begin{align}\label{lem-min-e4}
\inf_{(c,\lambda,S,x)\in\mathscr{A}}\|\nabla(T_{\lambda,S,x}u_k-cU)\|_{p}
\ge\varepsilon_0.
\end{align}
By \eqref{lem-min-e3} and
Theorem \ref{thm-comp}, we find that,
for each $k\in\mathbb{N}$, there are $\lambda_k\in(0,\infty)$,
$A_k\in \operatorname{SL}(n)$,
and $x_k\in\mathbb{R}^n$ such that, when $k\to\infty$,
$v_k:=\mathcal{T}^{(p^*)}_{\lambda_k,A_k,x_k}u_k\to u$ in $\dot{W}^{1,p}$ for some
$u\in\mathcal{M}_{\rm aff}$.
Hence, from \eqref{def-min-equiv}, it follows that $u=c_0 T^{-1}_{\lambda_0,S_0,x_0}U$
for some $(c_0,\lambda_0,S_0,x_0)\in\mathscr{A}$.
Then, applying the polar decomposition and the fact that
$\|\cdot\|_{\dot{W}^{1,p}}$ is invariant under
$O(n)$ transformations, we conclude that
\eqref{lem-min-e4} also holds with $u_k$ replaced by
$v_k$.
Therefore, using a change of variables and the
operator norm inequality, we further obtain, for all $k\in\mathbb{N}$,
\begin{align}\label{lem-min-e2}
\varepsilon_0
\le \inf_{(c,\lambda,S,x)\in\mathscr{A}}\|\nabla(T_{\lambda,S,x}v_k-cU)\|_{p}
\le\|\nabla(T_{\lambda_0,S_0,x_0}v_k-u)\|_p
\le
e^{-\|S_0\|_{\mathcal{L}(\mathbb{R}^n)}}\|\nabla(v_k-u)\|_{p},
\end{align}
where $\|S_0\|_{\mathcal{L}(\mathbb{R}^n)}$ denotes the operator norm
of $S_0$ mapping $\mathbb{R}^n$ to itself.
Letting $k\to\infty$, we conclude $\varepsilon_0\le0$, a contradiction.
This then implies that the desired conclusion holds and then completes the proof.
\end{proof}

\begin{lemma}\label{lem-orth}
Let $p\in(1,n)$. Assume that $v_0=c_0T_{\lambda_0,{S_0},x_0}U$ is a positive function
with parameters $(c_0,\lambda_0,S_0,x_0)\in\mathscr{A}$,
and $\mathcal{U}$ is a compact neighborhood of $(c_0,\lambda_0,S_0,x_0)$ in
$\mathscr{A}$. Then there are a positive constant
  $\varepsilon_0$ and a positive function $\omega$ on
  $(0,\infty)$ satisfying $\lim_{t\to0^+}\omega(t)=0$ such that,
  if $\varepsilon\in(0,\varepsilon_0)$ and
  $f\in\dot{W}^{1,p}$ satisfy $\|\nabla f-\nabla v_0\|_p<\varepsilon$,
  then there is a tuple of parameters $(c_1,\lambda_1,S_1,x_1)$
  in the interior of
  $\mathcal{U}$ such that $v:=c_1T_{\lambda_1,{S_1},x_1}U$ is positive,
  $\|\nabla f-\nabla v\|_p<\omega(\varepsilon)$,
  and, for any $\eta\in T_v\mathcal{M}_{\rm aff}$,
  $f-v\perp\eta$ in $L^2(v^{p^*-2})$.
\end{lemma}

\begin{proof}
This proof is essentially the same as the proof of
\cite[Lemma 4.1]{fz22}; nevertheless, we outline the necessary details below.
For $f\in\dot{W}^{1,p}$, define the functional $F_f$ by setting,
for any
$w\in\mathcal{M}_{\rm aff}$,
\begin{align*}
F_f(w):=\frac{1}{p^*}\int_{\mathbb{R}^n}|w|^{p^*}
-\frac{1}{p^*-1}\int_{\mathbb{R}^n}|w|^{p^*-2}wf.
\end{align*}
If $f=v_0$, then, by H\"older's inequality and the property of
the function
\begin{align*}
(0,\infty)\ni s\mapsto\frac1{p^*}s^{p^*}-
\frac{1}{p^*-1}As^{p^*-1}
\end{align*}
for $A\in(0,\infty)$, we find that
$F_{v_0}(w)$ is uniquely minimized at $w=v_0$.
For general $f$, since the mapping
$\mathscr{A}
\ni (c,\lambda,S,x)\mapsto cT_{\lambda,S,x}(U)$
is continuous from $\mathscr{A}$ to $\dot{W}^{1,p}$, we infer that
\begin{align*}
\mathcal{F}_f:\ \mathscr{A}\ni (c,\lambda,S,x)\mapsto F_f(cT_{\lambda,S,x}U)
\end{align*}
is also continuous.
By this continuity and the compactness of $\mathcal{U}$,
we conclude that
there is a $(c_1,\lambda_1,S_1,x_1)\in\mathcal{U}$ such that
\begin{align}\label{lem-min-e1}
F_f(c_1T_{\lambda_1,{S_1},x_1}U)=
\mathcal{F}_f(c_1,\lambda_1,S_1,x_1)=
\inf_{(c,\lambda,S,x)\in\mathcal{U}}\mathcal{F}_f(c,\lambda,S,x)=
\inf_{(c,\lambda,S,x)\in\mathcal{U}}F_f(cT_{\lambda,S,x}U)
\end{align}
is attained. Moreover, if $f$ is close to $v_0$,
then, from a standard compactness argument and the observations that
$F_{v_0}(w)$ is uniquely minimized at $w=v_0$,
$f\mapsto F_{f}(w)$ is continuous in $\dot{W}^{1,p}$,
and the tuple of parameters
$(c,\lambda,S,x)\in\mathscr{A}$ is uniquely determined by the corresponding extremal
function,
we deduce that the tuple of parameters $(c_1,\lambda_1,S_1,x_1)\in\mathscr{A}$
at which the infimum in \eqref{lem-min-e1} is attained is close
to the tuple of parameters $(c_0,\lambda_0,S_0,x_0)$ in $\mathscr{A}$.
Hence, if $\varepsilon$ is sufficiently small, then,
for every $f\in\dot{W}^{1,p}$ such that $\|\nabla f-\nabla v_0\|_p<\varepsilon$,
$v:=c_1T_{\lambda_1,S_1,x_1}U$ is also positive and is close to $f$ in $\dot{W}^{1,p}$.
This shows the existence of the desired $\varepsilon_0$ and
$\omega$. Furthermore, if $f$ is close to $v_0$,
then the tuple of parameters $(c_1,\lambda_1,S_1,x_1)$ is in the interior of $\mathcal{U}$;
thus, for every direction vector $\mathbf{e}$ in $\mathscr{A}$,
\begin{align*}
\left.\frac{d}{dt}\right|_{t=0}
\mathcal{F}_f\big((c_1,\lambda_1,S_1,x_1)+t\mathbf{e}\big)=0.
\end{align*}
This, together with the definition of
the tangent space given in \eqref{eq-tangent},
further implies that,
for any $\eta\in T_v\mathcal{M}_{\rm aff}$,
$\int_{\mathbb{R}^n}v^{p^*-2}(v-f)\eta=0$,
which gives the desired orthogonality,
and hence completes the present proof.
\end{proof}

\begin{proof}[Proof of Theorem \ref{thm-stab}]
We actually prove the improved version \eqref{eq-stab-e}.
In the present proof,
unless otherwise specified, the implicit positive constants
depend only on $n$ and $p$.
By \cite[Theorem 1.2]{hl16} [see also \eqref{eq-p}] and
the affine invariance of \eqref{eq-stab}, we
can assume $\mathcal{E}_p(f)=1$ and
$\|\nabla f\|_p\lesssim1$.
Then, applying H\"older's inequality, we obtain
\begin{align}\label{thm-stab-e5}
&\inf_{(c,\lambda,S,x)\in\mathscr{A}}\left\{\|\nabla(T_{\lambda,S,x}f-cU)\|_{p}^p
+\int_{\mathbb{R}^n}
|c\nabla U|^{p-2}|\nabla(T_{\lambda,S,x}f-cU)|^2\right\}\notag\\
&\quad\le\|\nabla(f-cU)\|_{p}^p
+\int_{\mathbb{R}^n}
|\nabla U|^{p-2}|\nabla(f-cU)|^2\notag\\
&\quad\le \|\nabla(f-cU)\|_{p}^p+\|\nabla U\|_p^{p-2}
\|\nabla(f-cU)\|_{p}^2\lesssim1.
\end{align}
Therefore, to show \eqref{eq-stab-e}, we only need to
consider the case where ${\mathcal{E}_p(f)}^p
-S_{\rm aff}^p\|f\|_{p^*}^p$ is sufficiently small.
Then, from Lemma \ref{lem-min},
we deduce that there are a sufficiently small $\varepsilon$
and a tuple of affine transformation parameters $(c_0,\lambda_0,S_0,x_0)\in\mathscr{A}$ such that
$\|\nabla(T_{\lambda_0,S_0,x_0}f-c_0U)\|_p<\varepsilon$.
By replacing $f$ with $-f$ if necessary, we may assume that $c_0 > 0$.
Moreover, from the classical Sobolev inequality
\eqref{eq-Sob} and $\|T_{\lambda_{0},S_0,x_0}f\|_{p^*}=\|U\|_{p^*}=1$,
we deduce $c_0\sim 1$ and hence
$$\left\|\nabla\left(c^{-1}_0T_{\lambda_0,S_0,x_0}f-U\right)\right\|_p\lesssim\varepsilon.$$
By this and Lemma \ref{lem-orth} for the tuple of
parameters $(1,1,{\bf 0},{\bf 0})\in\mathscr{A}$ and its fixed
small compact neighborhood $\mathcal{U}\subset\mathscr{A}$,
there are a positive function $\omega$ on $(0,\infty)$ satisfying
$\lim_{t\to0^+}\omega(t)=0$ and a positive function $v=c_1T_{\lambda_1,S_1,x_1}U$ with
$(c_1,\lambda_1,S_1,x_1)\in\mathcal{U}$
such that
\begin{align*}
\left\|\nabla\left(c_0^{-1}T_{\lambda_{0},S_0,x_0}f-v\right)
\right\|_p\lesssim\omega(\varepsilon)\quad\text{and}
\quad \left(c_{0}^{-1}T_{\lambda_{0},S_0,x_0}f-v\right)\perp T_v\mathcal{M}_{\rm aff}
\quad \text{in}\quad L^2(v^{p^*-2}).
\end{align*}
Since $(c_1,\lambda_1,S_1,x_1)$ is close to
$(1,1,{\bf 0},{\bf 0})$ in $\mathscr{A}$,
similar to the last inequality
of \eqref{lem-min-e2}, we have
\begin{align*}
\left\|\nabla\left(c_0^{-1}c_{1}^{-1}T_{\lambda_1, S_1,x_1}^{-1}
T_{\lambda_{0},S_0,x_0}f-U\right)\right\|_p
&\lesssim\left\|\nabla\left(c_{0}^{-1}
T_{\lambda_{0},S_0,x_0}f-c_1T_{\lambda_1,S_1,x_1}U\right)\right\|_p
\lesssim\omega(\varepsilon).
\end{align*}
Moreover, applying a change of variables, one can easily find that
the orthogonality is invariant under the
affine transformation; that is,
\begin{align*}
\left(c_0^{-1}c_{1}^{-1}T_{\lambda_1, S_1,x_1}^{-1}
T_{\lambda_{0},S_0,x_0}f-U\right)\perp T_U\mathcal{M}_{\rm aff}
\quad \text{in}\quad L^2(U^{p^*-2}).
\end{align*}

Consequently, by the affine invariance of
\eqref{eq-stab-e} and the behavior of $\omega$ near $0$, we only need to consider the
case $f=U+\varphi$,
where $\varphi\perp T_{U}\mathcal{M}$ and
$\|\nabla\varphi\|_p<\varepsilon<\varepsilon_0$ for fixed and sufficiently
small $\varepsilon_0,\varepsilon$. Then, by Proposition \ref{prop-expan-local} and
Remark \ref{rem-3.4}, there are two positive constants
$C_{n,p}$ and $\widetilde{C}_{n,p}$, depending only on $n$ and $p$,
such that
\begin{align}\label{thm-stab-e1}
\mathcal{E}_p(f)^p
&\ge \mathcal{E}_p(U)^p
+p|\mathbb{S}_{n-1}|^{-1-\frac pn}
\int_{\mathbb{S}^{n-1}}\left[L^{(1)}_\xi(U,\varphi)
+\frac{1-C_{n,p}\varepsilon_0}{2}L^{(2)}_\xi(U,\varphi)\right]
\,d\xi\notag\\
&\qquad-\frac{p(n+p)}{2}|\mathbb{S}^{n-1}|^{-1-\frac{2p}{n}}
\mathcal{E}_p(U)^{-p}\operatorname{Var}(U,\varphi)
+\widetilde{C}_{n,p}\varepsilon_0\|\nabla \varphi\|_{p}^p,
\end{align}
where $L^{(1)}$ and $L^{(2)}$ are as in
\eqref{def-L1} and \eqref{def-L2}, respectively,
with $f$ replaced by $U$.
On the other hand, from the expansion for
$|a+b|^{p^*}$ obtained in \cite[Lemma 3.2]{fn19} (see also
\cite[Lemma 2.4]{fz22}):
for any $\varepsilon_0\in(0,\infty)$, there is a positive constant
$C_{p,\varepsilon_0}$, depending only on $p$ and $\varepsilon_0$,
such that, for all $a,b\in\mathbb{R}$ with $a\ne 0$,
\begin{align*}
    |a + b|^{p^*} \leq & |a|^{p^*} + p^*|a|^{p^*-2}ab
     + \left[ \frac{p^*(p^* - 1)}{2} + \varepsilon_0 \right]
     |a|^{p^*-2}|b|^2 + C_{p,\varepsilon_0}|b|^{p^*};
\end{align*}
therefore, we have
\begin{align*}
\|f\|_{p^*}^{p^*}
&\le\int_{\mathbb{R}^n}\left\{U^{p^*}
+p^{*}U^{p^*-1}\varphi
+\left[\frac{p^*(p^*-1)}{2}+\varepsilon_0\right]U^{p^*-2}\varphi^2
+C_{p,\varepsilon_0}|\varphi|^{p^*}\right\}.
\end{align*}
Combining this, the concavity of $t\mapsto t^{\frac{p}{p^*}}$,
and $\|U\|_{p^*}=1$, we conclude that
\begin{align*}
\|f\|_{p^*}^p
\le 1+p
\int_{\mathbb{R}^n}\left\{U^{p^*-1}\varphi
+\left[\frac{(p^*-1)}{2}+\frac1{p^*}\varepsilon_0\right]U^{p^*-2}\varphi^2
+\widetilde{C}_{p,\varepsilon_0}|\varphi|^{p^*}\right\},
\end{align*}
where the positive constant $\widetilde{C}_{p,\varepsilon_0}$ depends only on $p$ and $\varepsilon_0$.
This, together with \eqref{thm-stab-e1} and $\mathcal{E}_p(U)=S_{\rm aff}$,
further implies
\begin{align}\label{thm-stab-e3}
&{\mathcal{E}_p(f)}^p
-S_{\rm aff}^p\|f\|_{p^*}^p\notag\\
&\quad\gtrsim p\left[|\mathbb{S}^{n-1}|^{-1-\frac pn}
\int_{\mathbb{S}^{n-1}}L^{(1)}_\xi(U,\varphi)-
\int_{\mathbb{R}^n}U^{p^*-1}\varphi\right]\notag\\
&\qquad+\frac p2
\left\{|\mathbb{S}^{n-1}|^{-1-\frac pn}(1-C_{n,p}\varepsilon_0)
\int_{\mathbb{S}^{n-1}}L^{(2)}_\xi(U,\varphi)-(n+p)|\mathbb{S}^{n-1}|^{-1-\frac{2p}{n}}
\mathcal{E}_p(U)^{-p}
\operatorname{Var}(U,\varphi)\right.\notag\\
&\left.\qquad-\left[(p^*-1)S_{\rm aff}^p+\frac{S_{\rm aff}^p}{p^*}
{\varepsilon_0}\right]\int_{\mathbb{R}^n}
U^{p^*-2}\varphi^2
\right\}\notag\\
&\qquad+\widetilde{C}_{n,p}\varepsilon_0\|\nabla \varphi\|_p^p
-S_{\rm aff}^p\widetilde{C}_{p,\varepsilon_0}\|\varphi\|_{p^*}^{p^*}\notag\\
&\quad=:p{\rm I}_1+\frac p2{\rm I}_2(\varepsilon_0)+
\widetilde{C}_{n,p}\varepsilon_0\|\nabla \varphi\|_p^p
-S_{\rm aff}^p\widetilde{C}_{p,\varepsilon_0}\|\varphi\|_{p^*}^{p^*}.
\end{align}
By the fact that $U$ is an extremal function, it holds that
\begin{align*}
\left.\frac{d}{dt}\right|_{t=0}
\left[\mathcal{E}_p(U+t\varphi)^p
-S_{\rm aff}^p\|U+t\varphi\|_{p^*}^p\right]=0.
\end{align*}
Applying this, the facts that
$\mathcal{E}_p(U)=S_{\rm aff}$, $\|U\|_{p^*}=1$,
and $\mathcal{E}_p(U)=|\mathbb{S}^{n-1}|^{-\frac1n}\|\nabla_\xi U\|_p$
for all $\xi\in\mathbb{S}^{n-1}$ (which is a direct consequence of
$U$ being radial), we can find that
the term ${\rm I}_1$ defined in \eqref{thm-stab-e3} vanishes.
Moreover, from the affine spectral gap inequality
Theorem \ref{thm-spectral-wei} and Remark \ref{rem-spectral}, it follows that there is a
$\gamma\in(0,\infty)$, depending only on $n$ and $p$, such that, for
sufficiently small $\varepsilon_0$,
\begin{align*}
{\rm I}_2(\varepsilon_0)
\ge \left(\gamma-\frac{S_{\rm aff}^p}{p^*}\varepsilon_0\right)
\int_{\mathbb{R}^n}U^{p^*-2}\varphi^2
+\varepsilon_0\int_{\mathbb{R}^n}|\nabla U|^{p-2}
|\varphi|^2.
\end{align*}
Hence, when $\varepsilon_0$ is sufficiently small, it also holds
$
{\rm I}_2(\varepsilon_0)\ge \varepsilon_0\int_{\mathbb{R}^n}|\nabla U|^{p-2}
|\varphi|^2$.
On the other hand,
when $\varepsilon$ is sufficiently small, applying
\eqref{eq-Sob} and the fact the exponent $p^*>p$, we find that
\begin{align}\label{thm-stab-e4}
\widetilde{C}_{n,p}\varepsilon_0\|\nabla \varphi\|_p^p
-S_{\rm aff}^pC_{p,\varepsilon_0}\|\varphi\|_{p^*}^{p^*}
\ge \widetilde{C}_{n,p}\varepsilon_0\|\nabla \varphi\|_p^p
-S_{\rm aff}^{p}S_{\rm Sob}^{-p^*}C_{p,\varepsilon_0}
\|\nabla\varphi\|_{p^*}^{p^*}\gtrsim\|\nabla\varphi\|_p^p.
\end{align}
These further imply
\begin{align*}
{\mathcal{E}_p(f)}^p
-S_{\rm aff}^p\|f\|_{p^*}^p&\gtrsim\|\nabla\varphi\|_p^p
+\int_{\mathbb{R}^n}|\nabla U|^{p-2}
|\varphi|^2\\
&\ge \inf_{(c,\lambda,S,x)\in\mathscr{A}}\left\{\|\nabla(T_{\lambda,S,x}f-cU)\|_{p}^p
+\int_{\mathbb{R}^n}
|c\nabla U|^{p-2}|\nabla(T_{\lambda,S,x}f-cU)|^2\right\}.
\end{align*}
This shows \eqref{eq-stab-e} and hence completes the proof of Theorem
\ref{thm-stab}.
\end{proof}


\subsection{Proof of Theorem \ref{thm-bubble}}

In this subsection, we prove Theorem \ref{thm-bubble}.
First, applying the relative compactness for critical points (Theorem \ref{thm-critical})
and an argument similar to
Lemma \ref{lem-min}, we immediately obtain the following lemma.

\begin{lemma}\label{lem-critical}
If $p\in[2,n)$, then, for any $\varepsilon\in(0,\infty)$,
there is a positive constant $\delta$ such that,
for any $f\in\dot{W}^{1,p}$ satisfying \eqref{thm-bubble-e8} and
$\delta_{\rm aff}(f)<\delta$,
it holds that
\begin{align*}
\inf_{{\genfrac{}{}{0pt}{}{(c,\lambda,S,x)\in\mathscr{A}}{c\in\{-1,1\}}}}\|\nabla (T_{\lambda,S,x}f-cU)\|_p<\varepsilon.
\end{align*}
\end{lemma}

Based on this, we are able to show Theorem \ref{thm-bubble}.

\begin{proof}[Proof of Theorem \ref{thm-bubble}]
We actually prove the improved version \eqref{thm-bubble-e6}.
In the present proof,
unless otherwise specified, the implicit positive constants
depend only on $n$ and $p$.
Since $f$ satisfies \eqref{thm-bubble-e8}, from
an argument similar to \eqref{thm-stab-e5}, we infer that
$d(f,\mathcal{M}_{\rm aff}^{(1)})\lesssim 1$.
Hence, to show \eqref{thm-bubble-e6}, we can assume that $\delta_{\rm aff}$
is small. Fix a sufficiently small $\varepsilon\in(0,\infty)$.
By Lemma \ref{lem-critical}, there are $c_0\in\{-1,+1\}$,
$\lambda_0\in(0,\infty)$,
symmetric matrix $S_0$
satisfying $\operatorname{Tr}S_0=0$,
and $x_0\in\mathbb{R}^n$ such that
$\|\nabla(T_{\lambda_0,S_0,x_0}f-c_0U)\|_p<\varepsilon$.
By replacing $f$ with $-f$ if necessary, we may assume that $c_0 =1$.
Then, applying Lemma \ref{lem-orth} for the tuple of
parameters $(1,1,{\bf 0},{\bf 0})\in\mathscr{A}$ and its fixed
small compact neighborhood $\mathcal{U}\subset\mathscr{A}$,
there is a positive function $\omega$ on $(0,\infty)$ satisfying
$\lim_{t\to0^+}\omega(t)=0$ and a positive function $v=c_1T_{\lambda_1,{S_1},x_1}U$ with
$(c_1,\lambda_1,S_1,x_1)\in\mathcal{U}$
such that
\begin{align*}
\left\|\nabla\left(T_{\lambda_{0},S_0,x_0}f-v\right)
\right\|_p\le\omega(\varepsilon)\quad\text{and}
\quad T_{\lambda_{0},S_0,x_0}f-v\perp T_v\mathcal{M}_{\rm aff}
\quad \text{in}\quad L^2(v^{p^*-2}).
\end{align*}
Let $g:=T_{\lambda_1,{S_1},x_1}^{-1}T_{\lambda_{0},S_0,x_0}f$
and $\varphi:=g-c_1U$.
Since $(c_1,\lambda_1,S_1,x_1)$ is close to
$(1,1,{\bf 0},{\bf 0})$ in $\mathscr{A}$, it holds that
\begin{align*}
\|\nabla\varphi\|_p
=\left\|\nabla\left(T_{\lambda_1,{S_1},x_1}^{-1}
T_{\lambda_{0},S_0,x_0}f-c_1U\right)\right\|_p
\le |\lambda_1|\|e^{-S_1}\|_{\mathcal{L}(\mathbb{R}^n)}
\left\|\nabla\left(T_{\lambda_{0},S_0,x_0}f-v\right)
\right\|_p\le C_{n}\omega(\varepsilon),
\end{align*}
where $C_{n}$ is a positive constant
depending only on $n$, and in the second inequality we used the operator norm inequality for $e^{-S_1}$.
Moreover, by the behavior of $\omega$ near $0$,
we find that, for any sufficiently small $\varepsilon_0\in(0,\infty)$,
there is an $\varepsilon_1\in(0,1)$ such that, for every
$\varepsilon\in(0,\varepsilon_1)$,
$C_n\omega(\varepsilon)<\varepsilon_0$. In what follows,
we fix sufficiently small $\varepsilon,\varepsilon_0,\varepsilon_1$ that satisfies this property
at the moment.

Then, by Proposition \ref{prop-expan-lap} and
Remark \ref{rem-3.3}, there are two positive constants
$C_{n,p}$ and $\widetilde{C}_{n,p}$, depending only on $n$ and $p$,
such that
\begin{align*}
-\int_{\mathbb{R}^n}
\Delta_p^{\rm aff}(g)\varphi
&\ge
-\int_{\mathbb{R}^n}\Delta_p^{\rm aff}(c_1U)\varphi
+(1-C_{n,p}\varepsilon_0)|\mathbb{S}^{n-1}|^{-1-\frac pn}
\int_{\mathbb{S}^{n-1}}\widetilde{L}^{(2)}_\xi(c_1U,\varphi)\,d\xi\notag\\
&\qquad-(n+p)|\mathbb{S}^{n-1}|^{-1-\frac{2p}{n}}
\mathcal E_p(c_1U)^{-p}
\operatorname{Var}(c_1U,\varphi)
+\widetilde C_{n,p}\varepsilon_0\|\nabla \varphi\|_{p}^p.
\end{align*}
where $\operatorname{Var}$ and $\widetilde{L}^{(2)}$ are as,
respectively, in \eqref{def-var} and \eqref{eq-L2w}.
On the other hand, from the expansion for
$|a+b|^{p^*-2}(a+b)b$ obtained in \cite[Lemma 2.2]{lz25}:
for any $\varepsilon_0\in(0,\infty)$, there is a positive constant
$C_{p,\varepsilon_0}$, depending only on $p$ and $\varepsilon_0$,
such that, for all $a,b\in\mathbb{R}$ with $a\ne 0$,
\begin{align*}
    |a + b|^{p^*-2}(a+b)b \leq & |a|^{p^*-2}ab
    +(p^*-1+\varepsilon_0)|a|^{p^*-2}b^2
    + C_{p,\varepsilon_0}|b|^{p^*}.
\end{align*}
Abbreviating
\begin{align}\label{def-P}
P(g):=-\Delta_{p}^{\rm aff}(g)-S_{\rm aff}^p|g|^{p^*-2}g,
\end{align}
we infer that
\begin{align*}
\int_{\mathbb{R}^n}
P(g)\varphi
&=-\int_{\mathbb{R}^n}
\Delta_p^{\rm aff}(g)\varphi
-S_{\rm aff}^p\int_{\mathbb{R}^n}|g|^{p^*-2}g\varphi\\
&\ge \int_{\mathbb{R}^n}
P(c_1U)\varphi
+\left\{(1-C_{n,p}\varepsilon_0)|\mathbb{S}^{n-1}|^{-1-\frac pn}
\int_{\mathbb{S}^{n-1}}\widetilde{L}^{(2)}_\xi(c_1U,\varphi)\,d\xi\right.\notag\\
&\quad-(n+p)|\mathbb{S}^{n-1}|^{-1-\frac{2p}{n}}
\mathcal{E}_p(c_1U)^{-p}
\operatorname{Var}(c_1U,\varphi)\notag\\
&\left.\quad-\left[(p^*-1)S_{\rm aff}^p-{S_{\rm aff}^p}{\varepsilon_0}\right]\int_{\mathbb{R}^n}
|c_1U|^{p^*-2}\varphi^2
\right\}\notag\\
&\quad+\widetilde{C}_{n,p}\varepsilon_0\|\nabla \varphi\|_p^p
-S_{\rm aff}^pC_{p,\varepsilon_0}\|\varphi\|_{p^*}^{p^*}\notag\\
&=:{\rm I}_1+{\rm I}_2(\varepsilon_0)+
\widetilde{C}_{n,p}\varepsilon_0\|\nabla \varphi\|_p^p
-S_{\rm aff}^pC_{p,\varepsilon_0}\|\varphi\|_{p^*}^{p^*}.
\end{align*}
By the fact $-\Delta_p^{\rm aff}U=S_{\rm aff}^pU^{p^*-1}$,
we obtain
\begin{align*}
P(c_1U)=-c_1^{p-1}\Delta_p^{\rm aff}(U)-S_{\rm aff}c_1^{p^*-1}U^{p^*-1}
=S_{\rm aff}^p\left(c_1^{p-1}-c_1^{p^*-1}\right)U^{p^*-1}.
\end{align*}
Applying a change of variables, one can easily find that
the orthogonality is invariant under the
affine transformation; that is,
$\varphi\perp T_U\mathcal{M}_{\rm aff}$
in $L^2(U^{p^*-2})$.
This further implies that $${\rm I}_1=\int_{\mathbb{R}^n}P(c_1U)\varphi=0.$$
Moreover, from the affine spectral gap inequality
Theorem \ref{thm-spectral-wei} and the dilation properties
of $\widetilde{L}^{(2)}$ [see \eqref{eq-L-dila}], it follows that there is a
$\gamma\in(0,\infty)$, depending only on $n$ and $p$, such that
\begin{align*}
{\rm I}_2(\varepsilon_0)
&\ge \left[c_1^{p-2}\gamma+(c_1^{p-2}-c_1^{p^*-2})(p^*-1)
S_{\rm aff}^p-c_1^{p^*-2}S_{\rm aff}^p
\varepsilon_0\right]
\int_{\mathbb{R}^n}U^{p^*-2}\varphi^2\\
&\qquad+c_1^{p-2}\varepsilon_0\int_{\mathbb{R}^n}
|\nabla U|^{p-2}|\nabla\varphi|^2.
\end{align*}
Hence, when $\varepsilon_0>0$ and
the compact neighborhood $\mathcal{U}\subset\mathscr{A}$ are both sufficiently small,
it also holds
$
{\rm I}_2(\varepsilon_0)\gtrsim\int_{\mathbb{R}^n}
|\nabla U|^{p-2}|\nabla\varphi|^2$.
Furthermore,
when $\varepsilon_1>0$ is sufficiently small, applying
\eqref{eq-Sob} and \eqref{thm-stab-e4}, we find that
$$\widetilde{C}_{n,p}\varepsilon_0\|\nabla \varphi\|_p^p
-S_{\rm aff}^pC_{p,\varepsilon_0}\|\varphi\|_{p^*}^{p^*}
\gtrsim\|\nabla\varphi\|_p^p.$$
This, combined with the fact that, for each $\lambda\in(0,\infty)$,
symmetric matrix $S$ satisfying $\operatorname{Tr}S=0$,
and $x\in\mathbb{R}^n$,
\begin{align*}
\int_{\mathbb{R}^n}P(T_{\lambda,S,x}f)g
=\int_{\mathbb{R}^n}P(f)T_{\lambda,S,x}^{-1}g,
\end{align*}
further implies
\begin{align}\label{thm-bubble-e4a}
\|\nabla\varphi\|_p^p
+\int_{\mathbb{R}^n}
|\nabla U|^{p-2}|\nabla\varphi|^2
&\lesssim
\int_{\mathbb{R}^n}
P(g)\varphi
=\int_{\mathbb{R}^n}P\left(f\circ e^{S_1-S_0}\right)
T_{\lambda_0,{\bf 0},x_1}T^{-1}_{\lambda_1,{\bf 0},x_1}\varphi\notag\\
&\le \left\|P\left(f\circ e^{S_1-S_0}\right)\right\|_{\dot{W}^{-1,p'}}
\|\nabla\varphi\|_{p}\le
\delta_{\rm aff}(f)\|\nabla\varphi\|_{p},
\end{align}
where we used the fact that if the trace-zero symmetric matrix
in the affine transformation is zero, then
this transformation preserves the $\dot{W}^{1,p}$ norm.
This yields that
\begin{align*}
\|\nabla\varphi\|_p^{p-1}
+\frac{\int_{\mathbb{R}^n}
|\nabla U|^{p-2}|\nabla\varphi|^2}{\|\nabla\varphi\|_p}
\lesssim\delta_{\rm aff}(f).
\end{align*}
From this and the element inequality that, for all $a,b\in(0,\infty)$,
$a^{p-1}+\frac ba\ge p(\frac{b}{p-1})^{\frac{p-1}{p}}$, we deduce that
\begin{align}\label{thm-bubble-e7}
\|\nabla\varphi\|_p^{p-1}
+\left(\int_{\mathbb{R}^n}
|\nabla U|^{p-2}|\nabla\varphi|^2\right)^{\frac{p-1}{p}}
\lesssim\delta_{\rm aff}(f).
\end{align}
Furthermore, by this and $p\ge2$, we have $\frac{2(p-1)}{p}\le p-1$ and hence
\begin{align}\label{thm-bubble-e5}
d(f,\mathcal{M}_{\rm aff}^{(1)})
&\le \|\nabla (g-U)\|_p^{p-1}
+\left[\int_{\mathbb{R}^n}
|\nabla U|^{p-2}|\nabla(g-U)|^2\right]^{\frac{p-1}{p}}\notag\\
&\lesssim\|\nabla\varphi\|_p^{p-1}
+\left(\int_{\mathbb{R}^n}|\nabla U|^{p-2}|\nabla\varphi|^2\right)^{\frac{p-1}{p}}
+|c_1-1|^{p-1}+|c_1-1|^{\frac{2(p-1)}{p}}\notag\\
&\lesssim
\delta_{\rm aff}(f)+|c_1-1|^{\frac{2(p-1)}{p}}.
\end{align}

Therefore, we still need to estimate $|c_1-1|$.
Note that
\begin{align*}
\int_{\mathbb{R}^{n}}
P(c_1U)(c_1U)
&=\int_{\mathbb{R}^{n}}
P(g)g
-\left[\int_{\mathbb{R}^{n}}
P(g)g
-\int_{\mathbb{R}^{n}}
P(c_1U)(c_1U)\right]\\
&=:{\rm I}_3-{\rm I}_4.
\end{align*}
By $\|\nabla(T_{\lambda_0,S_0,x_0}f-U)\|_p<\varepsilon$,
the fact that $(c_1,\lambda_1,S_1,x_1)$ is close to
$(1,1,{\bf 0},{\bf 0})$, and a similar change of variables
argument used in \eqref{thm-bubble-e4a}, we conclude that
\begin{align*}
|{\rm I}_3|\le \left\|P\left(f\circ e^{S_1-S_0}\right)\right\|_{\dot{W}^{-1,p'}}
\|\nabla g\|_{p}\lesssim\delta_{\rm aff}(f).
\end{align*}
On the other hand,
\begin{align*}
{\rm I}_4
=\mathcal{E}_{p}(g)^p-\mathcal{E}_p(c_1U)^p
-S_{\rm aff}^p
\left(\|g\|_{p^*}^{p^*}-\|c_1U\|_{p^*}^{p^*}\right).
\end{align*}
From the fact that $\|\nabla\varphi\|_p$ is sufficiently small and
Proposition \ref{prop-expan-E}, we deduce that
\begin{align*}
\left|\mathcal{E}_p(g)^p
-\mathcal{E}_p(c_1U)^p
+p\int_{\mathbb{R}^n}\Delta_p^{\rm aff}(c_1U)\varphi\right|
&\lesssim\|\nabla \varphi\|_{p}^p
+\int_{\mathbb{R}^n}|\nabla(c_1U)|^{p-2}|\nabla\varphi|^2\\
&\lesssim\|\nabla \varphi\|_{p}^p
+\int_{\mathbb{R}^n}|\nabla U|^{p-2}|\nabla\varphi|^2.
\end{align*}
Moreover, by  $-\Delta_p^{\rm aff}U=S_{\rm aff}^pU^{p^*-1}$ and
$\varphi\perp T_U\mathcal{M}_{\rm aff}$
in $L^2(U^{p^*-2})$,
we have
\begin{align*}
\int_{\mathbb{R}^n}\Delta_p^{\rm aff}U\varphi
=\int_{\mathbb{R}^n}U^{p^*-1}\varphi=0.
\end{align*}
Then
\begin{align}\label{thm-bubble-e2}
\left|\mathcal{E}_p(g)^p
-\mathcal{E}_p(c_1U)^p\right|
\lesssim\|\nabla \varphi\|_{p}^p
+\int_{\mathbb{R}^n}|\nabla U|^{p-2}|\nabla\varphi|^2.
\end{align}
Similarly, we can obtain
\begin{align*}
\left|\|g\|_{p^*}^{p^*}-\|c_1U\|_{p^*}^{p^*}\right|
&=\left|\int_{\mathbb{R}^n}
|g|^{p^*}-|c_1U|^{p^*}-p^*|c_1U|^{p^*-2}(c_1U)\varphi\right|\\
&\lesssim\|\varphi\|_{p^*}^{p^*}
+\int_{\mathbb{R}^n}U^{p^*-2}\varphi^2.
\end{align*}
By this, the Sobolev inequality \eqref{eq-Sob}, the embedding
$\dot{W}^{1,2}(|\nabla U|^{p-2})\hookrightarrow L^2(U^{p^*-2})$
(see \cite[Proposition 3.2]{fz22}), and $p^*>p$ together with
the fact that $\|\nabla\varphi\|_p$ is sufficiently small,
we conclude that
\begin{align*}
\left|\|g\|_{p^*}^{p^*}-\|c_1U\|_{p^*}^{p^*}\right|
&\lesssim\|\nabla\varphi\|_{p}^{p^*}
+\int_{\mathbb{R}^n}|\nabla U|^{p-2}|\nabla\varphi|^2
\le \|\nabla\varphi\|_{p}^p
+\int_{\mathbb{R}^n}|\nabla U|^{p-2}|\nabla\varphi|^2.
\end{align*}
Combining this, \eqref{thm-bubble-e2}, and the estimate of ${\rm I}_3$,
we obtain
\begin{align}\label{thm-bubble-e3}
\left|\int_{\mathbb{R}^{n}}
P(c_1U)(c_1U)\right|
\lesssim\delta_{\rm aff}(f)
+\|\nabla\varphi\|_p^p+\int_{\mathbb{R}^n}
|\nabla U|^{p-2}|\nabla\varphi|^2.
\end{align}
On the other hand,
note that $\mathcal{E}_p(U)=S_{\rm aff}$
and $\|U\|_{p^*}=1$.
Consequently,
\begin{align*}
\int_{\mathbb{R}^{n}}
P(c_1U)(c_1U)
=\mathcal{E}_{p}(c_1U)^p-S_{\rm aff}^p\|c_1U\|_{p^*}^{p^*}
=(c_1^{p}-c_1^{p^*})
S_{\rm aff}^p.
\end{align*}
From $\lim_{c\to 1}\frac{c^p-c^{p^*}}{1-c}
=p^*-p>0$
and $c_1$ is close $1$,
it follows that $|c_1^{p}-c_1^{p^*}|\gtrsim|c_1-1|$.
Hence, by \eqref{thm-bubble-e3},
we find that
\begin{align}\label{thm-bubble-e4}
|c_1-1|\lesssim\delta_{\rm aff}(f)
+\|\nabla\varphi\|_p^p+\int_{\mathbb{R}^n}
|\nabla U|^{p-2}|\nabla\varphi|^2,
\end{align}
which, combined with \eqref{thm-bubble-e7},
implies that, if $\delta_{\rm aff}(f)$
is sufficiently small, then
\begin{align*}
|c_1-1|^{\frac{2(p-1)}{p}}&\lesssim\delta_{\rm aff}(f)^{\frac{2(p-1)}{p}}
+\|\nabla\varphi\|_p^{2(p-1)}
+\left(\int_{\mathbb{R}^n}|\nabla U|^{p-2}|\nabla\varphi|^2\right)^{\frac{{2(p-1)}}{p}}\\
&\lesssim\delta_{\rm aff}(f)
+\left[\|\nabla\varphi\|_p^{p-1}
+\left(\int_{\mathbb{R}^n}|\nabla U|^{p-2}|\nabla\varphi|^2\right)^{\frac{{(p-1)}}{p}}\right]^2
\lesssim \delta_{\rm aff}(f).
\end{align*}
Using this and \eqref{thm-bubble-e5},
we obtain \eqref{thm-bubble-e6}.
This completes the proof of Theorem
\ref{thm-bubble}.
\end{proof}


\section{Sharpness of the exponents in Theorems \ref{thm-stab} and \ref{thm-bubble}}\label{Sec-sharp}

In this section, we are concerned with proving
the sharpness of the exponents in Theorems
\ref{thm-stab} and \ref{thm-bubble}.
This follows from similar
constructions to those in \cite[Remark 1.2]{fz22} and
\cite[Section 5.1]{fp24}.
Recall that $U:=c_0(1+|x|^{p'})^{\frac{p-n}{p}}$
for $x\in\mathbb{R}^n$
satisfies $\|U\|_{p^*}=1$ and the operator $P$ is defined in \eqref{def-P},
and fix $\varphi\in C_{\rm c}^\infty(B({\bf0},1))$ be nonzero.

\subsection{Sharpness of the exponents for the $\dot{W}^{1,p}$ norm}\label{subsec-61}

We first consider the exponents for the $\dot{W}^{1,p}$ norm.
Let $\mathbf{e}_1:=(1,0,\ldots,0)\in\mathbb{R}^n$.
Fix $\{R_\varepsilon\}_{\varepsilon\in(1,\infty)}$
in $(0,\infty)$.
For every $\varepsilon\in(0,1)$,
define $\varphi_\varepsilon(\cdot):=\varphi(\cdot-R_\varepsilon \mathbf{e}_1)$
and $f_\varepsilon:=U+\varepsilon\varphi_\varepsilon$.
We first provide the following upper estimate for the deficits.

\begin{lemma}\label{lem-tail-1}
For every $\varepsilon\in(0,1)$,
choose $R_\varepsilon\in(1,\infty)$ such that
\begin{align}\label{eq:small-tail}
\sup_{|x-R_\varepsilon \mathbf{e}_1|<1}
\left\{
|U(x)|+|\nabla U(x)|+|\nabla^2U(x)|
\right\}
\le \varepsilon.
\end{align}
Then there are positive constants $c$ and $C$, depending only on $n$, $p$, $U$, and
$\varphi$, such that, for every $\varepsilon\in(0,c)$,
$${\mathcal{E}_p(f_\varepsilon)}^p
-S_{\rm aff}^p\|f_\varepsilon\|_{p^*}^p\le C\varepsilon^p\quad
\text{and}\quad
\delta_{\rm aff}(f_\varepsilon)
\le C\varepsilon^{p-1}.$$
\end{lemma}

\begin{proof}
In the present proof,
unless otherwise specified, the implicit positive constants
depend only on $n$, $p$, $U$, and $\varphi$.
Due to \eqref{eq-holder},
the first assertion is a direct consequence of
\cite[Remark 1.2]{fz22}. Hence, we only prove the estimate on
$\delta_{\rm aff}$.
Let $q:=(p^*)'$.
Then, by the Sobolev inequality \eqref{eq-Sob}
and duality, we obtain, for any $R\in L^q$,
$\|R\|_{\dot W^{-1,p'}}
\lesssim\|R\|_q.$
For every symmetric matrix $S$ satisfying
$\operatorname{Tr}S=0$,
by this and a change of variables combined with
$\det e^S=1$, we conclude that, for every $\varepsilon\in(0,1)$,
\begin{align*}
\left\|P(f_\varepsilon\circ e^S)\right\|_{\dot W^{-1,p'}}
&\lesssim\left\|P(f_\varepsilon\circ e^S)\right\|_q
=\left\|P(f_\varepsilon)\circ e^S\right\|_q
=\left\|P(f_\varepsilon)\right\|_q.
\end{align*}
Taking the supremum over all such $S$ and
applying a polar decomposition argument
similar to that used in \eqref{def-min-equiv}, we obtain
\begin{align}
\delta_{\rm aff}(f_\varepsilon)
\lesssim\left\|P(f_\varepsilon)\right\|_q.
\label{eq:delta-Lq}
\end{align}
Therefore, to complete the present proof, we only need to show
$\|P(f_\varepsilon)\|_q\lesssim\varepsilon^{p-1}.$

Now, fix $\varepsilon\in(0,1)$.
Since $U$ is radial, $\|\nabla_\xi U\|_p$ is independent of $\xi\in\mathbb{S}^{n-1}$.
Define $a:=\|\nabla_\xi U\|_p^p$ for any
$\xi\in\mathbb{S}^{n-1}$.
Then, from the argument similar to that used in \eqref{prop-expan-E-e3} and
the fact that $\operatorname{supp}(\varphi_\varepsilon)\subset
B(R_\varepsilon\mathbf{e}_1,1)$, we deduce that, for every
$\xi\in\mathbb{S}^{n-1}$,
\begin{align}\label{eq:directional-p-bound}
\left|\|\nabla_\xi f_\varepsilon\|_p^p-a\right|
&\lesssim
\int_{B(R_\varepsilon \mathbf{e}_1,1)}
\left[\varepsilon
|\nabla_\xi U|^{p-1}|\nabla_\xi\varphi_\varepsilon|
+
\varepsilon^p|\nabla_\xi\varphi_\varepsilon|^p
\right]
\notag\\
&\lesssim\varepsilon^p\quad\text{by \eqref{eq:small-tail}}.
\end{align}
 Hence, if $\varepsilon$ is sufficiently small, then,
 for all $\xi\in\mathbb{S}^{n-1}$,
$\frac{a}{2}
\le \|\nabla_\xi f_\varepsilon\|_p^p
\le 2a.$
Applying this, the mean value theorem, respectively,
for $$t\mapsto (a+t\rho)^{-1-\frac np}\quad\text{and}\quad
t\mapsto\left(\int_{\mathbb{S}^{n-1}}\left[a+tV(\xi)\,d\xi\right]^{-\frac np}\right)^{1-\frac pn},$$
and \eqref{eq:directional-p-bound},
we find that, for all $\xi\in\mathbb{S}^{n-1}$,
\begin{align*}
\left|\|\nabla_\xi f_\varepsilon\|_p^{-n-p}-
\|\nabla_\xi U\|_p^{-n-p}\right|
\lesssim a^{-2-\frac np}
\sup_{\xi\in\mathbb{S}^{n-1}}
\left|\|\nabla_\xi f_\varepsilon\|_p^p-a\right|
\lesssim\varepsilon^p
\end{align*}
and
\begin{align*}
\left|{\mathcal E}_p(f_\varepsilon)-{\mathcal E}_p(U)\right|
\lesssim a^{-\frac np}
\sup_{\xi\in\mathbb{S}^{n-1}}
\left|\|\nabla_\xi f_\varepsilon\|_p^p-a\right|
\lesssim\varepsilon^p.
\end{align*}
These further imply
\begin{align}\label{eq:coefficient-bound}
\left|
{\mathcal E}_p(f_\varepsilon)^{n+p}
\|\nabla_\xi f_\varepsilon\|_p^{-n-p}
-
{\mathcal E}_p(U)^{n+p}\|\nabla_\xi U\|_p^{-n-p}
\right|
\lesssim\varepsilon^p
\end{align}
for all $\xi\in{\mathbb S}^{n-1}$.

Now, we turn to estimate $P(f_\varepsilon)$.
Since $P(U)=0$, we infer that
\begin{align}\label{eq:split-P}
P(f_\varepsilon)
=-
\left(\Delta_p^{\rm aff}f_\varepsilon-\Delta_p^{\rm aff}U\right)
-S_{\rm aff}^p
\left[
|f_\varepsilon|^{p^*-2}f_\varepsilon-U^{p^*-1}
\right].
\end{align}
We first estimate the affine Laplacian term. By the definition of $\Delta_p^{\rm aff}$,
we obtain
\begin{align*}
&\Delta_p^{\rm aff}f_\varepsilon-\Delta_p^{\rm aff}U
\notag\\
&\quad=\operatorname{div}
\left(
\int_{{\mathbb S}^{n-1}}
{\mathcal E}_p(U)^{n+p}\|\nabla_\xi U\|_p^{-n-p}
\left[
|\nabla_\xi f_\varepsilon|^{p-2}\nabla_\xi f_\varepsilon
-
|\nabla_\xi U|^{p-2}\nabla_\xi U
\right]\xi\,d\xi
\right)
\notag\\
&\qquad+\operatorname{div}
\left(
\int_{{\mathbb S}^{n-1}}
\left[
{\mathcal E}_p(f_\varepsilon)^{n+p}
\|\nabla_\xi f_\varepsilon\|_p^{-n-p}
-
{\mathcal E}_p(U)^{n+p}\|\nabla_\xi U\|_p^{-n-p}
\right]
|\nabla_\xi f_\varepsilon|^{p-2}\nabla_\xi f_\varepsilon\,\xi\,d\xi
\right)\notag\\
&\quad=:I_\varepsilon+II_\varepsilon.
\end{align*}
We first deal with $I_\varepsilon$.
By the support assumption,
$I_\varepsilon$ vanishes outside $B(R_\varepsilon \mathbf{e}_1,1)$.
On the other hand, using
\eqref{eq:small-tail} and the smoothness of $\varphi$,
we obtain, for any $x\in B(R_\varepsilon \mathbf{e}_1,1)$,
\begin{align*}
&\left|
\nabla
\left[
|\nabla_\xi f_\varepsilon(x)|^{p-2}\nabla_\xi f_\varepsilon(x)
-
|\nabla_\xi U(x)|^{p-2}\nabla_\xi U(x)
\right]
\right|
\\
&\quad=(p-1)
\left|
|\nabla_\xi f_\varepsilon(x)|^{p-2}\nabla\nabla_\xi f_\varepsilon(x)
-
|\nabla_\xi U(x)|^{p-2}\nabla\nabla_\xi U(x)
\right|
\\
&\quad
\lesssim|\nabla_\xi f_\varepsilon(x)|^{p-2}
|\varepsilon\nabla\nabla_\xi\varphi_\varepsilon|
+
\left||\nabla_\xi f_\varepsilon(x)|^{p-2}-|\nabla_\xi U(x)|^{p-2}\right|
|\nabla\nabla_\xi U(x)|
\lesssim\varepsilon^{p-1}.
\end{align*}
Thus, it holds that
\begin{align*}
|I_\varepsilon|
&\le
\int_{{\mathbb S}^{n-1}}
{\mathcal E}_p(U)^{n+p}\|\nabla_\xi U\|_p^{-n-p}
\left|\nabla\left[
|\nabla_\xi f_\varepsilon|^{p-2}\nabla_\xi f_\varepsilon
-
|\nabla_\xi U|^{p-2}\nabla_\xi U
\right]\right|\,d\xi
\lesssim\varepsilon^{p-1}{\bf 1}_{B(R_\varepsilon \mathbf{e}_1,1)}
\end{align*}
and hence $\|I_\varepsilon\|_q\lesssim\varepsilon^{p-1}$.
In addition, for $II_\varepsilon$, from the choice of $\varphi$
and \eqref{eq:small-tail},
it follows that
\begin{align}\label{eq-secon}
\left|\nabla
\left(|\nabla_\xi f_\varepsilon(x)|^{p-2}\nabla_\xi f_\varepsilon\right)\right|
&=(p-1)\left||\nabla_\xi f_\varepsilon|^{p-2}\nabla\nabla_\xi f_\varepsilon\right|
\lesssim\left(|\nabla_\xi U|+|\nabla_\xi\varphi_\varepsilon|\right)^{p-2}
\left(|\nabla\nabla_\xi U|+|\nabla\nabla_\xi \varphi_\varepsilon|\right)\notag\\
&\lesssim|\nabla_\xi U|^{p-2}|\nabla\nabla_\xi U|
+{\bf 1}_{B(R_\varepsilon\mathbb{e}_1,1)}.
\end{align}
By the construction of $U$, we find that, when $|x|\to\infty$,
\begin{align*}
|\nabla_\xi U(x)|^{p-2}|\nabla\nabla_\xi U(x)|
=O\left(|x|^{\frac{p-n}{p-1}-1}\right)^{p-2}
O\left(|x|^{\frac{p-n}{p-1}-2}\right)
=O\left(|x|^{-n}\right).
\end{align*}
Hence, combining this and \eqref{eq-secon},
we conclude that $\|\nabla
(|\nabla_\xi f_\varepsilon(x)|^{p-2}\nabla_\xi f_\varepsilon)\|_q\lesssim1.$
Then, by \eqref{eq:coefficient-bound}, we obtain
$\|II_\varepsilon\|_q\lesssim\varepsilon^p$.
Consequently, we obtain
\begin{align}\label{eq:laplacian-bound}
\left\|
\Delta_p^{\rm aff}f_\varepsilon-\Delta_p^{\rm aff}U
\right\|_q
\lesssim\varepsilon^{p-1}.
\end{align}

Next, we consider the Lebesgue term in \eqref{eq:split-P}.
Following the strategy used in the estimation of $I_\varepsilon$,
we treat the domain $B(R_\varepsilon \mathbf{e}_1,1)$ and its complement separately;
utilizing the inequality
\begin{align*}
\left||a+b|^{p^*-2}(a+b)-|a|^{p^*-2}a\right| \lesssim |a|^{p^*-2}|b| +|b|^{p^*-1}
\end{align*}
(which is a direct consequence of Taylor's formula) in $B(R_\varepsilon \mathbf{e}_1,1)$,
we conclude that
\begin{align}\label{eq:zero-order-bound-e1}
\left||f_\varepsilon|^{p^*-2}f_\varepsilon-U^{p^*-1} \right| \lesssim
\varepsilon^{p^*-1}{\bf 1}_{B(R_\varepsilon\mathbf{e}_1,1)}
\end{align}
and hence
\begin{align}\label{eq:zero-order-bound}
\left\| |f_\varepsilon|^{p^*-2}f_\varepsilon-U^{p^*-1} \right\|_q \lesssim\varepsilon^{p^*-1}.
\end{align}
Since $p^*>p$, from \eqref{eq:split-P},
\eqref{eq:laplacian-bound}, and \eqref{eq:zero-order-bound},
we infer that, for sufficiently small $\varepsilon$,
$\|P(f_\varepsilon)\|_q
\lesssim\varepsilon^{p-1}.$
This, together with \eqref{eq:delta-Lq}, further implies
$\delta_{\rm aff}(f_\varepsilon)
\lesssim\varepsilon^{p-1}$
and hence completes the present proof.
\end{proof}

On the other hand, we now estimate the
distance from below.

\begin{lemma}\label{lem-pterm}
For every $\varepsilon\in(0,1)$, choose
$R_\varepsilon$ such that
\begin{align}\label{lem-d-e1}
\sup_{|y|\ge \frac{R_\varepsilon}{4}} |U(y)|\le \varepsilon^2.
\end{align}
Then there are positive constants $c$ and $C$, depending only on $n$, $p$, $U$, and
$\varphi$, such that, for every $\varepsilon\in(0,c)$,
$$
\inf_{v\in \mathcal M_{\rm aff}^{(1)},\, B\in \operatorname{SL}(n)}\|\nabla(f\circ B -v)\|_{p}
\ge\inf_{v\in \mathcal M_{\rm aff},\, B\in \operatorname{SL}(n)}\|\nabla(f\circ B -v)\|_{p}
\ge C\varepsilon.$$
\end{lemma}

\begin{proof}
We only need to prove the second estimate.
We show this by contradiction.  Suppose that it is false.
Then there are
$\{\varepsilon_j\}_{j\in\mathbb{N}}$ in $(0,1)$,
$\{R_j\}_{j\in\mathbb{N}}:=\{R_{\varepsilon_j}\}_{j\in\mathbb{N}}$
in $(1,\infty)$,
$\{B_j\}_{j\in\mathbb{N}}$ in $\operatorname{SL}(n)$, and $\{v_j\}_{j\in\mathbb{N}}$
in $\mathcal{M}_{\rm aff}$
such that, when $j\to\infty$,
$\varepsilon_j\to0$ and $f_j:=f_{\varepsilon_j}$ satisfies
\begin{align}\label{lem-d-e3}
\frac1{\varepsilon_j}
\left\|\nabla(f_j\circ B_j-v_j)\right\|_p
\to0.
\end{align}
Applying this, the Sobolev inequality [\eqref{eq-Sob}], and a change of variables,
we find that,
for every $j\in\mathbb{N}$, there is a tuple of affine transformation parameters
$(c_j,\lambda_j,S_j,x_j)\in\mathscr{A}$ such that
\begin{align}\label{lem-d-e4}
\left\|f_j\circ B_j-v_j\right\|_{p^*}
=\left\|T_{\lambda_j,S_j,x_j}f_j-c_j U\right\|_{p^*}\to0
\end{align}
as $j\to\infty$.
Since $T_{\lambda_j,S_j,x_j}$ is an isometry on $L^{p^*}$, we deduce that,
when $j\to\infty$,
\begin{align*}
\left\|T_{\lambda_j,S_j,x_j}U-c_j U\right\|_{p^*}
&\le
\left\|T_{\lambda_j,S_j,x_j}f_j-c_j U\right\|_{p^*}
+
\varepsilon_j
\left\|T_{\lambda_j,S_j,x_j}\varphi(\cdot-R_j \mathbf{e}_1)\right\|_{p^*}
\\
&=
\left\|T_{\lambda_j,S_j,x_j}f_j-c_jU\right\|_{p^*}
+
\varepsilon_j\|\varphi\|_{p^*}
\to0.
\end{align*}
By this and Corollary \ref{cor-escape},
we find that, when $j\to\infty$,
\begin{align}\label{lem-d-e2}
c_j\to 1,\quad\lambda_j\to1,\quad
S_j\to{\bf 0}, \quad\text{and}\quad x_j\to{\bf 0}.
\end{align}

For each $j\in\mathbb{N}$, define
$
E_j:=x_j+\lambda_je^{S_j}B(R_j{\bf e}_1,1).$
Applying \eqref{lem-d-e2}, we conclude that,
for sufficiently large $j\in\mathbb{N}$,
$E_j\subset\{y\in\mathbb{R}^n: |y|\ge \frac{R_j}4\}$.
Therefore, by \eqref{lem-d-e1} and a change of variables, we obtain, for
sufficiently large $j\in\mathbb{N}$, $\|U\|_{L^{p^*}(E_j)}
\lesssim\varepsilon_j^2$ and
\begin{align*}
\left\|T_{\lambda_j,S_j,x_j}U\right\|_{L^{p^*}(E_j)}
=\|U\|_{L^{p^*}(B(R_j\mathbf{e}_j,1))}
\lesssim\varepsilon_j^2;
\end{align*}
from these, the Sobolev inequality, and \eqref{lem-d-e4},
it follows that
\begin{align*}
\left\|\nabla\left(f_j\circ B_j-v_j\right)\right\|_p
&\gtrsim\left\|f_j\circ B_j-v_j\right\|_{p^*}\ge
\left\|T_{\lambda_j,S_j,x_j}f_j-c_jU\right\|_{L^{p^{*}}(E_j)}
\\
&\ge
\varepsilon_j
\left\|T_{\lambda_j,S_j,x_j}\varphi_{\varepsilon_j}
\right\|_{L^{p^*}(E_j)}
-
\left\|T_{\lambda_j,S_j,x_j}U-c_jU\right\|_{L^{p^*}(E_j)}
\\
&\ge\varepsilon_j\|\varphi\|_{p^*}
-\|U\|_{L^{p^*}(E_j)}-|c_j|\|U\|_{L^{p^*}(B(R_j\mathbf{e}_j,1))}
\gtrsim\varepsilon_j.
\end{align*}
This contradicts \eqref{lem-d-e3} and hence
completes the present proof.
\end{proof}

Based on these two lemmas, we can immediately conclude
the sharpness of the exponents $p$ and $p-1$
for the $\dot{W}^{1,p}$ norm, respectively,
in Theorems \ref{thm-stab} and \ref{thm-bubble}.

\subsection{Sharpness of the exponents for the weighted Sobolev norm}

In this subsection, we focus on the weighted quadratic term in
Theorem \ref{thm-stab}. To this end, further assume that, for every
$\psi\in T_U\mathcal{M}_{\rm aff}$,
\begin{align}\label{eq-second}
\int_{B({\bf0},1)}
U^{p^*-2}\varphi\psi=0.
\end{align}
Indeed, we can choose $\varphi$ in the orthocomplement
of $T_U\mathcal{M}_{\rm aff}$ in $L^2(B({\bf0},1),U^{p^*-2})$.
For any $\varepsilon\in(0,1)$,
define
$g_\varepsilon:=U+\varepsilon \varphi$
and
\begin{align*}
d_\varepsilon:=
\inf_{v\in \mathcal M_{\rm aff},\, B\in \operatorname{SL}(n)}\left\{\|\nabla(f_\varepsilon\circ B -v)\|_{p}^p
+\int_{\mathbb{R}^n}
|\nabla v|^{p-2}|\nabla(f_\varepsilon\circ B- v)|^2
\right\}.
\end{align*}

\begin{lemma}\label{lem-second}
Let $p\in(2,n)$.
Then there is a positive constant $c$, depending only on $n$, $p$, $U$, and
$\varphi$, such that, for every $\varepsilon\in(0,c)$,
$d_\varepsilon\sim\varepsilon^2$
with the positive equivalence constants
depending only on $n$, $p$, $U$, and $\varphi$,
and the main part in $d_\varepsilon$ is
the quadratic term.
\end{lemma}

\begin{proof}
In the present proof,
unless otherwise specified, the implicit positive constants
depend only on $n$, $p$, $U$, and $\varphi$.
By the definition and H\"older's inequality, we find that, for every $\varepsilon\in(0,1)$,
\begin{align}\label{eq-second-e1}
d_\varepsilon
&\le\left\|\nabla(g_\varepsilon-U)\right\|_p^p
+\int_{\mathbb{R}^n}
|\nabla U|^{p-2}|\nabla (g_\varepsilon-U)|^2\notag\\
&\le \left\|\nabla(g_\varepsilon-U)\right\|_p^p
+\|\nabla U\|_p^{p-2}
\left\|\nabla(g_\varepsilon-U)\right\|_p^2
\lesssim\varepsilon^p+\varepsilon^2\lesssim\varepsilon^2.
\end{align}

Now, we consider the lower estimate by contradiction.
Assume that the lower bound is not true.
Then there are $\{\varepsilon_j\}_{j\in\mathbb{N}}$
in $(0,1)$, $\{B_j\}_{j\in\mathbb{N}}$ in $\operatorname{SL}(n)$,
and $\{\varphi_j\}_{j\in\mathbb{N}}$ in $\mathcal{M}_{\rm aff}$ such that,
when $j\to\infty$,
$\varepsilon_j\to0$ and, for every $j\in\mathbb{N}$,
$g_j:=g_{\varepsilon_j}$ satisfies
\begin{align}\label{eq-second-e2}
\|\nabla(g_j\circ B_j-v_j)\|_{p}^p
+\int_{\mathbb{R}^n}
|\nabla v_j|^{p-2}|\nabla(g_j\circ B_j-v_j)|^2
<\frac{1}{j}\varepsilon_j.
\end{align}
By this and the Sobolev inequality \eqref{eq-Sob},
we have, for any sufficiently large $j\in\mathbb{N}$,
$\|v_j\|_{p^*}\sim\|g_j\circ B_j\|_{p^*}=\|g_j\|_{p^*}\sim1$;
hence, from \cite[Lemma 3.1]{fz22}, it follows that
the embedding $\dot{W}^{1,2}(|\nabla v_j|^{p-2})\hookrightarrow
L^2(|v_j|^{p^*-2})$ holds with the implicit positive
constant depending only on $n$ and $p$.
Using this, \eqref{eq-second-e2}, the Sobolev inequality \eqref{eq-Sob},
and a change of variables, we obtain,
for every sufficiently large $j\in\mathbb{N}$, there is a tuple
of affine transformation parameters
$(c_j,\lambda_j,S_j,x_j)\in\mathscr{A}$ such that
\begin{align}\label{eq-second-e4}
&\|T_{\lambda_j,S_j,x_j}g_j-c_jU\|_{p^*}^p
+\int_{\mathbb{R}^n}
|c_j U|^{p^*-2}|T_{\lambda_j,S_j,x_j}g_j-c_jU|^2\notag\\
&\quad=\|g_j\circ B_j-v_j\|_{p^*}^p
+\int_{\mathbb{R}^n}
|v_j|^{p-2}|g_j\circ B_j-v_j|^2
\lesssim\frac{1}{j}\varepsilon_j.
\end{align}
Applying this and an argument similar to that used in the proof
of Lemma \ref{lem-pterm},
we find that, when $j\to\infty$,
$(c_j,\lambda_j,S_j,x_j)\to(1,1,{\bf0},{\bf 0})$ in $\mathscr{A}$.
This implies that $U_j:=c_jT_{\lambda_j,S_j,x_j}^{-1}U$
converge to $U$ uniformly in $B({\bf0},1)$. Thus,
for sufficiently large $j$,
$U_j\ge\frac12U$ holds in $B({\bf0},1)$.
Moreover, using Taylor's formula for the parameter space, we obtain,
when $j\to\infty$,
\begin{align*}
U_j=U+\sum_\ell {\bf a}_{j,\ell}\psi_\ell
+o(|\mathbf{a}_j|),
\end{align*}
where ${\bf a}_{j,\ell}$ is one component of
${\bf a}_j:=(c_j,\lambda_j,S_j,x_j)$ and $\psi_\ell$ is the corresponding
element in the tangent space $T_U\mathcal{M}_{\rm aff}$.
From this, the above uniform embedding
$\dot{W}^{1,2}(|\nabla v_j|^{p-2})\hookrightarrow
L^2(|v_j|^{p^*-2})$, \eqref{eq-second-e4},
and a change of variables,
it follows that, for sufficiently large $j$,
\begin{align*}
\int_{\mathbb{R}^n}
|\nabla v_j|^{p-2}|\nabla(g_j\circ B_j-v_j)|^2
&\gtrsim\int_{\mathbb{R}^n}|v_j|^{p-2}|g_j\circ B_j-v_j|^2\\
&=\int_{\mathbb{R}^n}
(c_j U)^{p^*-2}|T_{\lambda_j,S_j,x_j}g_j-c_jU|^2\\
&\ge \int_{B({\bf0},1)}U_j^{p^*-2}|g_j-U_j|^2\\
&\gtrsim\int_{B({\bf0},1)}
U^{p^*-2}\left(\varepsilon_j\varphi-\sum_\ell
{\bf a}_{j,\ell}\psi_\ell\right)^2+o(|{\bf a}_j|^2).
\end{align*}
This, combined with the orthogonality condition \eqref{eq-second},
further implies, for sufficiently large $j\in\mathbb{N}$,
\begin{align*}
\int_{\mathbb{R}^n}
|\nabla v_j|^{p-2}|\nabla(g_j\circ B_j-v_j)|^2
\gtrsim\varepsilon_j^2
\int_{B({\bf0},1)}U^{p^*-2}\varphi^2
+\int_{B({\bf0},1)}U^{p^*-2}\left(\sum_\ell
{\bf a}_{j,\ell}\psi_\ell\right)^2
\gtrsim\varepsilon_j^2,
\end{align*}
which contradicts \eqref{eq-second-e2}.
Therefore, we obtain the lower estimate $d_\varepsilon\gtrsim
\varepsilon^2$.
Furthermore,
by the estimation of \eqref{eq-second-e1}, we find that
the first term is controlled by $\varepsilon^p$.
Hence, since $p>2$,
the second term (namely the quadratic term) is the main part.
This completes the present proof.
\end{proof}

Now, we have constructed a family of functions such that the distance
$d_\varepsilon$ is dominated by the quadratic term and is of order $\varepsilon^2$. Therefore,
to prove the sharpness of the quadratic term, we only need to
verify ${\mathcal{E}_p(g_\varepsilon)}^p
-S_{\rm aff}^p\|g_\varepsilon\|_{p^*}^p\lesssim\varepsilon^2$.
By \eqref{eq-holder}, it suffices to show
\begin{align}\label{eq-second-e3}
\delta(\varepsilon):=
\|\nabla g_\varepsilon\|_p^p
-S_{\rm Sob}^p\|g_\varepsilon\|_{p^*}^p\lesssim\varepsilon^2.
\end{align}
Indeed, since $U$ is the extremal function of \eqref{eq-Sob}
and $g_\varepsilon=U+\varepsilon\varphi$, we infer that
$\delta(0)=0=\delta'(0)$. This shows \eqref{eq-second-e3}
and hence the sharpness.

Finally, we consider the sharpness of the exponent
for the weighted Sobolev norm in \eqref{thm-bubble-e6}.
For every $\varepsilon\in(0,1)$, let $f_\varepsilon$ be the same as
in Subsection \ref{subsec-61};
define
\begin{align*}
\widetilde{d}_{\varepsilon}:=\inf&\left\{\left[\int_{\mathbb{R}^n}
|\nabla U|^{p-2}|\nabla (T_{\lambda,S,x}f_\varepsilon-U)|^2\right]^{
\frac{p-1}{p}}:\lambda\in(0,\infty),
\ S=S^T,\ \operatorname{tr}S=0,\ x\in\mathbb{R}^n\right\}.
\end{align*}
Then the following lemma gives the desired sharpness.

\begin{lemma}\label{lem-bubble-second}
Let $p\in(2,n)$.
For every $\varepsilon\in(0,1)$,
let $R_\varepsilon:=\varepsilon^{-\frac{p-1}{n-1}}$.
Then there is a positive constant $c$, depending only on $n$, $p$, $U$, and
$\varphi$, such that, for every $\varepsilon\in(0,c)$,
$$\delta_{\rm aff}(f_\varepsilon)\lesssim\varepsilon^{p-1}\quad
\text{and}\quad\widetilde{d}_\varepsilon\sim\varepsilon^{p-1}$$
with the implicit positive constants
depending only on $n$, $p$, $U$, and $\varphi$.
\end{lemma}

\begin{proof}
We first prove the estimate on $\delta_{\rm aff}(f_\varepsilon)$.
Note that we can not apply Lemma \ref{lem-tail-1} directly.
Indeed, by the construction of $U$, we have, when $|x|$ is sufficiently large,
\begin{align*}
|U(x)|\sim|x|^{-\frac{n-p}{p-1}},
\quad|\nabla U(x)|\sim|x|^{-\frac{n-1}{p-1}},
\quad|\nabla^2U(x)|\sim|x|^{-\frac{n-1}{p-1}-1};
\end{align*}
combining these and the definition $R_\varepsilon:=\varepsilon^{-\frac{p-1}{n-1}}$,
we can actually conclude that, if $\varepsilon$ is sufficiently small,
then, for every $x\in B(R_\varepsilon \mathbf{e}_1, 2)$,
\begin{align}\label{lem-bubble-second-e1}
|\nabla U(x)|\sim\varepsilon
\quad\text{and}\quad
|\nabla^2U(x)|\lesssim\varepsilon,
\end{align}
but $|U(x)|\sim\varepsilon^{\frac{n-p}{n-1}}>\varepsilon$. Therefore,
the sequence $\{R_\varepsilon\}_{\varepsilon\in(0,1)}$
considered here does not satisfied \eqref{eq:small-tail}.
Fortunately, the idea used in the proof of Lemma \ref{lem-tail-1} still valids.
Observe that $\nabla U$ and $\nabla^2U$ have satisfied the desired estimates.
Furthermore, the decreasing assumption on $U$
in Lemma \ref{lem-tail-1} is only used to obtain \eqref{eq:zero-order-bound-e1}.
Repeating the proof of \eqref{eq:zero-order-bound-e1}, we
obtain the following more exact estimate:
\begin{align*}
\left||f_\varepsilon|^{p^*-2}f_\varepsilon-U^{p^*-1} \right| \lesssim
|\varepsilon\varphi_\varepsilon|^{p^*-1}+U^{p^*-2}|\varepsilon\varphi_\varepsilon|
\lesssim
\varepsilon^{p^*-1}{\bf 1}_{B(R_\varepsilon\mathbf{e}_1,1)}
+\varepsilon U^{p^*-2}{\bf 1}_{B(R_\varepsilon\mathbf{e}_1,1)}.
\end{align*}
From the definition $R_\varepsilon=\varepsilon^{-\frac{p-1}{n-1}}$,
it follows that, if $\varepsilon$ is sufficiently small,
then, for every $x\in B(R_\varepsilon \mathbf{e}_1, 2)$,
\begin{align*}
\varepsilon U(x)^{p^*-2}
\sim \varepsilon\left(\varepsilon^{-\frac{p-1}{n-1}}\right)^{-\frac{n-p}{p-1}(p^*-2)}
=\varepsilon^{\frac{np-n+2p-1}{n-1}}< \varepsilon^{p-1},
\end{align*}
which, combined with $p^*>p$, then implies that
$\| |f_\varepsilon|^{p^*-2}f_\varepsilon-U^{p^*-1} \|_q \lesssim\varepsilon^{p-1}.$
Substituting this estimate into the proof of Lemma \ref{lem-tail-1}, we can still obtain
$\delta_{\rm aff}(f_\varepsilon)\lesssim\varepsilon^{p-1}$ for sufficiently
small $\varepsilon$.

Now, we deal with the estimate on $\widetilde{d}_\varepsilon$.
Using the support assumption on $\varphi$ and using
\eqref{lem-bubble-second-e1}, we find that, for every $\varepsilon\in(0,1)$,
\begin{align*}
\widetilde{d}_\varepsilon\le
\left(\int_{\mathbb{R}^n}|\nabla U|^{p-2}|\varepsilon\nabla
\varphi_\varepsilon|^2\right)^{\frac{p-1}{p}}
=\left[\int_{B(R_\varepsilon \mathbf{e}_1,1)}|\nabla U|^{p-2}
|\varepsilon\nabla\varphi(x-R_\varepsilon\mathbf{e}_1)|^2\,dx\right]^{\frac{p-1}{p}}
\sim\varepsilon^{p-1}.
\end{align*}
Finally, we show the lower estimate by contradiction.
Assume that the lower estimate is not true.
Then there are $\{\varepsilon_j\}_{j\in\mathbb{N}}$
in $(0,1)$,
$\{\lambda_j\}_{j\in\mathbb{N}}$ in $(0,\infty)$,
symmetric matrices $\{S_j\}_{j\in\mathbb{N}}$ with
$\operatorname{Tr}S_j=0$ for all $j\in\mathbb{N}$,
and $\{x_j\}_{j\in\mathbb{N}}$ in $\mathbb{R}^n$
such that, when $j\to\infty$,
$\varepsilon_j\to0$ and, for every $j\in\mathbb{N}$,
$f_j:=f_{\varepsilon_j}$ satisfies
\begin{align}\label{lem-bubble-second-e2}
\int_{\mathbb{R}^n}
|\nabla U|^{p-2}|\nabla(T_{\lambda_j,S_j,x_j}f_j-U)|^2
<\frac{1}{j}\varepsilon_j^p.
\end{align}
Applying this and the embedding $\dot{W}^{1,2}(|\nabla U|^{p-2})\hookrightarrow L^2(U^{p^*-2}) $,
we obtain, for every $j\in\mathbb{N}$,
\begin{align*}
\int_{\mathbb{R}^n}U^{p^*-2}
\left|T_{\lambda_j,S_j,x_j}U-U\right|^2
&\lesssim\int_{\mathbb{R}^n}
U^{p^*-2}\left|T_{\lambda_j,S_j,x_j}f-U\right|^2\\
&\qquad+\int_{\mathbb{R}^n}U^{p^*-2}
\left|T_{\lambda_j,S_j,x_j}(f-U)\right|^2\\
&\lesssim\frac{1}{j}\varepsilon_j^p
+\int_{\mathbb{R}^n}U^{p^*-2}
\left|T_{\lambda_j,S_j,x_j}(f-U)\right|^2.
\end{align*}
From a change of variables, we infer that, for any $j\in\mathbb{N}$,
\begin{align*}
\int_{\mathbb{R}^n}U^{p^*-2}
\left|T_{\lambda_j,S_j,x_j}(f-U)\right|^2
=\varepsilon_j^2\int_{\mathbb{R}^n}
|T_{\lambda_j,S_j,x_j}^{-1}U|^{p^*-2}|\nabla\varphi_{\varepsilon_j}|^2
\lesssim\varepsilon_j^2.
\end{align*}
Therefore, when $j\to\infty$,
$\int_{\mathbb{R}^n}U^{p^*-2}|T_{\lambda_j,S_j,x_j}U-U|^2\to0$.
This further implies that $\{T_{\lambda_j,S_j,x_j}U\}_{j\in\mathbb{N}}$ converges to $U$
locally in measure. Then, repeating the proof of
Corollary \ref{cor-escape}, we conclude that, when $j\to\infty$,
$\lambda_j\to1$,
$S_j\to{\bf 0}$, and $x_j\to{\bf 0}$. For given sufficiently large $j\in\mathbb{N}$,
define $R_j:=R_{\varepsilon_j}$ and $E_j:=x_j+\lambda_je^{S_j}B(R_j\mathbf{e}_1,1)$.
By the above convergence conditions, we have
$E_j\subset B(R_j\mathbf{e}_1,2)$.
Thus, applying \eqref{lem-bubble-second-e1}, we find that
\begin{align*}
&\left[\int_{\mathbb{R}^n}
|\nabla U|^{p-2}|\nabla(T_{\lambda_j,S_j,x_j}f_j-U)|^2\right]^{\frac12}\\
&\quad\gtrsim \varepsilon_{j}^{\frac{p}{2}-1}
\left[\int_{E_j}|\nabla(T_{\lambda_j,S_j,x_j}f_j-U)|^2\right]^{\frac12}\\
&\quad\ge \varepsilon_j^{\frac p2} \left\{\left[\int_{E_j}
|\nabla(T_{\lambda_j,S_j,x_j}\varphi_{\varepsilon_j})|^2\right]^{\frac12}
-\varepsilon_j^{-1}\left[\int_{E_j}|\nabla(T_{\lambda_j,S_j,x_j}U-U)|^2\right]^{\frac12}\right\}\\
&\quad=:\varepsilon_j^{\frac p2}\left(I_{1,j}-I_{2,j}\right).
\end{align*}
Using the change of variables $x:=T_j(y):=x_j+\lambda_j e^{S_j}
(R_j\mathbf{e}_1+y)$ with $y\in B({\bf 0},1)$, and using
the above convergence conditions on parameters, we obtain, when $j\to\infty$,
\begin{align*}
\int_{E_j}
|\nabla(T_{\lambda_j,S_j,x_j}\varphi_{\varepsilon_j})|^2
=\lambda_j^{n(1-\frac 2p)}\int_{B({\bf0},1)}\left|e^{-S_j}\nabla\varphi(y)\right|^2\,dy
\to\|\nabla\varphi\|_{p}^2.
\end{align*}
Consequently, when $j$ is sufficiently large, $I_{1,j}\gtrsim 1$. In addition,
from the construction of $U$ and the definition $R_j=\varepsilon_j^{-\frac{p-1}{n-1}}$, it follows that there is
a positive constant $c_{n,p}$, depending only on $n$ and $p$, such that,
for every $x\in\mathbb{R}^n$,
\begin{align*}
\varepsilon_j^{-1}\nabla U(x)
=-c_{n,p}\left(\frac{1}{R_j^{p'}}+\left|\frac{x}{R_j}\right|^{p'}\right)^{-\frac np}
\left|\frac{x}{R_j}\right|^{p'-2}\frac{x}{R_j}.
\end{align*}
This implies that, when $j\to\infty$,
\begin{align*}
\varepsilon_j^{-1}\nabla U(y+R_j\mathbf{e}_1)
\to -c_{n,p}\mathbf{e}_1\quad\text{and}\quad
\varepsilon_j^{-1}\nabla U(T_j(y))\to -c_{n,p}\mathbf{e}_1
\end{align*}
uniformly on $B({\bf0},1)$. Combining this, a change of variables, and
the above convergence conditions, we obtain, when $j\to\infty$,
\begin{align*}
I_{j,2}
=\lambda_j^{\frac n2}
\left[\int_{B({\bf0},1)}\left|\lambda_j^{-\frac np}
e^{-S_j}\left(\varepsilon_j^{-1}\nabla U(y+R_j\mathbf{e}_1)\right)
-\varepsilon_j^{-1}\nabla U(T_j(y))\right|^2\,dy\right]^{\frac12}\to0.
\end{align*}
From this and the estimate on $I_{j,1}$, we infer that, for sufficiently large $j$,
\begin{align*}
\int_{\mathbb{R}^n}
|\nabla U|^{p-2}|\nabla(T_{\lambda_j,S_j,x_j}f_j-U)|^2
\gtrsim\varepsilon_j^p,
\end{align*}
which contradicts \eqref{lem-bubble-second-e2}.
Hence, the desired lower bound is true, which concludes the proof.
\end{proof}


\begin{appendices}

\section{Vanishing under escaping affine transformations}

We record the following elementary compactness observation.
While the corresponding statements for translations and dilations are
standard in concentration-compactness arguments, the appearance
of anisotropic affine deformations in our setting requires us to state a version
for general linear maps. Recall that
the general affine transformation $\mathcal{T}$ is deinfed
in \eqref{def-tran-g}.

\begin{lemma}\label{lem-escape}
Let $q\in[1,\infty)$ and $f\in L^q({\mathbb R}^n)$.  Let
$\{\lambda_j\}_{j\in\mathbb{N}}\in(0,\infty)$,
$\{A_j\}_{j\in\mathbb{N}}$ in
$\operatorname{SL}(n)$, and $\{x_j\}_{j\in\mathbb{N}}$
in ${\mathbb R}^n$ be such that, when $j\to\infty$,
\begin{align*}
|\log\lambda_j|+
\|A_j\|_{\mathcal{L}(\mathbb{R}^n)}+|x_j|\to\infty .
\end{align*}
Then ${\mathcal T}_{\lambda_j,A_j,x_j}^{(q)}f\to0$
locally in measure as $j\to\infty$.
\end{lemma}

\begin{proof}
By the density argument,
we can assume $f\in C_{\rm c}$.
To prove the present lemma for such $f$, we only need to show
that, for every compact set $K\subset {\mathbb R}^n$ and
every $t\in(0,\infty)$,
\begin{align}\label{lem-escape-e1}
\lim_{j\to\infty}
\left|\left\{x\in K:
\left|{\mathcal T}_{\lambda_j,A_j,x_j}^{(q)}f(x)\right|>t\right\}\right|=0.
\end{align}
Assume that both
$\operatorname{supp}f$ and $K$ are contained in
$B({\bf0},R)$ for some $R\in(0,\infty)$.
Moreover, it is enough to prove \eqref{lem-escape-e1}
along subsequences. Indeed,
if \eqref{lem-escape-e1} does not hold,
then there are a subsequence, still denoted by $\{(\lambda_j,A_j,x_j)\}_{j\in\mathbb{N}}$,
and an $\varepsilon_0\in(0,\infty)$ such that, for all $j$,
\begin{align}\label{lem-escape-e2}
\left|\left\{x\in K:
\left|{\mathcal T}_{\lambda_j,A_j,x_j}^{(q)}f(x)\right|>t\right\}\right|\ge\varepsilon_0.
\end{align}
The argument below for subsequences shows that the left-hand side of
\eqref{lem-escape-e1} converges to zero along a further subsequence,
which contradicts \eqref{lem-escape-e2}.
Thus, after passing to a
subsequence and still denoting it by $\{(\lambda_j,A_j,x_j)\}_{j\in\mathbb{N}}$,
it suffices to consider the following four cases.

\emph{Case 1.} $\lim_{j\to\infty}
\lambda_j=\infty.$
In this case, for any $x\in\mathbb{R}^n$,
\begin{align*}
\left|{\mathcal T}_{\lambda_j,A_j,x_j}^{(q)}f(x)\right|
\le
\lambda_j^{-\frac nq}\|f\|_\infty
\to0
\end{align*}
as $j\to\infty$.
This further implies \eqref{lem-escape-e1}.

\emph{Case 2.} $\lim_{j\to\infty}
\lambda_j=0.$ In this case, by the support assumption,
we find that, for all $j\in\mathbb{N}$,
$\operatorname{supp}{\mathcal T}_{\lambda_j,A_j,x_j}^{(q)}f
\subset x_j+\lambda_jA_jB({\bf 0},R)$.
Applying this, we further obtain, when $j\to\infty$,
\begin{align*}
\left|\left\{x\in K:
\left|{\mathcal T}_{\lambda_j,A_j,x_j}^{(q)}f(x)\right|>t\right\}\right|
\le
\left|x_j+\lambda_jA_jB({\bf 0},R)\right|
=
\lambda_j^n |B({\bf 0},R)|
\to0,
\end{align*}
which proves \eqref{lem-escape-e1} in this case.

\emph{Case 3.} The sequence
$\{|\log\lambda_j|\}_{j\in\mathbb{N}}$ is bounded, but
\begin{align}\label{lem-escape-e3}
\lim_{j\to\infty}\|A_j\|_{\mathcal{L}(\mathbb{R}^n)}=\infty.
\end{align}  Let
$s_j^{(1)}\ge \cdots \ge s_j^{(n)}>0$ denote the singular value of $A_j$.
Then
$$s_j^{(1)}=\|A_j\|_{\mathcal{L}(\mathbb{R}^n)}
\quad\text{and}\quad
s_{j}^{(n)}=\frac{1}{\|A^{-1}_j\|_{\mathcal{L}(\mathbb{R}^n)}}
=\frac{1}{\|A^{-T}_j\|_{\mathcal{L}(\mathbb{R}^n)}}
=\min_{|v|=1}|A_j^{T}v|.$$
Applying this,
we conclude that, for every $j\in\mathbb{N}$,
\begin{align*}
1=|\det A_j|
=\prod_{k=1}^{n}s_j^{(k)}
\ge s_{j}^{(1)}\left[s_j^{(n)}\right]^{n-1}
=\|A_j\|_{\mathcal{L}(\mathbb{R}^n)}
\left[s_j^{(n)}\right]^{n-1}.
\end{align*}
From this and \eqref{lem-escape-e3}, it follows that
$\lim_{j\to\infty}s_j^{(n)}=0$.
Moreover, there is a unit vector $\nu_j$ such that
$|A_j^T\nu_j|=s_j^{(n)}$,
which implies that, for every $j\in\mathbb{N}$ and
$y\in x_j+\lambda_jA_jB({\bf 0},R)$,
$|\nu_j\cdot(y-x_j)|
\le\lambda_j
Rs_j^{(n)}.$
 Hence, by the support assumption and the boundedness of $\{|\log\lambda_j|\}_{j\in\mathbb{N}}$,
we obtain, when $j\to\infty$,
\begin{align*}
\left|\left\{x\in K:
\left|{\mathcal T}_{\lambda_j,A_j,x_j}^{(q)}f(x)\right|>t\right\}\right|
&\le \left|B({\bf 0},R)\cap \left[x_j+\lambda_jA_jB({\bf 0},R)\right]\right|\\
&\le\left|B({\bf 0},R)\cap\left\{y: |\nu_j\cdot(y-x_j)|\le \lambda_jRs_j^{(n)}\right\}\right|
\lesssim s_j^{(n)}\to0.
\end{align*}
Then the desired convergence holds.

\emph{Case 4.}
The sequences $\{|\log\lambda_j|\}_{j\in\mathbb{N}}$
and $\{\|A_j\|_{\mathcal{L}(\mathbb{R}^n)}\}_
{j\in\mathbb{N}}$ are both bounded.
In this case, it holds that
$\lim_{j\to\infty}|x_j|=\infty$.
In addition, there is an $M\in(0,\infty)$ such that,
for all $j\in\mathbb{N}$,
$x_j+\lambda_j A_jB({\bf 0},R)\subset x_j+B({\bf0},MR)$.
Thus, if $j$ is sufficiently large,
then $K\cap [x_j+\lambda_jA_jB({\bf 0},R)]=\emptyset$ and hence
$$\left\{x\in K:|{\mathcal T}_{\lambda_j,A_j,x_j}^{(q)}f(x)|>t\right\}=\emptyset.$$
This proves the desired convergence in \eqref{lem-escape-e1}.

Combining the above four cases, we complete the present proof.
\end{proof}

As a consequence, we have the following
convergence of bubbles in terms of parameters.

\begin{corollary}\label{cor-escape}
Let $p\in(1,n)$.  Assume that the
sequence of tuples of affine transformation parameters
$\{(c_j,\lambda_j,S_j,x_j)\}_{j\in\mathbb{N}}\in\mathscr{A}$
satisfies
\begin{align*}
\lim_{j\to\infty}\left\|T_{\lambda_j,S_j,x_j}U-c_j U\right\|_{p^*}=0.
\end{align*}
Then, when $j\to\infty$,
$(c_j,\lambda_j,S_j,x_j)\to(1,1,{\bf0},{\bf0})$
in $\mathscr{A}$.
\end{corollary}

\begin{proof}
Since the $L^{p^*}$ norm is invariant under $T_{\lambda_j,S_j,x_j}$,
we infer that $c_j\to 1$ as $j\to\infty$.
Hence, in what follows, we can assume $c_j=1$ for every
$j\in\mathbb{N}$.
For every $j\in\mathbb{N}$, define $A_j:=e^{S_j}$.  Then
$T_{\lambda_j,S_j,x_j}U={\mathcal T}_{\lambda_j,A_j,x_j}^{(p^*)}U.$
If
\begin{align*}
|\log\lambda_j|+
\|A_j\|_{\mathcal{L}(\mathbb{R}^n)}+|x_j|\to\infty
\end{align*}
as $j\to\infty$,
then, using Lemma \ref{lem-escape}, we
find that $\{T_{\lambda_j,S_j,x_j}U\}_{j\in\mathbb{N}}$ converges to zero
locally in measure.  On the other hand, the convergence of
$\{T_{\lambda_j,S_j,x_j}U\}_{j\in\mathbb{N}}$
in $L^{p^*}$ implies the
local convergence in measure to $U$.  This is impossible because $U$ is not identically
zero. Therefore, there are a positive constant $M$ and
a subsequence, still denoted by
$\{(\lambda_j,S_j,x_j)\}_{j\in\mathbb{N}}$, such that,
for every $j\in\mathbb{N}$,
\begin{align*}
|\log\lambda_j|+
\|S_j\|_{\mathcal{L}(\mathbb{R}^n)}+|x_j|\le M.
\end{align*}
From this, we further deduce that
there are a subsequence, still denoted by
$\{(\lambda_j,S_j,x_j)\}_{j\in\mathbb{N}}$,
and $(\lambda,S,x)$ with $\lambda>0$ such that,
when $j\to\infty$,
$\lambda_j\to\lambda$, $S_j\to S$, and
$x_j\to x$.
By the Lebesgue dominated convergence theorem
and the construction of $U$,
we then conclude that
\begin{align*}
T_{\lambda_j,S_j,x_j}U\to T_{\lambda,S,x}U
\end{align*}
strongly in $L^{p^*}$. Therefore, $T_{\lambda,S,x}U=U$
and hence $\lambda=1$,
$S={\bf 0}$, and $x={\bf 0}$.
This establishes the conclusion of Corollary \ref{cor-escape} for a subsequence.
Then, arguing by contradiction similar to the proof of Lemma \ref{lem-escape},
we find that the conclusion actually holds for the full sequence, which completes the proof.
\end{proof}

\end{appendices}

\begin{appendices}

\section{Nonlinear Brezis--Lieb lemmas}\label{App-B}

Recall that the well-known Brezis--Lieb lemma (see \cite[Theorem 1.9]{ll01}):
If $(\Omega,\mu)$ is a measure space,
$p\in(0,\infty)$, and $\{f_k\}_{k\in\mathbb{N}}$ is a bounded sequence in
$L^p(\Omega)$ such that $f_k\to f$ almost everywhere as $k\to\infty$, then
\begin{align}\label{eq-BL}
\lim_{k\to\infty}\left\||f_k|^p-|f|^p-|f-f_k|^p\right\|_{L^1(\Omega)}=0.
\end{align}

In this section, we establish the following two nonlinear Brezis--Lieb lemmas,
which play an important role in our proof of the relative compactness.

\begin{lemma}\label{lem-BL}
Let $(\Omega,\mu)$ be a measure space.
\begin{enumerate}[{\rm(i)}]
  \item\label{it1-BL} If $p\in(1,\infty)$,
  $\{f_k\}_{k\in\mathbb{N}}$ is a bounded sequence in
  $L^p(\Omega)$, and $f_k\to f$ almost everywhere as $k\to\infty$,
  then
  \begin{align*}
  \lim_{k\to\infty}\left\||f_k|^{p-2}f_k-|f|^{p-2}f-|f_k-f|^{p-2}(f_k-f)
  \right\|_{L^{p'}{(\Omega)}}=0.
  \end{align*}
  \item\label{it2-BL} Further assume that $(\Omega,\mu)$ is a finite measure space and
  that the sequences
  $\{a_k\}_{k\in\mathbb{N}}$ and $\{b_{k}\}_{k\in\mathbb{N}}$ of nonnegative measurable functions
  and the measurable function $a$ satisfy that there are positive constants
  $c_0,c_1,c_2$ such that $a\ge c_0$ and $c_1\leq a_k\leq c_2$ for any
  $k\in\mathbb{N}$. If $\|a_k-a-b_k\|_{L^1(\Omega)}\to0$ as $k\to\infty$,
  then, for any $\alpha\in(0,\infty)$,
  \begin{align*}
  \lim_{k\to\infty}\left|\left(\int_{\Omega}a_k^{-\alpha}\right)^{-\frac1\alpha}
  -\left[\int_{\Omega}(a+b_k)^{-\alpha}\right]^{-\frac1\alpha}\right|=0.
  \end{align*}
\end{enumerate}
\end{lemma}

To show Lemma \ref{lem-BL}\eqref{it1-BL}, we need the following
inequalities.

\begin{lemma}\label{lem-pine}
Let $p\in(1,\infty)$.
For any $\varepsilon\in(0,\infty)$, there is a positive constant
$C_{\varepsilon}$ such that, for any $a,b\in\mathbb{C}$,
\begin{align}\label{lem-pine-e0}
\left||a+b|^p-|b|^{p}\right|
\le\varepsilon|b|^p+C_{\varepsilon}|a|^p
\end{align}
and
\begin{align}\label{lem-pine-ee}
\left||a+b|^{p-2}(a+b)-|b|^{p-2}b\right|^{p'}
\le\varepsilon|b|^p+C_{\varepsilon}|a|^p.
\end{align}
\end{lemma}

\begin{proof}
Inequality \eqref{lem-pine-e0} is well known;
see, for instance, \cite[p.\,699]{Bre02}.
Next, we prove \eqref{lem-pine-ee} by contradiction.
Indeed, assume \eqref{lem-pine-ee} does not hold.
Then there are a positive constant $\varepsilon_0$
and two complex sequences $\{a_k\}_{k\in\mathbb{N}}$ and $\{b_{k}\}_{k\in\mathbb{N}}$
such that, for all $k\in\mathbb{N}$,
\begin{align}\label{lem-pine-e1}
\left||a_k+b_k|^{p-2}(a_k+b_k)-|b_k|^{p-2}b_k\right|^{p'}
>\varepsilon_0|b_k|^p+k|a_k|^p.
\end{align}
Then $b_k$ should not be zero and hence, by the homogeneity,
we can assume $|b_k|=1$ for all $k\in\mathbb{N}$.
This, together with the triangle inequality, implies that, for any $k\in\mathbb{N}$,
\begin{align*}
\left||a_k+b_k|^{p-2}(a_k+b_k)-|b_k|^{p-2}b_k\right|^{p'}
\le\left[(|a_k|+1)^{p-1}+1\right]^{p'}.
\end{align*}
From this and \eqref{lem-pine-e1}, we infer that,
for all $k\in\mathbb{N}$,
\begin{align*}
k|a_k|^p\le\left[(|a_k|+1)^{p-1}+1\right]^{p'}.
\end{align*}
This shows that $\{a_k\}_{k\in\mathbb{N}}$ can not have a positive lower bound.
Hence, we can choose the subsequences,
still denoted by $\{a_k\}_{k\in\mathbb{N}}$ and $\{b_k\}_{k\in\mathbb{N}}$,
such that, when $k\to\infty$, $a_k\to0$ and $b_k\to b_0$ for some $b_0\in\mathbb{C}$.
Then, letting $k\to\infty$ in \eqref{lem-pine-e1},
we obtain $0\ge \varepsilon_0$, which is impossible.
This then completes the proof of the present lemma.
\end{proof}

Now, we give the proof of Lemma \ref{lem-BL}.

\begin{proof}[Proof of Lemma \ref{lem-BL}]
We first consider \eqref{it1-BL}. To this end, fix $\varepsilon\in(0,1)$.
Applying Lemma \ref{lem-pine},
we find that there are two positive constants $C_{\varepsilon,1}$
and $C_{\varepsilon,2}$ such that, for all $k\in\mathbb{N}$,
\begin{align*}
F_k:=&\,\left||f_k|^{p-2}f_k-|f|^{p-2}f-|f_k-f|^{p-2}(f_k-f)\right|^{p'}\\
\le&\,(1+\varepsilon)\left||f_k|^{p-2}f_k-|f_k-f|^{p-2}(f_k-f)\right|^{p'}
+C_{\varepsilon,1}|f|^p\quad\text{by \eqref{lem-pine-e0}}\\
\le&\,\varepsilon(1+\varepsilon)|f_k-f|^p
+\left[C_{\varepsilon,1}+(1+\varepsilon)C_{\varepsilon,2}\right]
|f|^p\quad\text{by \eqref{lem-pine-ee}}.
\end{align*}
Thus, by Fatou's lemma, we obtain
\begin{align}\label{lem-BL-e1}
\limsup_{k\to\infty}\int_{\Omega}
\left[F_k-\varepsilon(1+\varepsilon)|f_k-f|^p\right]
\le \int_{\Omega}\limsup_{k\to\infty}
\left[F_k-\varepsilon(1+\varepsilon)|f_k-f|^p\right]=0.
\end{align}
On the other hand,
from the boundedness of $\{f_k\}_{k\in\mathbb{N}}$
in $L^p(\Omega)$, it follows that there is an
$M\in(0,\infty)$ such that, for all $k\in\mathbb{N}$,
\begin{align*}
\int_{\Omega}F_k&=
\int_{\Omega}
\left[F_k-\varepsilon(1+\varepsilon)|f_k-f|^p\right]+
\varepsilon(1+\varepsilon)\int_{\Omega}|f_k-f|^p\\
&\le \int_{\Omega}
\left[F_k-\varepsilon(1+\varepsilon)|f_k-f|^p\right]+M\varepsilon.
\end{align*}
Hence, applying this and \eqref{lem-BL-e1}, we conclude
$\limsup_{k\to\infty}\int_{\Omega}F_k\le M\varepsilon$.
Then the arbitrariness of $\varepsilon$, together
with the non-negativity of $F_k$,
further implies
$\int_{\Omega}F_k\to0$ as $k\to\infty$.
This shows \eqref{it1-BL}.

Next, we show \eqref{it2-BL}. It is much simpler.
Indeed, from the assumptions, we infer that
$a_k\ge c_1$ and $a+b_k\ge c_0$ for all $k\in\mathbb{N}$.
Using this and Taylor's formula, we find that,
for any $k\in\mathbb{N}$, there is a
$t\in(0,1)$ such that
\begin{align*}
\left|a_k^{-\alpha}-(a+b_k)^{-\alpha}\right|
=\alpha|ta_k+(1-t)(a+b_k)|^{-\alpha-1}
|a_k-a-b|
\lesssim|a_k-a-b|,
\end{align*}
where the implicit positive constant
is independent of $k$ and
is uniform on $\Omega$.
Therefore, using $\lim_{k\to\infty}\|a_k-a-b\|_{L^1(\Omega)}=0$,
we obtain
\begin{align}\label{lem-BL-e2}
\lim_{k\to\infty}\left\|a_k^{-\alpha}-(a+b_k)^{-\alpha}\right\|_{L^1(\Omega)}=0.
\end{align}
Since $a_k\le c_2$ for all $k\in\mathbb{N}$,
it follows that $\int_{\Omega}a_k^{-\alpha}\ge c_2^{-\alpha}|\Omega|$.
Then, by \eqref{lem-BL-e2}, we find that
$\int_{\Omega}(a+b_k)^{-\alpha}\gtrsim 1$ uniformly
for sufficiently large $k$.
Thus, using Taylor's formula again, we conclude that, for sufficiently large $k$,
there is a $t\in(0,1)$ such that
\begin{align*}
&\left|\left(\int_{\Omega}a_k^{-\alpha}\right)^{-\frac1\alpha}
-\left[\int_{\Omega}(a+b_k)^{-\alpha}\right]^{-\frac1\alpha}\right|\\
&\quad\le\frac{1}{\alpha}
\left\{t\int_{\Omega}a_k^{-\alpha}
+(1-t)\left[\int_{\Omega}(a+b_k)^{-\alpha}\right]\right\}^{-1-\frac1\alpha}
\left\|a_k^{-\alpha}-(a+b_k)^{-\alpha}\right\|_{L^1(\Omega)}\\
&\quad\lesssim\left\|a_k^{-\alpha}-(a+b_k)^{-\alpha}\right\|_{L^1(\Omega)}
\end{align*}
with the implicit positive constant independent of $k$.
Then \eqref{lem-BL-e2} further implies
\eqref{it2-BL}.
This completes the proof of the present lemma.
\end{proof}

\end{appendices}

\begin{appendices}

\section{Asymptotic decoupling formulae}\label{App-C}

In the relative compactness argument, we need to deal with the asymptotic decoupling
of Lebesgue norms or directional derivatives.
Hence, we establish the following two general
asymptotic decoupling formulae.

\begin{lemma}
Let $p\in(1,\infty)$, $J\in\mathbb{N}$, and $\{f_j\}_{j=1}^J$
be a sequence in $L^p$.
If the sequence $$\{\mathfrak{g}_{j,k}\}_{j=1,\ldots J,\,k\in\mathbb{N}}
=\{(\lambda_{j,k},x_{j,k})\}_{j=1,\ldots J,\,k\in\mathbb{N}}$$ in
$(0,\infty)\times\mathbb{R}^n$ satisfies that, for any $j_1\ne j_2$,
$\{\mathfrak{g}_{j_1,k}\}_{k\in\mathbb{N}}$
and $\{\mathfrak{g}_{j_2,k}\}_{k\in\mathbb{N}}$ are orthogonal, then
\begin{align}\label{lem-decou-e0}
\lim_{k\to\infty}\left|\left\|\sum_{j=1}^{J}
\mathfrak{g}_{j,k}(f_j;p)\right\|_{p}^p
-\sum_{j=1}^{J}\|f_j\|_p^p\right|=0
\end{align}
and
\begin{align}\label{lem-decou-ee}
\lim_{k\to\infty}
\left\|\sum_{j=1}^{J}
|\mathfrak{g}_{j,k}(f_j;p)|^{p-2}\mathfrak{g}_{j,k}(f_j;p)
-\left|\sum_{j=1}^{J}\mathfrak{g}_{j,k}(f_j;p)\right|^{p-2}
\sum_{j=1}^{J}\mathfrak{g}_{j,k}(f_j;p)
\right\|_{p'}=0.
\end{align}
\end{lemma}

\begin{proof}
We first prove \eqref{lem-decou-e0} by induction.
This is obvious when $J=1$. Now, assume \eqref{lem-decou-e0} holds
for $J\in\mathbb{N}$ and we aim to show this for $J+1$.
For any $k\in\mathbb{N}$, define
$F_k:=\sum_{j=1}^{J}\mathfrak{g}_{j,k}(f_j;p)$;
then, by a change of variables, we have
\begin{align*}
\left\|\sum_{j=1}^{J+1}\mathfrak{g}_{j,k}(f_j;p)\right\|_p^p
=\left\|F_k+\mathfrak{g}_{J+1,k}(f_{J+1};p)\right\|_p^p
=\left\|\mathfrak{g}_{J+1,k}^{-1}(F_k;p)+f_{J+1}\right\|_p^p.
\end{align*}
By the orthogonality assumption on $\mathfrak{g}_{j,k}$, we
can easily prove that, for all $j\in\{1,\ldots, J\}$,
the function $\mathfrak{g}_{J+1,k}^{-1}\mathfrak{g}_{j,k}(f_j;p)$
converges to $0$ locally in measure as $k\to\infty$.
Hence, there is a subsequence of $\{\mathfrak{g}_{J+1,k}^{-1}(F_k;p)\}_{k\in\mathbb{N}}
=\{\sum_{j=1}^{J}\mathfrak{g}_{J+1,k}^{-1}(f_j;p)\}_{k\in\mathbb{N}}$,
denoted by $\{\mathfrak{g}_{J+1,k_\ell}^{-1}(F_{k_\ell};p)\}_{\ell\in\mathbb{N}}$,
such that $\mathfrak{g}_{J+1,{k_\ell}}^{-1}(F_{k_\ell};p)\to0$ almost everywhere as $\ell\to\infty$.
Since $\mathfrak{g}_{j,k}(\cdot;p)$ keeps the $L^p$ norm, from
the Brezis--Lieb lemma [see \eqref{eq-BL}], we deduce that, when $\ell\to\infty$,
\begin{align}\label{lem-decou-e1}
\left\|\mathfrak{g}_{J+1,k_\ell}^{-1}(F_{k_\ell};p)+f_{J+1}\right\|_p^p
=\|F_{k_\ell}\|_p^p+\|f_{J+1}\|_p^p+o(1).
\end{align}
Indeed, this also holds for the whole sequence. Otherwise,
there are a positive constant $\varepsilon_0$ and
a subsequence $\{\mathfrak{g}_{J+1,k_m}^{-1}(F_{k_m};p)\}_{m\in\mathbb{N}}$
such that, for all $m\in\mathbb{N}$,
\begin{align*}
\left|\left\|\mathfrak{g}_{J+1,k_m}^{-1}(F_{k_m};p)+f_{J+1}\right\|_p^p
-\|F_{k_m}\|_p^p-\|f_{J+1}\|_p^p\right|\ge\varepsilon_0.
\end{align*}
Repeating the argument used above, we can further choose
a subsequence satisfying \eqref{lem-decou-e1}. This is impossible.
Therefore, by the assumption that \eqref{lem-decou-e0} holds
for $J$, we find that, when $k\to\infty$,
\begin{align*}
\left\|\sum_{j=1}^{J+1}\mathfrak{g}_{j,k}(f_j;p)\right\|_p^p
&=\left\|\mathfrak{g}_{J+1,k}^{-1}(F_{k};p)+f_{J+1}\right\|_p^p
=\|F_{k}\|_p^p+\|f_{J+1}\|_p^p+o(1)\\
&=\sum_{j=1}^{J+1}\|f_j\|_p^p+o(1).
\end{align*}
This shows that \eqref{lem-decou-e0} also holds for
$J+1$, which then completes the proof of \eqref{lem-decou-e0}.
Furthermore, a similar induction argument, combined with
the nonlinear Brezis--Lieb lemma established in Lemma
\ref{lem-BL}\eqref{it1-BL}, also yields \eqref{lem-decou-ee}.
This then completes the present proof.
\end{proof}

As a consequence, we can easily show
the asymptotic decoupling for profile decompositions via Lebesgue norms.
Let $p\in(1,n)$ and $\{w_k\}_{k\in\mathbb{N}}$
be a bounded sequence in $\dot{W}^{1,p}$ such that $\lim_{k\to\infty}\|w_k\|_{p^*}=1$.
Then $\{w_k\}_{k\in\mathbb{N}}$ has the profile decomposition
(see Lemma \ref{lem-pdecom}). Moreover,
the sequence $\{V_j\}_{j\in\mathbb{N}}$ given in Lemma \ref{lem-pdecom}
satisfies the following asymptotic decoupling:
\begin{align}\label{eq-pdecom}
1=\lim_{k\to\infty}\|w_k\|_{p^*}^{p^*}
=\sum_{j=1}^{\infty}\|V_j\|_{p^*}^{p^*}.
\end{align}
Indeed, for fixed $J\in\mathbb{N}$ and any $k\in\mathbb{N}$,
by Proposition \ref{lem-pdecom}\eqref{it1-pdecom} and  the triangle inequality, we have
\begin{align*}
\|w_k\|_{p^*}\le \left\|\sum_{j=1}^J \mathfrak{g}_{j,k}(V_j;p^*)\right\|_{p^*}
+\|r_k^J\|_{p^*}.
\end{align*}
Hence, using \eqref{lem-pine-e0}, we find that, for any $\varepsilon\in(0,1)$,
there is a positive constant $C_\varepsilon$ such that
\begin{align*}
\|w_k\|_{p^*}^{p^*}\le (1+\varepsilon)\left\|\sum_{j=1}^J \mathfrak{g}_{j,k}(V_j;p^*)\right\|_{p^*}^{p^*}
+C_\varepsilon\|r_k^J\|_{p^*}^{p^*}.
\end{align*}
Letting $k\to\infty$ and applying \eqref{lem-decou-e0}, we obtain
\begin{align*}
1=\lim_{k\to\infty}\|w_k\|_{p^*}^{p^*}\le (1+\varepsilon)\sum_{j=1}^J\|V_j\|_{p^*}^{p^*}
+C_\varepsilon\limsup_{k\to\infty}\|r_k^J\|_{p^*}^{p^*}.
\end{align*}
This, combined with the arbitrariness of $J,\varepsilon$ and Lemma \ref{lem-pdecom}\eqref{it3-pdecom},
further implies that
$$1\le  \sum_{j=1}^\infty\|V_j\|_{p^*}^{p^*}.$$
We can then similarly conclude $1\ge  \sum_{j=1}^\infty\|V_j\|_{p^*}^{p^*}$,
and hence complete the proof of \eqref{eq-pdecom}.


\end{appendices}

\begin{appendices}

\section{The reverse Minkowski inequality}

The following inequality is probably known, but we give a proof for the sake of completeness. We follow \cite[Theorem 2.4]{ll01}.

\begin{lemma}\label{lem:minkowskireverse}
    Let $(\Omega,\mu)$ and $(\Gamma,\nu)$ be sigma-finite measure spaces,
    and let $f$ be a nonnegative, $\mu\times\nu$-measurable function on $\Omega\times\Gamma$. Then, for any $p\in(-\infty,1)\setminus\{0\}$,
    \begin{align}
        \label{eq:minkowskireverse}
    \left( \int_\Omega \left[ \int_\Gamma f(x,y)\,d\nu(y)\right]^p d\mu(x) \right)^{\frac1p}
    \geq \int_\Gamma \left[ \int_\Omega f(x,y)^p\,d\mu(x) \right]^{\frac1p}d\nu(y) \,.
    \end{align}
\end{lemma}

\begin{proof}
    All integrals involve nonnegative measurable functions, so they are well-defined, though possibly infinite. Set
    \[
    H(x) := \int_\Gamma f(x,y)\,d\nu(y)
    \qquad\text{for all}\ x\in\Omega \,.
    \]
    Then, using Fubini's theorem,
    \[
    \int_\Omega H(x)^p\,d\mu(x) = \int_\Gamma
    \left[ \int_\Omega H(x)^{p-1} f(x,y)\,d\mu(x) \right] d\nu(y) \,.
    \]
    By the reverse H\"older's inequality with $p'=p/(p-1)$, we have, for every $y\in\Gamma$,
    \[
    \int_\Omega H(x)^{p-1} f(x,y)\,d\mu(x) \geq \left[ \int_\Omega H(x)^p\,d\mu(x) \right]^{\frac{1}{p'}}
    \left[ \int_\Omega f(x,y)^p\,d\mu(x) \right]^{\frac1p}.
    \]
    Combining the previous two equations gives
    \begin{align}
        \label{eq:minkowskireverseproof}
            \int_\Omega H(x)^p\,d\mu(x) \geq \left[ \int_\Omega H(x)^p\,d\mu(x) \right]^{\frac{1}{p'}}
            \int_\Gamma \left[ \int_\Omega f(x,y)^p\,d\mu(x) \right]^{\frac1p}d\nu(y) \,.
    \end{align}
    If $0<\int_\Omega H(x)^p\,d\mu(x)<\infty$, then \eqref{eq:minkowskireverseproof} implies \eqref{eq:minkowskireverse}. Otherwise, one easily verifies that \eqref{eq:minkowskireverse} holds trivially.
    This completes the proof of the present lemma.
\end{proof}

\end{appendices}


\smallskip
\noindent\textbf{Acknowledgements}\quad
The authors acknowledge the use of AI tools during the exploratory stage of this project.
All mathematical arguments and proofs in the final manuscript were checked
and written by the authors.

\bigskip

\noindent Rupert L. Frank

\smallskip

\noindent Mathematisches Institut, Ludwig-Maximilians Universit\"at M\"unchen,
Theresienstr.~39, 80333 M\"u-nchen, Germany;
Munich Center for Quantum Science and Technology,
Schellingstr.~4, 80799 M\"unchen, Germany;
Mathematics 253-37, Caltech,
Pasadena, CA 91125, USA

\smallskip

\noindent {\it E-mail}: \texttt{r.frank@lmu.de}

\bigskip

\noindent Yinqin Li and Dachun Yang

\medskip

\noindent Laboratory of Mathematics and Complex Systems
(Ministry of Education of China),
School of Mathematical Sciences, Institute for Advanced Study,
Beijing Normal University,
Beijing 100875, The People's Republic of China

\smallskip

\noindent{\it E-mails:} \texttt{yinqli@mail.bnu.edu.cn} (Y. Li)

\noindent\phantom{\it E-mails:} \texttt{dcyang@bnu.edu.cn} (D. Yang)


\begin{thebibliography}{99}

\bibitem{a02}
C. O. Alves,
Existence of positive solutions for a problem with lack of compactness involving the $p$-Laplacian,
Nonlinear Anal. 51 (2002), 1187--1206.

\vspace{-.3cm}

\bibitem{acg26}
C. A. Antonini, G. Ciraolo and M. Gatti,
On the anisotropic critical $p$-Laplace equation: classification, decomposition,
and stability results, arXiv:2604.13758.

\vspace{-.3cm}

\bibitem{a76}
T. Aubin,
Probl\`emes isoperimetriques et espaces de Sobolev,
J. Differ. Geom. 11 (1976), 573--598.

\vspace{-.3cm}

\bibitem{be91}
G. Bianchi and H. Egnell,
A note on the Sobolev inequality,
J. Funct. Anal. 100 (1991), 18--24.

\vspace{-.3cm}

\bibitem{bfr26}
K. J. B\"or\"oczky, A. Figalli and J. P. G. Ramos,
Isoperimetric Inequalities, Brunn--Minkowski Theory and Minkowski-Type
Monge-Amp\`ere Equations on the Sphere, Zur. Lect. Adv. Math.,
European Mathematical Society (EMS), Berlin, 2026.

\vspace{-.3cm}

\bibitem{Bre02}
H. Brezis,
How to recognize constant functions. A connection with Sobolev spaces,
Russian Math. Surveys 57 (2002), 693--708.

\vspace{-0.3cm}

\bibitem{bl85}
H. Brezis and E. H. Lieb,
Sobolev inequalities with remainder terms,
J. Funct. Anal. 62 (1985), 73--86.

\vspace{-.3cm}

\bibitem{cfmp09}
A. Cianchi, N. Fusco, F. Maggi and A. Pratelli,
The sharp Sobolev inequality in quantitative form,
J. Eur. Math. Soc. (JEMS) 11 (2009), 1105--1139.

\vspace{-.3cm}

\bibitem{clyz09}
A. Cianchi, E. Lutwak, D. Yang and G. Zhang,
Affine Moser--Trudinger and Morrey--Sobolev inequalities,
Calc. Var. Partial Differential Equations 36 (2009), 419--436.

\vspace{-.3cm}

\bibitem{cfm18}
G. Ciraolo, A. Figalli and F. Maggi,
A quantitative analysis of metrics on $\mathbb R^n$ with almost constant positive scalar curvature,
with applications to fast diffusion flows,
Int. Math. Res. Not. IMRN 2018 (2018), 6780--6797.

\vspace{-.3cm}

\bibitem{cfr20}
G. Ciraolo, A. Figalli and A. Roncoroni,
Symmetry results for critical anisotropic $p$-Laplacian equations in convex cones,
Geom. Funct. Anal. 30  (2020), 770--803.

\vspace{-.3cm}

\bibitem{cg26}
G. Ciraolo and M. Gatti,
On the stability of the critical $p$-Laplace equation,
J. Funct. Anal. 291 (2026), Paper No. 111575, 58 pp.

\vspace{-.3cm}

\bibitem{cnv04}
D. Cordero-Erausquin, B. Nazaret and C. Villani,
A mass-transportation approach to sharp Sobolev and
Gagliardo--Nirenberg inequalities,
Adv. Math. 182 (2004), 307--332.

\vspace{-.3cm}

\bibitem{dx13}
F. Dai and Y. Xu,
Approximation Theory and Harmonic Analysis on Spheres and Balls,
Springer Monogr. Math, Springer, New York, 2013.

\vspace{-.3cm}

\bibitem{dw26}
J. Dai and T. Wang,
A positive solution to the $L^p$ projection centroid conjecture,
arXiv: 2605.20840.

\vspace{-.3cm}

\bibitem{dsw25}
B. Deng, L. Sun and J.-C. Wei,
Sharp quantitative estimates of Struwe's decomposition,
Duke Math. J. 174 (2025), 159--228.

\vspace{-.3cm}

\bibitem{dt25}
C. Dietze and P. T. Nam,
Minimizing sequences of Sobolev inequalities revisited,
Int. Math. Res. Not. 2025 (2025), Paper No. rnaf190, 11 pp.

\vspace{-.3cm}

\bibitem{deffl25}
J. Dolbeault, M. J. Esteban, A. Figalli, R. L. Frank and M. Loss, Sharp stability for Sobolev and log-Sobolev inequalities, with optimal dimensional dependence,
Camb. J. Math. 13 (2025), 359--430.

\vspace{-.3cm}

\bibitem{DLTYY25}
O. Dom\'inguez, Y. Li, S. Tikhonov, D. Yang and W. Yuan,
New approach to
affine Moser--Trudinger inequalities via Besov polar projection bodies,
Math. Ann. 392 (2025), 3319--3366.

\vspace{-.3cm}

\bibitem{ens22}
M. Engelstein, R. Neumayer and L. Spolaor,
Quantitative stability for minimizing Yamabe metrics,
Trans. Amer. Math. Soc. Ser. B 9 (2022), 395--414.

\vspace{-.3cm}

\bibitem{flz26}
S. Fan, G.-D. Li and J. J. Zhang,
Sharp stability for the affine fractional Sobolev inequality,
arXiv:2605.03732.

\vspace{-.3cm}

\bibitem{flz26b}
S. Fan, G.-D. Li and J. J. Zhang,
Sharp Quantitative Stability for the Affine \(p\)-Sobolev Inequality, Part I: The Case \(2\le p<n\),
arXiv:2606.09555

\vspace{-.3cm}

\bibitem{f69}
H. Federer, Geometric Measure Theory,
Die Grundlehren der mathematischen Wissen-schaften 153,
Springer-Verlag New York, Inc., New York, 1969.

\vspace{-.3cm}

\bibitem{fg20}
A. Figalli and F. Glaudo,
On the sharp stability of critical points of the Sobolev inequality,
Arch. Ration. Mech. Anal. 237 (2020), 201--258.

\vspace{-.3cm}

\bibitem{fn19}
A. Figalli and R. Neumayer,
Gradient stability for the Sobolev inequality: the case $p\ge2$,
J. Eur. Math. Soc. (JEMS) 21  (2019), 319--354.

\vspace{-.3cm}

\bibitem{fz22}
A. Figalli and Y. R.-Y. Zhang,
Sharp gradient stability for the Sobolev inequality,
Duke Math. J. 171 (2022), 2407--2459.

\vspace{-.3cm}

\bibitem{f22}
R. L. Frank,
Degenerate stability of some Sobolev inequalities,
Ann. Inst. H. Poincar\'e Anal. Non Lin\'eaire 39 (2022), 1459--1484.

\vspace{-.3cm}

\bibitem{f24}
R. L. Frank,
The sharp Sobolev inequality and its stability: An introduction. In: Geometric and Analytic Aspects
of Functional Variational Principles. Cetraro, Italy 2022 (A. Cianchi, V. Maz’ya, and T. Weth, eds.), pp. 1--64, Springer, Cham, 2024.

\vspace{-.3cm}

\bibitem{fl22}
R. L. Frank and M. Loss,
Existence of optimizers in a Sobolev inequality for vector fields,
Ars Inven. Anal. (2022), Paper No. 1, 31 pp.

\vspace{-.3cm}

\bibitem{fp24a}
R. L. Frank and J. W. Peteranderl,
Degenerate stability of the Caffarelli--Kohn--Nirenberg inequality along the Felli--Schneider curve,
Calc. Var. Partial Differential Equations 63 (2024), Paper No. 44, 33 pp.

\vspace{-.3cm}

\bibitem{fp24}
R. L. Frank and J. W. Peteranderl,
The sharp $\sigma_2$-curvature inequality on the sphere in quantitative form, arXiv:2412.12819.

\vspace{-.3cm}

\bibitem{fpr25}
R. L. Frank, J. W. Peteranderl and L. Read,
Sharp quantitative integral inequalities for harmonic extensions,
arxiv:2508.09940.

\vspace{-.3cm}

\bibitem{g02}
R. J. Gardner,
The Brunn--Minkowski inequality,
Bull. Amer. Math. Soc. (N.S.) 39 (2002), 355--405.

\vspace{-.3cm}

\bibitem{hs09}
C. Haberl and F. E. Schuster,
Asymmetric affine $L_p$ Sobolev inequalities,
J. Funct. Anal. 257 (2009), 641--658.

\vspace{-.3cm}

\bibitem{glz25}
A. Guerra, X. Lamy and K. Zemas,
Sharp quantitative stability of the M\"obius group among sphere-valued maps in arbitrary dimension,
Trans. Amer. Math. Soc. 378 (2025), 1235--1259.

\vspace{-.3cm}

\bibitem{hjm16}
J. Haddad, C. H. Jim\'enez and M. Montenegro,
Sharp affine Sobolev type inequalities via the $L_p$
Busemann--Petty centroid inequality,
J. Funct. Anal. 271 (2016), 454--473.

\vspace{-.3cm}

\bibitem{hjm21}
J. Haddad, C. H. Jim\'enez and M. Montenegro,
From affine Poincar\'e inequalities to affine spectral inequalities,
Adv. Math. 386 (2021), Paper No. 107808, 35 pp.

\vspace{-.3cm}

\bibitem{hl24}
J. Haddad and M. Ludwig,
Affine fractional $L^p$ Sobolev inequalities,
Math. Ann. 388 (2024), 1091--1115.

\vspace{-.3cm}

\bibitem{hl24-2}
J. Haddad and M. Ludwig,
Affine fractional Sobolev and isoperimetric inequalities,
J. Differential Geom. 129 (2025), 695--724.

\vspace{-.3cm}

\bibitem{hl16}
Q. Huang and A.-J. Li,
Optimal Sobolev norms in the affine class,
J. Math. Anal. Appl. 436 (2016), 568--585.

\vspace{-.3cm}

\bibitem{iz25}
L.~I.~Ignat and E.~Zuazua,
Optimal convergence rates for the finite element approximation of the Sobolev constant,
arXiv:2504.09637.

\vspace{-.3cm}

\bibitem{j99}
S. Jaffard,
Analysis of the lack of compactness in the critical Sobolev embeddings,
J. Funct. Anal. 161 (1999), 384--396.

\vspace{-.3cm}

\bibitem{ll01}
E. H. Lieb and M. Loss,
Analysis, 2nd ed, Grad. Stud. Math. 14,
American Mathematical Society (AMS), Providence, RI, 2001.

\vspace{-.3cm}

\bibitem{l85}
P.-L. Lions,
The concentration-compactness principle in the calculus of variations. The limit case. I,
Rev. Mat. Iberoamericana 1 (1985), no. 1, 145--201.

\vspace{-.3cm}

\bibitem{lz25}
G. Liu and Y. R.-Y. Zhang,
Sharp stability for critical points of the Sobolev inequality
in the absence of bubbling,
arXiv:2503.02340.

\vspace{-.3cm}

\bibitem{lyz00}
E. Lutwak, D. Yang and G. Zhang,
$L^p$ affine isoperimetric inequalities,
J. Differ. Geom. 56 (2000), 111--132 .

\vspace{-.3cm}

\bibitem{lyz02}
E. Lutwak, D. Yang and G. Zhang,
Sharp affine $L_p$ Sobolev inequalities,
J. Differ. Geom. 62 (2002), 17--38.

\vspace{-.3cm}

\bibitem{mw10}
C. Mercuri and M. Willem,
A global compactness result for the $p$-Laplacian involving critical nonlinearities,
Discrete Contin. Dyn. Syst. 28 (2010), 469--493.

\vspace{-.3cm}

\bibitem{msy25}
E. Milman, S. Shabelman and A. Yehudayoff,
Fixed and periodic points of the intersection body operator,
Invent. Math. 241 (2025), 509--558.

\vspace{-.3cm}

\bibitem{n20}
R. Neumayer,
A note on strong-form stability for the Sobolev inequality,
Calc. Var. Partial Differ. Equ. 59 (2020), Paper No. 25, 8 pp.

\vspace{-.3cm}

\bibitem{n16}
V. H. Nguyen,
New approach to the affine P\'olya--Szeg\"o principle and
the stability version of the affine Sobolev inequality,
Adv. Math. 302 (2016), 1080--1110.

\vspace{-.3cm}

\bibitem{ps07}
P. Pucci and J. Serrin, The Maximum Principle,
Prog. Nonlinear Differ. Equ. Appl. 73, Birkh\"auser Verlag, Basel, 2007.

\vspace{-.3cm}

\bibitem{r26}
C. Reuter,
Local fixed point results for centroid body operators,
Trans. Am. Math. Soc. 379 (2026), 3043--3061.

\vspace{-.3cm}

\bibitem{r66}
E. Rodemich,
The Sobolev inequality with best possible constant.
Analysis Seminar Caltech, Spring 1966.

\vspace{-.3cm}

\bibitem{st18}
I. Schindler and C. Tintarev,
Compactness properties and ground states for the affine Laplacian,
Calc. Var. Partial Differential Equations 57 (2018), Paper No. 48,
14 pp.

\vspace{-.3cm}

\bibitem{s14}
R. Schneider,
Convex Bodies: The Brunn--Minkowski Theory,
2nd expanded ed.,
Encycl. Math. Appl. 151,
Cambridge University Press, Cambridge, 2014.

\vspace{-.3cm}

\bibitem{s84}
M. Struwe,
A global compactness result for elliptic boundary value problems involving limiting nonlinearities,
Math. Z. 187 (1984), 511--517.

\vspace{-.3cm}

\bibitem{s75}
G. Szeg\"o,
Orthogonal Polynomials, 4th ed,
Colloq. Publ., Am. Math. Soc. 23,
American Mathematical Society (AMS), Providence, RI, 1975.

\vspace{-.3cm}

\bibitem{t76}
G. Talenti,
Best constant in Sobolev inequality,
Ann. Mat. Pura Appl. (4) 110 (1976), 353--372.

\vspace{-.3cm}

\bibitem{t96}
A. C. Thompson, Minkowski Geometry, Cambridge Univ. Press, Cambridge, 1996.

\vspace{-.3cm}

\bibitem{wz25}
G. Wang and M. Zhang,
Stability of spinorial Sobolev inequalities on $\mathbb S^n$,
arXiv:2508.09047.

\vspace{-.3cm}

\bibitem{w12}
T. Wang,
The affine Sobolev--Zhang inequality on $BV(\mathbb R^n)$,
Adv. Math. 230 (2012), 2457--2473.

\vspace{-.3cm}

\bibitem{w13}
T. Wang,
The affine P\'olya--Szeg\"o principle: equality cases and stability,
J. Funct. Anal. 265 (2013), 1728--1748.

\vspace{-.3cm}

\bibitem{ww26}
J. Wei and Y. Wu,
Stability of the Caffarelli--Kohn--Nirenberg inequality along the Felli--Schneider curve: critical points at infinity.
Proc. Lond. Math. Soc. (3) 132 (2026), Paper No. e70146, 102 pp.

\vspace{-.3cm}

\bibitem{z99}
G. Zhang,
The affine Sobolev inequality,
J. Differ. Geom. 53 (1999), 183--202.

\vspace{-.3cm}

\bibitem{zz26}
Y. Zhou and W. Zou,
Degenerate stability of critical points of the Caffarelli--Kohn--Nirenberg inequality along the Felli--Schneider curve,
J. Lond. Math. Soc. (2) 113 (2026), Paper No. e70454, 37 pp.

\end{thebibliography}
\end{document}